\documentclass{amsart}

\usepackage{comment}
\usepackage{graphicx,xcolor}

\numberwithin{equation}{section}
\usepackage[initials,nobysame]{amsrefs}
\usepackage{amsmath,mathtools}
\usepackage{hyperref,cleveref}
\usepackage{enumitem}

\setlist[enumerate,1]{label=\upshape{(\roman*)},ref=\roman*}

\usepackage{autonum}
\newcommand{\FR}[1]{{\color{green!60!black}\ #1}}
\newcommand{\TM}[1]{{\color{purple}\ #1}}

\newtheorem{theorem}{Theorem}[section]
\newtheorem{proposition}[theorem]{Proposition}
\newtheorem{corollary}[theorem]{Corollary}
\newtheorem{lemma}[theorem]{Lemma}
\newtheorem{conjecture}[theorem]{Conjecture}
\theoremstyle{definition}

\newtheorem*{acknowledgements}{Acknowledgements}
\theoremstyle{remark}
\newtheorem{remark}[theorem]{Remark}
\newtheorem{example}[theorem]{Example}

\DeclareMathOperator{\sech}{sech}

\newcommand{\R}{\mathbf{R}}
\newcommand{\N}{\mathbf{N}}
\newcommand{\Z}{\mathbf{Z}}

\DeclareMathOperator{\Span}{\textup{span}} 

\title[Shortening and straightening complete curves]{A new energy method for shortening and straightening complete curves} 

\author[T.~Miura]{Tatsuya Miura}
\address[T.~Miura]{Department of Mathematics, Graduate School of Science, Kyoto University, Kitashirakawa Oiwake-cho, Sakyo-ku, Kyoto 606-8502, Japan}
\email{tatsuya.miura@math.kyoto-u.ac.jp}

\author[F.~Rupp]{Fabian Rupp}
\address[F.~Rupp]{Faculty of Mathematics, University of Vienna, Oskar-Morgenstern-Platz 1, 1090 Vienna, Austria.}
\email{fabian.rupp@univie.ac.at}

\date{\today}

\begin{document}

\begin{abstract}
We introduce a novel energy method that reinterprets ``curve shortening'' as ``tangent aligning''. This conceptual shift enables the variational study of infinite-length curves evolving by the curve shortening flow, as well as higher order flows such as the elastic flow, which involves not only the curve shortening but also the curve straightening effect. For the curve shortening flow, we prove convergence to a straight line under mild assumptions on the ends of the initial curve. For the elastic flow, we establish a global well-posedness theory, and investigate the precise long-time behavior of solutions. In fact, our method applies to a more general class of geometric evolution equations including the surface diffusion flow, Chen's flow, and the free elastic flow.
\end{abstract}

\maketitle
\setcounter{tocdepth}{1}
% \tableofcontents

\section{Introduction}

The \emph{curve shortening flow} is one of the most classical geometric flows (see, e.g., \cite{Chou_Zhu_2001_book,Andrews_etal_2020_book}), defined by a one-parameter family of immersed curves $\gamma=\gamma(t,x):[0,T)\times I\to\R^n$, for $n\geq2$ and an interval $I\subset\R$, satisfying the evolution equation
\begin{equation}\label{eq:CSF}
    \partial_t\gamma = \kappa, \tag{CSF}
\end{equation}
where $\kappa:=\partial_s^2\gamma$ denotes the curvature vector and $\partial_s:=|\partial_x\gamma|^{-1}\partial_x$ is the arclength derivative along $\gamma$.
The flow is referred to as ``curve shortening'' because it arises as the $L^2(ds)$-gradient flow of the length functional 
\[
L[\gamma]:=\int_\gamma ds,
\]
where $ds \vcentcolon = |\partial_x\gamma|dx$.
In particular, if the initial curve has finite length, then $L$ monotonically decreases along the flow, and this fact is essential in energy methods for asymptotic analysis. However, such methods clearly break down for curves of infinite length.

Identifying gradient flow structures is a fundamental technique in the analysis of parabolic PDEs. Such an underlying variational structure is however not unique---the same equation can be a gradient flow for different energies, possibly with respect to different metrics, as observed, for instance, in the seminal work of Jordan--Kinderlehrer--Otto \cite{MR1617171}. In this spirit, this work
develops a conceptually new energy method for the curve shortening flow, based on the direction energy, see \eqref{eq:direction_energy} below. 
Our approach opens the possibility to apply energy methods to study the long-time behavior of curve-shortening-type flows for non-compact complete curves, avoiding the use of maximum principles, and thus provides a reliable and robust tool also in higher codimensions.

We first apply this method to the curve shortening flow.
While the well-estab\-lished theory based on maximum principles, monotonicity properties, and the classification of solitons provides some global existence results even for curves of infinite length, the asymptotic analysis remains challenging.
As a direct application of our energy method, we obtain a new result on the convergence of the curve shortening flow to a straight line.
Notably, this result applies to initial curves that are neither necessarily planar nor confined to a slab, which are typically difficult to handle using maximum principles.

More importantly, our method is directly applicable to higher-order flows, due to its independence from maximum principles.
For higher-order flows, various energy methods have been well developed in the compact case, but the non-compact case remains significantly less explored due to the absence of a canonical finite energy.

Here we apply our direction energy method to a class of fourth-order flows of curve shortening type, including the classical \emph{surface diffusion flow} \cite{Mullins_1957},
\begin{equation}\label{eq:SDF}
    \partial_t \gamma = -\nabla_s^2 \kappa, \tag{SDF}
\end{equation}
as well as \emph{Chen's flow} introduced in \cite{Bernard-Wheeler-Wheeler_2019_Chen},
\begin{equation}\label{eq:CF}
    \partial_t \gamma = -\partial_s^4 \gamma = -\nabla_s^2 \kappa + |\kappa|^2 \kappa + 3\langle \kappa, \partial_s \kappa \rangle \partial_s \gamma, \tag{CF}
\end{equation}
to obtain results analogous to those for the curve shortening flow.
Here $\nabla_s$ denotes the normal derivative along $\gamma$ defined by $\nabla_sf:=(\partial_sf)^\perp=\partial_sf-\langle\partial_sf,\partial_s\gamma\rangle\partial_s\gamma$. 

Besides the direction energy method, we also develop a new technical tool for controlling higher-order derivatives, extending the interpolation method of Dziuk--Kuwert--Sch\"atzle \cite{Dziuk-Kuwert-Schatzle_2002} from the compact to the non-compact setting.
This tool is not only necessary for the analysis of curve-shortening-type flows but is particularly well-suited for analyzing curve-straightening-type flows, which decrease quantities involving the bending energy
\begin{equation}
    B[\gamma] := \frac{1}{2} \int_\gamma |\kappa|^2 ds.
\end{equation}

In this paper, we also address the \( L^2(ds) \)-gradient flow of the elastic energy \( B + \lambda L \) for \( \lambda \geq 0 \), the so-called \emph{\( \lambda \)-elastic flow},
\begin{equation}\label{eq:lambda-EF}
    \partial_t \gamma = -\nabla_s^2 \kappa - \frac{1}{2} |\kappa|^2 \kappa + \lambda \kappa, \tag{$\lambda$-EF}
\end{equation}
which was introduced in \cite{Polden1996,Dziuk-Kuwert-Schatzle_2002} and has since been extensively studied in the compact case (see also the survey \cite{Mantegazza_Pluda_Pozzetta_21_survey}).
We establish a foundational global existence theory for this flow in the non-compact case and then investigate the precise long-time behavior of solutions.
In the purely straightening case $\lambda=0$, we show that the flow always converges to a line.
On the other hand, when \( \lambda > 0 \), the interaction between the curve shortening and straightening effects leads to much more complex asymptotic behavior.

Our results also yield new rigidity theorems for solitons as direct corollaries.

This paper is organized as follows:
In \Cref{sec:main_results}, we provide a more detailed explanation of our main results described above, while comparing them with previous works.
\Cref{sec:preliminaries} presents preliminaries, particularly basic properties of the direction energy.
\Cref{sec:energy_decay} provides the key interpretation of the flows as gradient flows involving the direction energy.
In \Cref{sec:interpolation} we prove the crucial weighted interpolation estimate, and then in \Cref{sec:curvature_control} apply it to deduce global curvature bounds of any order along the flows.
Using these estimates, in \Cref{sec:long-time_behavior}, we obtain the main blow-up and convergence results for each flow.
Finally, \Cref{sec:soliton} discusses rigidity results for solitons. 
The paper is complemented by an appendix which includes a discussion of local well-posedness in the non-compact case (\Cref{sec:local_wellposedness}).

\begin{acknowledgements}
    The authors would like to thank Ben Andrews for discussions about the curve shortening flow.
    The first author is supported by JSPS KAKENHI Grant Numbers JP21H00990, JP23H00085, and JP24K00532.
    The second author is funded in whole, or in part, by the Austrian Science Fund (FWF), grant number \href{https://doi.org/10.55776/ESP557}{10.55776/ESP557}. 
    Part of this work was done when the second author was visiting Kyoto University supported by a Mobility Fellowship of the Strategic Partnership Program between the University of Vienna and Kyoto University.
\end{acknowledgements}

\section{Main results}\label{sec:main_results}

\subsection{Direction energy and interpolation method}

For a curve $\gamma$ immersed in $\R^n$ we define the \emph{direction energy} (with respect to the direction $e_1$) as
\begin{equation}\label{eq:direction_energy}
    D[\gamma]:=\frac{1}{2}\int_\gamma |\partial_s\gamma(s)-e_1|^2ds,
\end{equation}
where, here and in the sequel, $\{e_j\}_{j=1}^n$ denotes the canonical basis of $\R^n$.
Of course, the preferred tangent direction $e_1$ can be replaced by any other unit vector, after suitable rotation.
This energy can be equivalently represented by
\begin{equation}
    D[\gamma]=\int_\gamma \big(1-\langle\partial_s\gamma(s),e_1\rangle\big)ds,
\end{equation}
and hence, if $L[\gamma]<\infty$, is given by
\begin{equation}
    D[\gamma] = L[\gamma] - \int_\gamma \partial_s\langle\gamma(s),e_1\rangle ds.
\end{equation}
In particular, if $\gamma$ is closed, then we simply have $D = L$ as the last integral vanishes by integration by parts.
In contrast, for non-closed curves, the two functionals may differ.

However, an important observation is that the integrand $\langle \partial_s\gamma(s), e_1 \rangle$ is a \emph{null Lagrangian}.
As a result, the gradient flow of the direction energy still produces the curve shortening flow.
This perspective offers the new geometrical interpretation:
\begin{center}
    ``curve shortening $\Longleftrightarrow$ tangent aligning''.
\end{center}

Besides its conceptual interest, this idea is particularly useful in the case of non-compact complete curves, since $L=\infty$ whereas $D$ can remain finite.

\begin{remark}[Admissible ends]\label{rem:examples_finite_direction}
    The class of curves with finite direction energy comprises complete curves whose ends suitably converge to two parallel semi-lines with opposite directions. Not only that, this class also allows for slowly diverging ends with or without oscillation (see \Cref{fig:ends}), such as the graphs of 
    $u(x) = x^\alpha$ or $u(x)=x^\alpha\sin\log{x}$ as $x\to\infty$, where $\alpha\in(0,\frac{1}{2})$, as well as spatial spiraling-out ends like $u(x)=(x^\alpha\sin\log{x},x^\alpha\cos\log{x})$.
    In fact, the finiteness of $D$ is equivalent to having graphical ends with suitably integrable gradients (see \Cref{lem:direction_graphical}).
\end{remark}

\begin{figure}[htbp]
    \centering
    \includegraphics[width=0.7\linewidth]{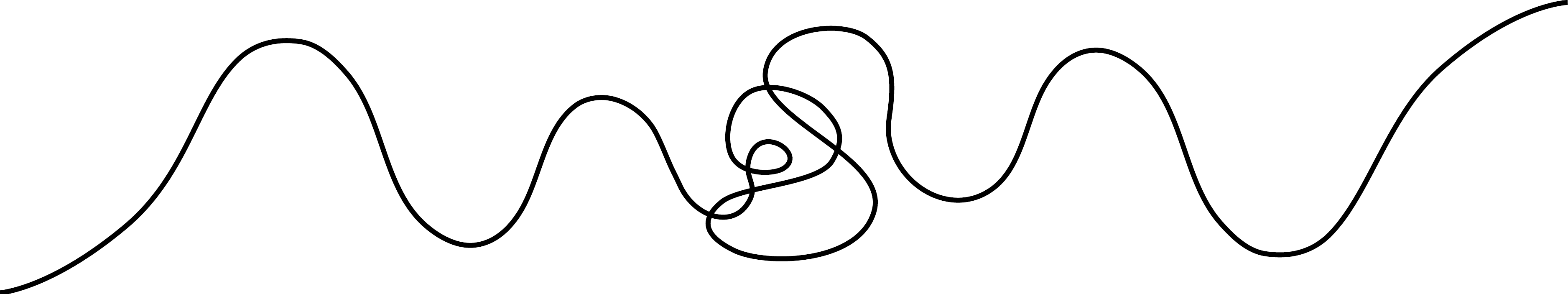}
    \caption{An immersed planar curve with admissible ends.}
    \label{fig:ends}
\end{figure}

The key idea of the present work is to replace \( L \) with \( D \) to produce finite monotone quantities even for non-compact curves.
This idea is inspired by previous works by the first author \cite{Miura20} and in collaboration with Wheeler \cite{miura2024uniqueness} on the analysis of elastica, a stationary problem involving \( B \) and \( L \).
In those papers, the direction energy is used to variationally handle the so-called \emph{borderline elastica} (see \Cref{fig:borderline}), which has infinite length and horizontal ends. Our work here is the first to apply this idea to gradient flows. 

\begin{figure}[htbp]
    \centering
    \includegraphics[width=0.6\linewidth]{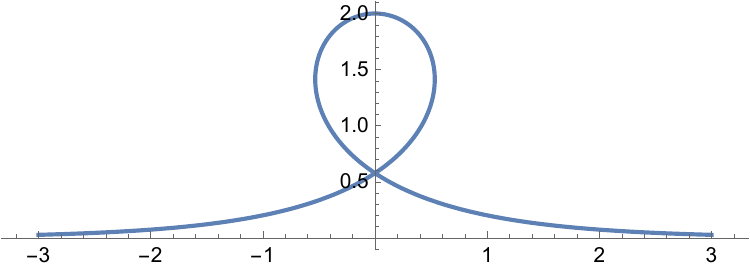}
    \caption{The borderline elastica, parametrized as in \eqref{eq:borderline}.}
    \label{fig:borderline}
\end{figure}

After establishing the energy decay, we also need to handle higher-order quantities through interpolation estimates.
In \cite{Dziuk-Kuwert-Schatzle_2002}, Dziuk--Kuwert--Sch\"atzle established the key Gagliardo--Nirenberg-type interpolation inequalities to apply energy-type methods to evolutions of closed curves.
%This has since been applied to a variety of flows and boundary conditions in the compact case, (see e.g.\ \cite{DallAcqua-Pozzi_2014_Willmore-Helfrich,Wheeler_13,Parkins_Wheeler_16,Cooper_Wheeler_Wheeler_23,miura2024asymptotic,MR3726834,MR4048466,MR2911840,MR3351503,MR4153655,dallacqua2025dimensionreductionwillmoreflows}),
This has since been applied to a variety of flows in the compact case (even of more than fourth-order, e.g., \cite{Parkins_Wheeler_16,MR4153655,langer2024dynamicselasticwirespreserving}),
but does not directly extend to the non-compact case, since the inequalities involve the total length---which is infinite here.
In this paper, we develop, for the first time, a variant of this interpolation method which is suitable to treat the non-compact situation.
To that end, inspired by the localization procedure in \cite{MR1900754}, we first derive an interpolation estimate with a power-type weight function (\Cref{prop:interpolation_P}).
This approach is necessary, since we do not assume any a priori integrability and it is of particular importance to keep track of the relation between the exponents and the order of derivatives. 
We then establish higher-order estimates under localization, which are finally globalized in both time and space by a time cut-off argument in the spirit of Kuwert--Sch\"atzle's work \cite{KSSI} (\Cref{sec:curvature_control}).

In what follows we discuss concrete geometric flows.
Note that we only consider \emph{classical solutions} in this work. 
However, using parabolic smoothing, our results easily transfer to appropriate weak solutions, e.g., in suitable Sobolev classes.

\subsection{Curve shortening flow}

%After Gage \cite{Gage_1983_isoperimetric,Gage_1984_circular}, 
The celebrated works of Gage--Hamilton \cite{Gage_Hamilton_1986_shrinking} and Grayson \cite{Grayson_1987_round} proved that the curve shortening flow of closed planar embedded curves must shrink to a round point in finite time. %(see also \cite{Huisken_1998_distance,Andrews_Bryan_2011_comparison}).

The case of non-compact complete curves is more delicate.
After Polden \cite{Polden_1991_evolving_curve} and Huisken \cite{Huisken_1998_distance}, Chou--Zhu \cite{Chou_Zhu_1998_complete} established global-in-time existence for a wide class of complete embedded planar initial curves.
Polden \cite{Polden_1991_evolving_curve} and Huisken \cite[Theorem 2.5]{Huisken_1998_distance} also proved convergence to self-similar solutions, assuming that the ends of the initial curve are asymptotic to semi-lines.
The limit curve depends on the angle formed by the asymptotic semi-lines; in particular, if the angle is \( \pi \), then the solution locally smoothly converges to a straight line parallel to the asymptotic semi-lines.
Their proof relies on the fact that the solution remains inside an initial slab
(see also \Cref{rem:convergence_to_line_comparison}).
Recently, Choi--Choi--Daskalopoulos \cite{Choi_Choi_Daskalopoulos_2021_translating} obtained a refined convergence to the grim reaper under convexity.
See also \cite{Nara-Taniguchi_2006_DCDS_stability,Nara_Taniguchi_2007_convergence_line,Wang-Wo_2011_stability_line_grim,Wang-Wo_2013_stability_line} for the planar stability of lines.

The above works strongly rely on several types of maximum principles, thus are based on planarity, since, in higher codimensions, some important maximum principles are not available; for example, embeddedness may be lost \cite{MR1131441}.
This makes the analysis much more complicated, and indeed the authors are not aware of any relevant convergence results for infinite-length curves if $n\geq3$.

Our main result here extends Polden and Huisken's result on convergence to the straight line, not only to more general ends (possibly not contained in any slab) but also to general codimensions by a purely energy-based approach.
We first formulate our result in the form of the following general dichotomy theorem.
Hereafter, we use the notation $\dot{C}^\infty$ for the space of smooth functions with bounded derivatives, as defined in \Cref{subsec:notation}; this regularity assumption on the initial datum is required in our proof of local well-posedness, see \Cref{sec:local_wellposedness} for a detailed discussion.

\begin{theorem}\label{thm:main_CSF_dichotomy}
    Let $\gamma:[0,T)\times\R\to\R^n$ be a curve shortening flow \eqref{eq:CSF} with initial datum $\gamma_0\in \dot{C}^\infty(\R;\R^n)$ such that $\inf_{\R}|\partial_x\gamma_0|>0$ and maximal existence time $T\in(0,\infty]$.
    Suppose that $D[\gamma_0]<\infty$.
    Then the direction energy $D[\gamma(t,\cdot)]$ continuously decreases in $t\geq0$.
    In addition, the following dichotomy holds:
    \begin{enumerate}
        \item\label{item:main_CSF_blow-up} If $T<\infty$, then
        \begin{equation}
            \liminf_{t\to T} (T-t)^{1/2}\int_{\gamma(t,\cdot)}|\kappa|^2ds>0.
        \end{equation}
        \item\label{item:main_CSF_convergence}
        If $T=\infty$, then
        \begin{equation}
            \lim_{t\to\infty}\int_{\gamma(t,\cdot)}|\kappa|^2ds =0,
        \end{equation} 
        and moreover,
        \begin{equation}\label{eq:0220-1}
            \lim_{t\to\infty}\sup_{\R}|\partial_s\gamma(t,\cdot)-e_1|=0, \quad  \lim_{t\to\infty}\sup_{\R}|\partial_s^m\kappa(t,\cdot)|=0,
        \end{equation}
        for all $m\in\N_0$.
        In particular, after reparametrization by arclength and translation, the solution locally smoothly converges to a horizontal line.
    \end{enumerate}
\end{theorem}

\begin{remark}
    Here the translation is with respect to any base points, see
    \eqref{eq:main_CSF_convergence} for the precise statement of convergence.
    % , i.e., for any choice of $p_t\in\gamma(t,\R)$ the translated image $\gamma(t,\R)-p_t$ converges to the $e_1$-axis.
    The convergence of derivatives is global in space.
    In particular, this directly ensures that every global-in-time solution is eventually graphical; this is in stark contrast to the elastic flow discussed below.
    Even for the curve shortening flow, the curve itself may not converge uniformly (see \Cref{ex:escaping_CSF}).
    Also, translation is essentially needed, because the solution curve can keep oscillating or even escape to infinity (see \Cref{ex:oscillating_CSF,ex:escaping_CSF}).
\end{remark}

\begin{remark}
    The $L^\infty$ blowup rate $\liminf_{t\to T}(T-t)^{1/2}\|\kappa(t,\cdot)\|_\infty>0$ is well known at least in the compact case and easily extended to the non-compact case, since the proof is based on pointwise estimates.
    Integral-type estimates are also known in the compact case; in particular, Angenent \cite{MR1078266} obtained the sharp $L^p$ blowup rate $\liminf_{t\to T}(T-t)^{(p-1)/2}\int_{\gamma(t,\cdot)}|\kappa|^pds>0$ for all $p\in(1,\infty)$ in the case of closed planar curves.
    Our result seems the first integral-type blowup estimate in the non-compact case.
\end{remark}

Now we recall previous work providing sufficient conditions for global existence.
There are two typical cases: planar embedded curves, and (possibly non-planar) graphical curves.
We say that an immersed curve $\gamma:\R\to\R^n$ is \emph{graphical} if     
\begin{equation}
    \inf_{\R}\langle \partial_s\gamma,e_1 \rangle >0.
\end{equation}
This together with $\inf_\R|\partial_x\gamma|>0$ implies that $\gamma$ may be represented by the graph of a function $u:\R\to\R^{n-1}$ over the $e_1$-axis.

\begin{corollary}\label{cor:CSF_convergence}
    Let $\gamma_0\in\dot{C}^\infty(\R;\R^n)$ with $\inf_{\R}|\partial_x\gamma_0|>0$ and $D[\gamma_0]<\infty$.
    Suppose that $\gamma_0$ is planar and embedded (resp.\ $\gamma_0$ is graphical).
    Then the unique curve shortening flow \eqref{eq:CSF} starting from $\gamma_0$ remains planar and embedded (resp.\ graphical), exists globally in time, and thus converges as in Theorem \ref{thm:main_CSF_dichotomy} \eqref{item:main_CSF_convergence}.
\end{corollary}

\begin{proof}
    The planar embedded case directly follows by Huisken's distance comparison argument for \cite[Theorem 2.5]{Huisken_1998_distance} since $D[\gamma_0]<\infty$ implies that $\gamma_0$ has appropriate graphical ends (\Cref{lem:direction_graphical}) so that the extrinsic/intrinsic distance ratio is bounded away from zero, also near infinity.
    The graphical case follows since the arguments of Altschuler--Grayson \cite[Theorem 1.13, Theorem 2.6 (1), (2)]{Altschuler_Grayson_1992_spacecurve} for periodic graphical curves in $\R^3$ directly extend to non-periodic graphical curves in $\R^n$ (see also \cite{Hattenschweiler_2015_CSF_higherdim,MR2156947}), where we need a maximum principle on the whole line as in \cite[Proposition 52.4]{Quittner_Souplet_2019_book}.
\end{proof}

We can also give a new, simple blow-up criterion in the planar case.

\begin{corollary}\label{cor:CSF_blowup}
    Let $\gamma_0\in\dot{C}^\infty(\R;\R^2)$ with $\inf_{\R}|\partial_x\gamma_0|>0$ and $D[\gamma_0]<\infty$.
    Suppose that $\gamma_0$ has nonzero rotation number.
    Then the unique curve shortening flow \eqref{eq:CSF} starting from $\gamma_0$ blows up in finite time as in \Cref{thm:main_CSF_dichotomy} \eqref{item:main_CSF_blow-up}.
\end{corollary}

\begin{proof}
    Since the rotation number is preserved (see \Cref{lem:rotation_number_preservation}), the flow keeps having a leftward tangent somewhere.
    If the flow exists globally, this contradicts the uniform convergence of the tangent vector in \Cref{thm:main_CSF_dichotomy} \eqref{item:main_CSF_convergence}.
\end{proof}

\subsection{Fourth-order flows of curve shortening type}

As mentioned in the introduction, our argument also works even for a class of higher-order flows including the surface diffusion flow \eqref{eq:SDF} and Chen's flow \eqref{eq:CF}.
These flows also decrease the length in the compact case, and some energy methods are already developed in previous work, see, e.g., \cite{Elliott_Garcke_1997_SDF,Dziuk-Kuwert-Schatzle_2002,Chou_2003_SDF,Wheeler_13,Miura_Okabe_2021_SDF} for the surface diffusion flow and \cite{Cooper_Wheeler_Wheeler_23} for Chen's flow.

For these flows, we obtain quite analogous results to \eqref{eq:CSF}, see \Cref{thm:main_SDF_Chen} for dichotomy and \Cref{cor:SDF_Chen_blowup} for blow-up.
To the authors' knowledge, the latter is the first rigorous blow-up result for the non-compact surface diffusion as well as Chen's flow.
Although no results corresponding to \Cref{cor:CSF_convergence} are known due to the lack of maximum principles, global existence holds at least for suitable small initial data (cf.\ \cite{Koch-Lamm_2012}).

In addition, the above results directly imply new classification results for solitons to \eqref{eq:CSF}, \eqref{eq:SDF}, and \eqref{eq:CF} in a unified manner.
For \eqref{eq:CSF}, comprehensive classification results for solitons have been obtained in $\R^2$ \cite{Halldorsson_2012_CSFsoliton} as well as in $\R^n$ \cite{AAAW_2013_CSFzoo}, while for the higher-order flows, the analysis is much more complicated and such classification results are not available.
However, some rigidity (i.e., partial classification) results are already known for \eqref{eq:SDF} in the planar case \cite{EGMWW15,Rybka-Wheeler_2024_classification}.
Here we obtain a rigidity theorem (\Cref{cor:CSF_SDF_CF_soliton}), which is not only new in $\R^2$ but also the first rigidity result in higher codimensions $\R^n$ with $n\geq3$.

\subsection{Fourth-order flows of curve straightening type}

Now we turn to flows involving the curve straightening effect, namely \eqref{eq:lambda-EF} with $\lambda\geq0$.
Up to rescaling, without loss of generality we may only consider the two cases $\lambda=0,1$.

We first address the case $\lambda=0$, called the \emph{free elastic flow},
\begin{equation}\label{eq:FEF}
    \partial_t\gamma = -\nabla_s^2\kappa - \frac{1}{2}|\kappa|^2\kappa, \tag{FEF}
\end{equation}
which is also equivalent to the Willmore flow with initial surfaces $\Sigma_0$ invariant under translation, i.e., $\Sigma_0=\gamma_0(\R)\times\R\subset\R^{n+1}$.
In the compact case, thanks to the decay property of the bending energy, the flow always exists globally in time only assuming finiteness of the bending energy \cite{Dziuk-Kuwert-Schatzle_2002}, and also some asymptotic properties have recently been obtained \cite{Wheeler-Wheeler_2024_parallel,Miura_Wheeler_2024_free}.
Here we establish the non-compact counterpart for the first time.
The limit curve is again a straight line, but the direction is not determined due to the absence of the tangent aligning effect.

\begin{theorem}\label{thm:main_FEF}
    Let $\gamma_0\in \dot{C}^\infty(\R;\R^n)$ be a properly immersed curve with $\inf_{\R}|\partial_x\gamma_0|>0$.
    Suppose that $B[\gamma_0]<\infty$.
    Then there exists a unique, smooth, properly immersed, global-in-time solution $\gamma:[0,\infty)\times\R\to\R^n$ to the free elastic flow \eqref{eq:FEF} with initial datum $\gamma_0$.
    Along the flow the energy $B[\gamma(t,\cdot)]$ continuously decreases in $t\geq0$.
    
    In addition,
    \begin{equation}
        \lim_{t\to\infty}\sup_\R|\partial_s^m\kappa(t,\cdot)| = 0
    \end{equation}
    holds for all $m\in\N_0$.
    In particular, after reparametrization by arclength, translation and rotation, the solution locally smoothly converges to a straight line.
\end{theorem}

Finally, we discuss the case $\lambda=1$, which we simply call the \emph{elastic flow},
\begin{equation}\label{eq:EF}
    \partial_t\gamma = -\nabla_s^2\kappa - \frac{1}{2}|\kappa|^2\kappa + \kappa. \tag{EF}
\end{equation}
There are a number of known results about global existence and convergence in the compact case, see, e.g.,
\cite{Polden1996,Dziuk-Kuwert-Schatzle_2002,DallAcqua-Pozzi_2014_Willmore-Helfrich,MR4048466,MR4160436,MR2911840,Mantegazza-Pozzetta_2021_LS,Novaga-Okabe_2017_conv,Novaga_Okabe_2014_infinite,Diana_2024_EF_FB,Lin-Schwetlick-Tran_2022_spline}.

To the authors' knowledge, the only known result for the infinite-length elastic flow is a pioneering work by Novaga--Okabe \cite{Novaga_Okabe_2014_infinite}.
Roughly speaking, they show that if the initial curve is planar and if the ends are asymptotic to the $e_1$-axis and represented by graphs of class \( W^{m,2} \) for all \( m \geq 0 \), then there exists a global-in-time elastic flow, and the solution locally smoothly sub-converges to either a line or a borderline elastica (\Cref{fig:borderline}).
Here and in the sequel, sub-convergence means that, for any time sequence $t_j\to\infty$, there exists a subsequence $t_{j'}\to\infty$ along which the flow converges (in a suitable sense).
Their result does not include any uniqueness because their construction is based on approximating an infinite-length solution by finite-length solutions using the Arzel\`a--Ascoli theorem.

In this paper, we apply our energy method to improve the result of Novaga--Okabe in several aspects:
we prove uniqueness, extend the analysis to general codimension, cover more general ends (possibly unbounded), and ensure the horizontality of the convergence limits.
To this end, we apply our direction energy method and regard the elastic flow as the gradient flow of the adapted energy
\begin{equation}\label{eq:energy_E=B+D}
    E[\gamma] \vcentcolon= B[\gamma]+D[\gamma] = \int_\gamma \left( \frac{1}{2}|\kappa|^2+\frac{1}{2}|\partial_s\gamma-e_1|^2 \right) ds.
\end{equation}
The finiteness of $E$ is equivalent to having graphical ends $u$ with $u'\in W^{1,2}$ (see \Cref{lem:direction_graphical}), so the diverging examples as in \Cref{rem:examples_finite_direction} are still admissible.

Our main result can be summarized as follows.

\begin{theorem}[\Cref{thm:global_lambda_elastic_flow,thm:convergence_lambda_elastic_flow,thm:convergence_elastic_flow}]\label{thm:main_EF_global}
    Let $\gamma_0\in \dot{C}^\infty(\R;\R^n)$ with $\inf_{\R}|\partial_x\gamma_0|>0$.
    Suppose that $E[\gamma_0]<\infty$.
    Then there exists a unique, smooth, properly immersed, global-in-time solution $\gamma:[0,\infty)\times\R\to\R^n$ to the elastic flow \eqref{eq:EF} with initial datum $\gamma_0$.
    Along the flow the energy $E[\gamma(t,\cdot)]$ continuously decreases in $t\geq0$.
    
    In addition, after any reparametrization by arclength and some translation, the solution sub-converges to an elastica $\gamma_\infty$ locally smoothly.
    The limit curve $\gamma_\infty$ is either a straight line or a (planar) borderline elastica, and in each case the curve $\gamma_\infty$ has asymptotically horizontal ends: $\partial_s\gamma_\infty(s)\to e_1$ as $s\to\pm\infty$.
\end{theorem}

\begin{remark}
    In stark contrast to the previous flows, global-in-space convergence does not hold, in general.
    Indeed, we can construct a peculiar example with loops ``escaping'' to infinity (see \Cref{ex:two_loop} and \Cref{fig:twoloop}).
\end{remark}

\begin{figure}[htbp]
    \centering
    \includegraphics[width=0.6\linewidth]{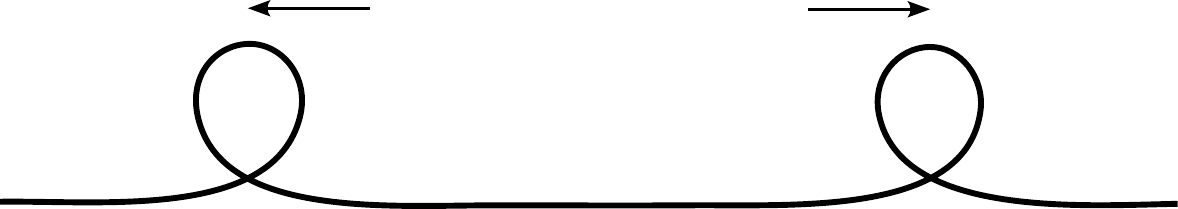}
    \caption{Two loops moving apart along the elastic flow.}
    \label{fig:twoloop}
\end{figure}

Furthermore, we find several conditions for initial data such that we can identify the type of the limit elastica, which in particular implies full (locally smooth) convergence \emph{without} taking a time subsequence.
Among other results (see \Cref{cor:EF_full_convergence}), we obtain the optimal energy threshold below which the same convergence holds as in the case of the curve shortening flow.

% Here is a brief list of results (see \Cref{cor:EF_full_convergence} for details):
% \begin{itemize}
%     \item If $\gamma_0$ has point symmetry, then the solution fully converges to a line.
%     \item If $\gamma_0$ has reflection symmetry in the $e_1$-direction, then full convergence holds.
%     The limit depends on the direction $\partial_s\gamma_0$ at the reflection point.
%     \item If $E[\gamma_0]<8$, then the solution fully converges to a line.
%     The energy threshold is optimal.
%     \item If $\gamma_0$ is planar and has nonzero rotation number, then the solution fully converges to a borderline elastica after suitable reparametrization.
% \end{itemize}

\begin{corollary}\label{cor:EF_convergence_line_<8}
    Let $\gamma_0\in \dot{C}^\infty(\R;\R^n)$ with $\inf_{\R}|\partial_x\gamma_0|>0$.
    Suppose that
    \begin{equation}
        E[\gamma_0]<8.
    \end{equation}
    Then the elastic flow \eqref{eq:EF} starting from $\gamma_0$ satisfies
    \begin{equation}
        \lim_{t\to\infty}\sup_{\R}|\partial_s\gamma(t,\cdot)-e_1|=0, \quad \lim_{t\to\infty}\sup_{\R}|\partial_s^m\kappa(t,\cdot)|=0,
    \end{equation}
    for all $m\in\N_0$.
    In particular, after reparametrization by arclength and translation, the solution locally smoothly converges to a horizontal line.
\end{corollary}

\begin{remark}
    The energy threshold is optimal.
    Indeed, a borderline elastica $\gamma_0$ with $E[\gamma_0]=8$ is a stationary solution.
\end{remark}

We finally mention that, although the curve shortening flow preserves many positivity properties such as planar embeddedness or graphicality (cf.\ \Cref{cor:CSF_convergence}), a wide class of higher-order flows does not due to the lack of maximum principles \cite{Blatt2010,ElliottMaier-Paape2001}. %(see, e.g., \cite{Blatt2010,ElliottMaier-Paape2001} and, for more on elastic flows, \cite{Kemmochi_Miura_2024_migration,miura2024migrating}).
Recently, optimal energy thresholds for preserving embeddedness of elastic flows of closed curves have been obtained in \cite{Miura_Muller_Rupp_2025_optimal}.
Our new energy method is also useful in this direction; we can obtain optimal thresholds in terms of $E$ for preserving planar embeddedness (resp.\ graphicality) in the non-compact case.
This will be addressed in our future work.

\section{Preliminaries}\label{sec:preliminaries}

\subsection{Notation}\label{subsec:notation}

Let $\N$ denote the set of positive integers and $\N_0:=\N\cup\{0\}$.
For $k\in\N_0$ and $\alpha\in (0,1]$, we denote by $C^{k,\alpha}(\R;\R^n)$ the usual space of $k$-times differentiable functions $u\colon\R\to\R^n$ with finite norm
\begin{align}
    \Vert u \Vert_{C^{k,\alpha}(\R;\R^n)} \vcentcolon = \sum_{i=0}^k \Vert \partial_x^i u\Vert_\infty + [\partial_x^i u]_\alpha,
\end{align}
where $[\cdot]_\alpha$ denotes the $\alpha$-H\"older seminorm if $\alpha<1$ and the best Lipschitz constant for $\alpha=1$.
Since we work with non-compact curves, we consider, for $k\in\N$ and $\alpha\in (0,1]$, the function spaces
\begin{align}
    \dot C^{k,\alpha}(\R;\R^n) \vcentcolon = \{u \colon \R\to\R^n \mid \partial_x u \in C^{k-1,\alpha}(\R;\R^n)\}.
\end{align}
In similar fashion, we define $\dot{C}^k(\R;\R^n)$ for $k\geq 1$.
We also set
\begin{align}
    \dot C^\infty(\R;\R^n) \vcentcolon = \bigcap_{k=1}^\infty \dot{C}^{k}(\R;\R^n).
\end{align}
If the codomain is clear from the context, we also just write $C^{k,\alpha}(\R)$ or $\dot{C}^{k,\alpha}(\R)$.

\subsection{Basic properties of the direction energy}\label{subsec:direction}

The bending energy $B$ is invariant with respect to isometries of $\R^n$, and also any reparametrization.
On the other hand, the direction energy $D$ has less invariances; translation, rotation around the $e_1$-axis, reflection in directions orthogonal to $e_1$, and orientation-preserving reparametrization.

\begin{lemma}\label{lem:direction_horizontal}
	Let $\gamma:\R\to\R^n$ be a smooth immersion with $\inf_\R|\partial_x\gamma|>0$ and $D[\gamma] <\infty$. Then
    \[
    \lim_{x\to\pm\infty}\langle\gamma(x),e_1\rangle = \pm\infty, 
    \]
    and hence $\gamma$ is proper.
    If in addition $\gamma\in \dot{C}^{1,1}(\R;\R^n)$, then
    \[
    \lim_{x\to\pm\infty} \partial_s\gamma(x) = e_1.
    \]
\end{lemma}

\begin{proof} 
    The first assertion follows by taking the limits $x\to\pm\infty$ in the identity
    \begin{align}
        \int_{0}^x |\partial_x\gamma(x')|dx'-\langle\gamma(x),e_1\rangle = \int_0^x\big( |\partial_x\gamma(x')| - \langle\partial_x\gamma(x'),e_1\rangle \big)dx' -\langle \gamma(0),e_1\rangle
    \end{align}
    since the right hand side is bounded thanks to the assumption $D[\gamma]<\infty$.
    
    The second assertion follows since the integrability of $|\partial_x\gamma| - \langle\partial_x\gamma,e_1\rangle$ and its Lipschitz continuity imply that $|\partial_x\gamma| - \langle\partial_x\gamma,e_1\rangle \to 0$ as $x\to\pm\infty$.
\end{proof}

\begin{lemma}\label{lem:direction_graphical}
    Let $\gamma\in \dot C^{1,1}(\R;\R^n)$ be a smooth immersion with $\inf_\R|\partial_x\gamma|>0$.
    \begin{enumerate}
        \item We have $D[\gamma]<\infty$ if and only if outside a slab $S_R:=\{ p\in\R^n \mid -R \leq \langle p,e_1\rangle \leq R\}$ the curve $\gamma$ is represented by a graph curve $\{(x,u(x))\in\R^n\mid x\in\R\}$ of $u\in \dot{C}^{1,1}(\R;\R^{n-1})$ with $u'\in L^2(\R;\R^{n-1})$.
        \item We have $B[\gamma]+D[\gamma]<\infty$ if and only if outside a slab $S_R$ the curve $\gamma$ is represented by a graph curve of $u\in \dot{C}^{1,1}(\R;\R^{n-1})$ with $u'\in W^{1,2}(\R;\R^{n-1})$.
    \end{enumerate}
\end{lemma}

\begin{proof}
    We first prove the ``only if'' part of (i).
    Suppose $D[\gamma]<\infty$.
    By \Cref{lem:direction_horizontal}, we deduce that outside a slab $S_R$ the curve $\gamma$ is a graph curve with bounded gradient; more precisely, there are $R>0$ and $u\in \dot{C}^{1,1}(\R;\R^{n-1})$ such that
    \[
    \gamma(\R)\setminus S_R = \{(x,u(x))\in\R^n\mid x\in\R\}\setminus S_R.
    \]
    Since $D[\gamma]<\infty$, and since the energy inside the slab $S_R$ does not affect the finiteness, the graph of $u$ also has finite direction energy $\tilde{D}[u]<\infty$, where 
    \[
    \tilde{D}[u]:=\int_\R \Big(  1 - \frac{1}{\sqrt{1+|u'|^2}} \Big) \sqrt{1+|u'|^2} dx = \int_\R  \frac{|u'|^2}{1+(1+|u'|^2)^{\frac{1}{2}}} dx \geq c\|u'\|_{L^2}^2,
    \]
    with $c>0$ depending only on $\|u'\|_\infty$.
    This completes the ``only if'' part of (i).
    The ``if'' part is easier since $\tilde{D}[u] \leq \frac{1}{2}\|u'\|_{L^2}^2 < \infty$,
    implying that $D[\gamma]<\infty$.
    This completes the proof of (i).

    Now we turn to (ii).
    Suppose $B[\gamma]+D[\gamma]<\infty$.
    Let $u\in \dot{C}^{1,1}$ be as above; note that $u''\in L^\infty$.
    Then in addition to $u'\in L^2$ obtained above, we have the finiteness of the bending energy $\tilde{B}[u]<\infty$, where 
    \begin{align}
        \tilde{B}[u] &:= \int_\R \frac{|u''|^2}{2(1+|u'|^2)^3} \sqrt{1+|u'|^2} dx \geq c\|u''\|_{L^2}^2.
    \end{align}
    This implies the ``only if'' part of (ii).
    The ``if'' part is again easier since $\tilde{B}[u] \leq \frac{1}{2}\|u''\|_{L^2}^2 < \infty$.
    The proof is now complete.
\end{proof}

\section{Gradient flow structures}\label{sec:energy_decay}

In this section we obtain key energy-decay properties for the flows under consideration, through a cut-off argument.

From this section onward, we will primarily use \( \gamma: (t,x) \mapsto \gamma(t,x) \) to denote a flow of curves.
In this case, we sometimes abbreviate \( \gamma(t,\cdot) \) as \( \gamma(t) \).
Furthermore, we write \( \int f ds \) to denote an integral of the following form:
\[
\int fds:= \int_{\gamma(t,\cdot)} f(t,\cdot) ds = \int_{\R} f(t,x) |\partial_x \gamma(t,x)| dx.
\]
Finally, we define the norm
\[
\|f\|_{L^p(ds)} := \left( \int |f|^p ds \right)^{1/p},
\]
which we may also express as \( \|f\|_{L^p(ds(t))} \) to emphasize the time dependence.

Throughout this section we use $\xi$ and $V$ to denote the tangential and normal velocity of a flow $\gamma$, respectively; namely,
\begin{equation}
    \xi\vcentcolon= \langle \partial_t \gamma,\partial_s\gamma \rangle, \quad V\vcentcolon=\partial_t^{\perp}\gamma = \partial_t\gamma - \xi\partial_s\gamma.
\end{equation}

\subsection{Evolution of localized energy functionals}

We first compute general time-derivative formulae for the direction and bending energy.

\begin{lemma}\label{lem:direction_derivative_cutoff}
    Let $T\in(0,\infty)$ and $\gamma:[0,T]\times\R\to\R^n$ be a smooth family of immersed curves.
    Let $\eta\in C_c^\infty(\R)$ and define
	\begin{align}
		D[\gamma|\eta] \vcentcolon = \int (1 -\langle \partial_s\gamma, e_1\rangle) \eta ds.
	\end{align}
    Then
	\begin{align}\label{eq:direction_derivative_cutoff}
		\frac{d}{dt} D[\gamma|\eta] + \int\langle\kappa,V\rangle\eta ds = -\int \xi (1-\langle \partial_s\gamma,e_1\rangle) \partial_s\eta ds + \int\langle V,e_1\rangle \partial_s\eta ds.
	\end{align}
\end{lemma}

\begin{proof}
	Standard geometric evolution formulae (cf.\ \cite[Lemma 2.1]{Dziuk-Kuwert-Schatzle_2002}) yield
	\begin{align}
        \begin{split}\label{eq:derivative_formula_1}
            \partial_t(ds) &= (\partial_s\xi-\langle\kappa, V\rangle)ds,\\
		\partial_t (\partial_s\gamma) &= \nabla_s V+\xi\kappa.
        \end{split}
	\end{align}
	We thus compute
	\begin{align}
		\partial_t\Big[ ( 1 -\langle \partial_s\gamma, e_1\rangle ) ds\Big]  = \Big( -\langle \nabla_sV+\xi\kappa,e_1\rangle + (1-\langle \partial_s\gamma,e_1\rangle)(\partial_s\xi-\langle\kappa,V\rangle) \Big)ds.
	\end{align}
    Now we differentiate $D[\gamma|\eta]$.
	By the regularity assumption, the integral is finite, and we may differentiate under the integral.
    Hence we have
	\begin{align}
		\frac{d}{dt} D[\gamma|\eta] &= \int \Big(-\langle \nabla_sV+\xi\kappa,e_1\rangle + (1-\langle \partial_s\gamma,e_1\rangle)(\partial_s\xi-\langle\kappa,V\rangle)\Big)\eta ds.
	\end{align}
    After rearranging,
    \begin{align}
		\frac{d}{dt} D[\gamma|\eta] + \int\langle\kappa,V\rangle\eta ds &= \int \partial_s\xi(1-\langle \partial_s\gamma,e_1\rangle)\eta ds - \int\langle\nabla_sV,e_1\rangle\eta ds\\
        & \qquad +\int \Big(-\xi\langle\kappa,e_1\rangle +\langle \partial_s\gamma,e_1\rangle\langle\kappa,V\rangle\Big)\eta ds.
	\end{align}
	Integrating by parts in the first and second term of the right hand side, and in particular using $\nabla_s (e_1^\perp) = -\langle \partial_s\gamma,e_1\rangle \kappa$ for the second term, we obtain the desired formula.
\end{proof}

\begin{lemma}\label{lem:bending_derivative_cutoff}
    Let $T\in(0,\infty)$ and $\gamma:[0,T]\times\R\to\R^n$ be a smooth family of immersed curves.
    Let $\eta\in C_c^\infty(\R)$ and define
	\begin{align}
		B[\gamma|\eta] \vcentcolon = \frac{1}{2}\int |\kappa|^2 \eta ds.
	\end{align}
    Then
	\begin{align}\label{eq:bending_derivative_cutoff}
		& \frac{d}{dt} B[\gamma|\eta] + \int\langle -\nabla_s^2\kappa-\frac{1}{2}|\kappa|^2\kappa ,V\rangle\eta ds \\
        &\qquad =  -\int \frac{1}{2}|\kappa|^2\xi \partial_s\eta ds + 2\int\langle \nabla_s\kappa,V\rangle\partial_s\eta  ds + \int\langle \kappa,V\rangle \partial_s^2 \eta ds. 
	\end{align}
\end{lemma}

\begin{proof}
	In addition to \eqref{eq:derivative_formula_1}, recalling from \cite[Lemma 2.1]{Dziuk-Kuwert-Schatzle_2002} that
	\begin{align}
		\partial_t^{\perp}\kappa &= \nabla_s^2 V+\langle \kappa,V\rangle\kappa +\xi\nabla_s\kappa, \label{eq:dtkappa}
	\end{align}
	we argue along the lines of the proof of Lemma \ref{lem:direction_derivative_cutoff} to obtain
	\begin{align}
		\frac{d}{dt} B[\gamma|\eta] &= \int \Big(\langle \kappa, \nabla_s^2V\rangle + \frac{1}{2}|\kappa|^2 \langle \kappa, V\rangle + \xi\langle \kappa,\nabla_s\kappa\rangle  +\frac{1}{2}|\kappa|^2 \partial_s\xi\Big)\eta ds.
	\end{align}
	Integrating by parts in the first term yields
	\begin{align}
		\int \langle \kappa, \nabla_s^2V\rangle \eta ds &= -\int\langle \nabla_s\kappa, \nabla_sV\rangle \eta ds - \int\langle \kappa,\nabla_sV\rangle \partial_s\eta  ds \\
		&= \int\langle \nabla_s^2 \kappa, V\rangle \eta ds  + 2\int\langle \nabla_s\kappa,V\rangle\partial_s\eta  ds + \int\langle \kappa,V\rangle \partial_s^2 \eta ds,
	\end{align}
	so that, after some rearranging,
	\begin{align}
		& \frac{d}{dt} B[\gamma|\eta] + \int \langle -\nabla_s^2\kappa-\frac{1}{2}|\kappa|^2\kappa, V\rangle \eta ds \\
        & \qquad =  \int \partial_s\Big(\frac{1}{2}|\kappa|^2\xi\Big)\eta ds + 2\int\langle \nabla_s\kappa,V\rangle\partial_s\eta  ds + \int\langle \kappa,V\rangle \partial_s^2 \eta ds.
	\end{align}
    Integration by parts for the first term yields the assertion.
\end{proof}

\subsection{Energy identity}

Now we turn to the proof of the energy identity.
We first address the flows having the structure of an $L^2$-gradient flow, namely \eqref{eq:CSF} and \eqref{eq:lambda-EF}.
Given $\sigma,\lambda\geq0$, we define the functional 
\[
F[\gamma]=F_{\sigma,\lambda}[\gamma]\vcentcolon = \sigma B[\gamma]+\lambda D[\gamma] = \int_\gamma\left( \frac{\sigma}{2}|\kappa|^2+\lambda(1-\langle \partial_s\gamma,e_1\rangle ) \right)ds.
\]

\begin{proposition}\label{prop:energy_decay}
    Let $\sigma,\lambda\geq 0$, $T\in(0,\infty)$, and $\gamma:[0,T]\times\R\to\R^n$ be a smooth solution to
	\begin{align}
		\partial_t^{\perp}\gamma = -\sigma \left( \nabla_s^2\kappa+\frac{1}{2}|\kappa|^2\kappa \right) + \lambda \kappa,
	\end{align}
	such that
    \begin{align}
        \inf_{[0,T]\times\R} |\partial_x\gamma|>0, \quad \sum_{k=1}^M \Vert \partial_x^k\gamma\Vert_{L^\infty([0,T]\times \R)} + \Vert \partial_t\gamma\Vert_{L^\infty([0,T]\times \R)} <\infty,\label{eq:energy_decay_regularity}
    \end{align}
    where $M=4$ if $\sigma>0$, while $M=2$ if $\sigma=0$.
    If $F[\gamma(0)]<\infty$, then
    \begin{align}
        F[\gamma(T)] - F[\gamma(0)] = -\int_0^T \int |\partial_t^\perp\gamma|^2 ds dt.
    \end{align}
\end{proposition}

\begin{proof}
    Suppose $\sigma>0$.
    Let $\eta\in C_c^\infty(\R)$ and define $D[\gamma|\eta]$ and $B[\gamma|\eta]$ as above.
    Combining \Cref{lem:direction_derivative_cutoff,lem:bending_derivative_cutoff}, we find
    \begin{align}
        &\frac{d}{dt} F[\gamma|\eta] + \int |V|^2 \eta ds \\
        &\qquad = -\sigma\int \frac{1}{2}|\kappa|^2\xi \partial_s\eta ds + 2\sigma\int\langle \nabla_s\kappa,V\rangle\partial_s\eta  ds + \sigma\int\langle \kappa,V\rangle \partial_s^2 \eta ds \\
        &\qquad \quad -\lambda \int \xi (1-\langle \partial_s\gamma,e_1\rangle) \partial_s\eta ds + \lambda \int\langle V,e_1\rangle \partial_s\eta ds.\label{eq:dt_F_1}
    \end{align}
    Take a cutoff function $\varphi_0\in C_c^\infty(\R)$ with $\chi_{[-1/2,1/2]}\leq \varphi_0\leq \chi_{[-1,1]}$, let $R>1$, define $\varphi(x):=\varphi_0(R^{-1}x)$, and set $\eta=\varphi^4$.
    In the following, $C\in (0,\infty)$ is a constant depending on $\lambda, \sigma, \varphi_0$, and the bounds on $\gamma$ in \eqref{eq:energy_decay_regularity}, that may change from line to line.
    Then, by \eqref{eq:energy_decay_regularity},
	\begin{align}\label{eq:Energy_decay_control_eta_derivatives}
		|\partial_s\eta|\leq CR^{-1} \varphi^3, \quad |\partial_s^2\eta|\leq C R^{-1} \varphi^2,
	\end{align}
    and also
    \begin{align}
        \int_{-R}^R ds \leq C R.\label{eq:energy_decay_local_length_control}
    \end{align}
    % Bounding the integrands on the right hand side of \eqref{eq:dt_F_1} using \eqref{eq:energy_decay_regularity} and \eqref{eq:Energy_decay_control_eta_derivatives}, from \eqref{eq:energy_decay_local_length_control} we conclude
    Using \eqref{eq:energy_decay_regularity}, we can bound the integrands on the right hand side of \eqref{eq:dt_F_1} by $C\int(|\partial_s\eta|+|\partial_s^2\eta|)ds$. Since this is bounded by $C$ thanks to \eqref{eq:Energy_decay_control_eta_derivatives} and \eqref{eq:energy_decay_local_length_control}, we conclude
	% By assumption \eqref{eq:energy_decay_regularity}, we have $|\xi|\leq C$. Using \eqref{eq:Energy_decay_control_eta_derivatives}, and applying Young's inequality for each term involving $V$ on the right hand side, we deduce that
 %    \begin{align}
 %        \frac{d}{dt} F[\gamma|\varphi^4] + \frac{1}{2}\int |V|^2 \varphi^4 ds &\leq  C \sigma^2 R^{-2} \int|\nabla_s\kappa|^2 \varphi^2 ds + C \sigma^2R^{-2}\int |\kappa|^2 \varphi^2 ds \\
 %        &\qquad + C\lambda R^{-1} F[\gamma|\varphi^3]  + C\lambda^2 R^{-2} \int_{-R}^R ds.\label{eq:dt_F_2}
 %    \end{align}
 %    The boundedness \eqref{eq:energy_decay_regularity} yields that $F[\gamma|\varphi^3]\leq CR$, 
 %    \begin{align}
 %        \int_{-R}^R ds \leq C R,\label{eq:energy_decay_local_length_control}
 %    \end{align}
 %    and
    \begin{align}
        &\frac{d}{dt} F[\gamma|\varphi^4] + \int |V|^2 \varphi^4 ds \leq C.\label{eq:dt_F_3}
    \end{align}
    % With $F[\gamma;R]\vcentcolon = F[\gamma|\chi_{[-R,R]}]$, by the choice of $\varphi$ we obtain from integrating \eqref{eq:dt_F_3}
    % \begin{align}
    %     F[\gamma;R/2] + \frac{1}{2} \int_0^t \int_{-R/2}^{R/2}|V|^2 ds d\tau \leq F[\gamma(0);R] + (C+C R^{-1})T.
    % \end{align}
    Integrating over $[0,t] \subset [0,T]$ and sending $R\to\infty$, we conclude 
    from the monotone convergence theorem that $F[\gamma] = \lim_{R\to\infty} F[\gamma|\varphi^4]<\infty$ for all $t\in [0,T]$ with
	\begin{align}\label{eq:1118-3}
		F[\gamma(t)] + \int_0^t \int |V|^2ds d\tau \leq F[\gamma(0)] + CT.
	\end{align}
    In particular, it follows that
    \begin{equation}\label{eq:0108-1}
        t\mapsto F[\gamma(t)]\in L^\infty(0,T) \quad \text{and} \quad t\mapsto \|V\|_{L^2(ds(t))}\in L^2(0,T).
    \end{equation}
    We have $\eta\to 1$ a.e.\ in $(0,T)\times \R$ as $R\to\infty$ and, by \eqref{eq:Energy_decay_control_eta_derivatives} and \eqref{eq:energy_decay_local_length_control},
    \begin{align}\label{eq:0108-2}
        \Vert\partial_s\eta\Vert_{L^2(ds(t))}\to 0, \quad \Vert\partial_s^2\eta\Vert_{L^2(ds(t))}\to 0 \quad \text{ in } L^2(0,T).
    \end{align}
    Therefore, integrating \eqref{eq:dt_F_1} in time and sending $R\to\infty$ imply the assertion; indeed, by \eqref{eq:energy_decay_regularity}, 
    the absolute value of the right hand side is bounded by
    \begin{align}
        &C \Vert\xi\Vert_\infty \Vert \partial_s\eta\Vert_\infty \int_0^T \int \Big(\frac{\sigma}{2}|\kappa|^2 + \lambda(1-\langle\partial_s\gamma,e_1\rangle)\Big)ds dt \\
        &\quad + C (\Vert \nabla_s\kappa\Vert_\infty + 1) \int_0^T \int |V| |\partial_s\eta| ds dt + C\Vert \kappa\Vert_\infty \int_0^T \int |V||\partial_s^2\eta| ds dt\\
        &\leq   C\|\partial_s\eta\|_\infty\int_0^{T} F[\gamma] dt + C\int_0^{T}\int|V|(|\partial_s\eta|+|\partial_s^2\eta|)dsdt,
    \end{align}
    which converges to zero thanks to \eqref{eq:Energy_decay_control_eta_derivatives}, \eqref{eq:energy_decay_local_length_control}, \eqref{eq:0108-1}, and \eqref{eq:0108-2}. 

    Finally, if \( \sigma = 0 \), the exact same argument applies after removing all terms involving \( \sigma \), which results in only terms involving derivatives up to order $2$.
    Thus, the statement follows. 
\end{proof}

Our method also works for gradient flows which are not of $L^2$-type.
We now address fourth-order flows of curve shortening type including \eqref{eq:SDF} and \eqref{eq:CF}.

\begin{proposition}\label{prop:SDF_Chen_energy_decay}
    Let $T\in(0,\infty)$ and $\gamma:[0,T]\times\R\to\R^n$ be a smooth solution to
	\begin{align}
		\partial_t^\perp \gamma = -\nabla_s^2\kappa + \mu |\kappa|^2\kappa
	\end{align}
	for some constant $\mu\geq 0$ such that
    \begin{align}
        \inf_{[0,T]\times\R} |\partial_x\gamma|>0, \quad \sum_{k=1}^4 \Vert \partial_x^k\gamma\Vert_{L^\infty([0,T]\times \R)}+ \Vert \partial_t\gamma\Vert_{L^\infty([0,T]\times\R)}<\infty.
    \end{align}
    If $D[\gamma(0)]<\infty$, then
    \begin{align}
        D[\gamma(T)]-D[\gamma(0)] = - \int_0^T\int (|\nabla_s\kappa|^2 +\mu |\kappa|^4) ds dt.
    \end{align}
\end{proposition}
\begin{proof}
    We use \Cref{lem:direction_derivative_cutoff} with the same cutoff function $\eta=\varphi^4$ as in the proof of \Cref{prop:energy_decay}.
    Integration by parts for terms involving $\nabla_s^2\kappa$ yields
    \begin{align}
        &\frac{d}{dt}D[\gamma|\eta] + \int (|\nabla_s\kappa|^2+\mu|\kappa|^4)\eta ds \\
        &= - \int \xi (1-\langle\partial_s\gamma,e_1\rangle) \partial_s\eta ds +\mu \int |\kappa|^2\langle \kappa,e_1\rangle \partial_s\eta ds \\
        &\quad + \int \langle \nabla_s\kappa,e_1\rangle \partial_s^2\eta ds - \int \langle \nabla_s\kappa, \kappa\rangle\langle \partial_s\gamma,e_1\rangle \partial_s\eta ds - \int \langle \nabla_s\kappa,\kappa\rangle\partial_s\eta ds,\label{eq:prop5-5-1}
    \end{align}
    where we again used $\nabla_s (e_1^{\perp}) = -\langle \partial_s\gamma,e_1\rangle \kappa$.   
    As before, bounding the right hand side by $C$, integrating in time over $[0,t]\subset[0,T]$, and sending $R\to\infty$, we obtain
    \begin{align}
        D[\gamma(t)] +\int_0^t \int (|\nabla_s\kappa|^2 + \mu|\kappa|^4) ds \leq D[\gamma(0)] + CT,
    \end{align}
    which implies 
    \begin{align}
        t \mapsto D[\gamma(t)] \in L^\infty(0,T), \qquad 
        t \mapsto \|\nabla_s\kappa\|_{L^2(ds(t))}^2+\mu\Vert \kappa\Vert_{L^4(ds(t))}^4 \in L^1(0,T).
    \end{align}
    Integrating \eqref{eq:prop5-5-1} over $[0,T]$, the right hand side may be bounded by
    \begin{align}
        &\Vert \xi\Vert_\infty \Vert \partial_s\eta\Vert_\infty \int_0^T D[\gamma] dt \\
        &\quad + \int_0^T \int \Big(\mu|\kappa|^3 |\partial_s\eta|+ |\nabla_s\kappa||\partial_s^2\eta|+2|\kappa| |\nabla_s\kappa| |\partial_s\eta|\Big) ds dt.
    \end{align}
    Using H\"older's inequality and the fact that, by \eqref{eq:Energy_decay_control_eta_derivatives}, for any $p>1$, we have
    \begin{align}
        \Vert \partial_s\eta\Vert_{L^p(ds(t))} + \Vert \partial_s^2\eta\Vert_{L^p(ds(t))} \to 0\quad \text{ in }L^p(0,T)
    \end{align}
    as $R\to\infty$  yield the claim.
\end{proof}

\section{Localized interpolation estimates}\label{sec:interpolation}

The goal of this section is to prove the fundamental Gagliardo--Nirenberg-type estimate, \Cref{prop:interpolation_P} below.
Throughout this section we fix a smooth immersion $\gamma\colon \R\to\R^{n}$ and a cutoff function
\begin{align}\label{eq:def_zeta}
	\zeta\in C_c^\infty(\R) \text{ with } \zeta\geq0 \text{ and  }|\partial_s\zeta|\leq \Lambda,
\end{align}
where the arclength derivative is along $\gamma$, i.e., $\partial_s\zeta=|\partial_x\gamma|^{-1}\partial_x\zeta$, and $\Lambda>0$.

We begin with establishing $L^2$-type interpolation estimates with the weight $\zeta$. Our approach is inspired by the interpolation estimates in \cite{DallAcqua-Pozzi_2014_Willmore-Helfrich,Dziuk-Kuwert-Schatzle_2002} and the spatial cutoff in \cite{MR1900754}.

\begin{lemma}\label{lem:interpolation_IBP}
	%Let $\gamma\colon \R\to\R^{n}$ be an immersion, l
	Let $\ell\in\N$ with $\ell\geq 2$. There exists $C=C(\Lambda,\ell)>0$ such that for any smooth vector field $X$ normal along $\gamma$ and any $\varepsilon\in (0,1]$ we have
	\begin{align}
		\int|\nabla_sX|^2\zeta^{\ell}ds \leq \varepsilon \int |\nabla_s^2X|^2\zeta^{\ell+2}ds + \frac{C}{\varepsilon}\int_{[\zeta>0]}|X|^2\zeta^{\ell-2}ds.
	\end{align}
\end{lemma}
\begin{proof}
	By integration by parts and Young's inequality
	\begin{align}
		&\int |\nabla_sX|^2\zeta^{\ell}ds = - \int \langle\nabla_s^2X,X\rangle\zeta^{\ell}ds - \ell  \int \langle \nabla_sX,X\rangle \zeta^{\ell-1}\partial_s\zeta ds \\
		&\quad \leq \frac{\varepsilon}{2} \int|\nabla_s^2X|^2 \zeta^{\ell+2}ds + \frac{C}{\varepsilon} \int_{[\zeta>0]}|X|^2\zeta^{\ell-2}ds\\
		&\quad \qquad  + \frac{1}{2}\int|\nabla_sX|^2 \zeta^{\ell}ds + C(\Lambda,\ell) \int_{[\zeta>0]}|X|^2 \zeta^{\ell-2}ds.
	\end{align}
	The claim follows after absorbing and using $\varepsilon\in (0,1]$.
\end{proof}

\begin{lemma}\label{lem:interpolation_epsilon}
	Let $k\in\N$ with $k\geq 2$, and let $i\in\N$ with $i<k$. There exists $C=C(\Lambda,i,k)>0$ such that for any smooth vector field $X$ normal along $\gamma$ and any $\varepsilon\in (0,1]$ we have
	\begin{align}
		\int|\nabla_s^iX|^2\zeta^{2i}ds \leq \varepsilon\int |\nabla_s^kX|^2\zeta^{2k}ds + C \varepsilon^{\frac{i}{i-k}}\int_{[\zeta>0]}|X|^2 ds.
	\end{align}
\end{lemma}
\begin{proof}
	We proceed by induction on $k$. First, for $k=2$ and $i=1$, this is \Cref{lem:interpolation_IBP} with $\ell=2$.
	
	Suppose the statement is true for $k\geq2$. Let $0<i<k+1$. In the following $C=C(\Lambda,i,k)>1$ is a constant that is allowed to change from line to line.
	
	We first consider the case $i<k$. Then by the induction hypothesis, we obtain
	\begin{align}
		\int|\nabla_s^{k-1}X|^2\zeta^{2k-2}ds &\leq \delta \int|\nabla_s^kX|^2\zeta^{2k}ds + C \delta^{1-k}\int_{[\zeta>0]}|X|^2ds.
	\end{align}
	for all $\delta\in(0,1]$. On the other hand, \Cref{lem:interpolation_IBP} (with $\ell=2k$) yields
	\begin{align}
		&\int|\nabla_s^{k}X|^2\zeta^{2k}ds \leq \frac{\eta}{2} \int|\nabla_s^{k+1}X|^2\zeta^{2k+2}ds + \frac{C}{\eta}\int|\nabla_s^{k-1}X|^2\zeta^{2k-2}ds \\
        &\qquad \leq \frac{\eta}{2} \int|\nabla^{k+1}_sX|^2\zeta^{2k+2}ds + \frac{C \delta}{\eta} \int |\nabla_s^kX|^2\zeta^{2k}ds + \frac{C\delta^{1-k}}{\eta}\int_{[\zeta>0]}|X|^2ds, \label{eq:lem:interpolation_1}
	\end{align}
	for any $\eta\in(0,1]$.
	Choosing $\delta = \frac{\eta}{2C}<1$ in the above equation, we may absorb the second term. Now, again, by the induction hypothesis, for any $\varepsilon\in (0,1]$, we have	
	\begin{align}
		&\int |\nabla_s^{i}X|^2\zeta^{2i}ds \leq \varepsilon \int|\nabla_s^kX|^2\zeta^{2k}ds + C \varepsilon^{\frac{i}{i-k}}\int_{[\zeta>0]}|X|^2ds,\\
		&\qquad \leq \varepsilon\eta \int|\nabla_s^{k+1}X|^2\zeta^{2k+2}ds + C\varepsilon\delta^{-k} \int_{[\zeta>0]}|X|^2ds + C \varepsilon^{\frac{i}{i-k}} \int_{[\zeta>0]}|X|^2ds,
	\end{align}
	plugging in \eqref{eq:lem:interpolation_1} and using the choice of $\delta$. We now take $\eta = \varepsilon^{\frac{1}{k-i}}\in (0,1]$ and $\tilde{\varepsilon} = \varepsilon^{\frac{k+1-i}{k-i}}\in (0,1]$, and observe that $\varepsilon\eta = \tilde{\varepsilon}$, that $\varepsilon\delta^{-k} = C \varepsilon\eta^{-k} = C\varepsilon^{\frac{i}{i-k}}$, and that $\varepsilon^{\frac{i}{i-k}} = \tilde{\varepsilon}^{\frac{i}{i-(k+1)}}$. This yields the statement (for $k$ replaced by $k+1$) in the case $0<i<k$ for $\varepsilon$ replaced by $\tilde{\varepsilon}$.
	
	It remains to consider the case $i=k$. In this case, for any $\varepsilon,\delta\in (0,1]$, by \Cref{lem:interpolation_IBP} (with $\ell=2k$) and the induction hypothesis, we have
	\begin{align}
		&\int|\nabla_s^kX|^2 \zeta^{2k}ds \leq \frac{\varepsilon}{2}\int|\nabla_s^{k+1}X|^2\zeta^{2k+2}ds + \frac{C}{\varepsilon}\int|\nabla_s^{k-1}X|^2\zeta^{2k-2}ds \\
		&\qquad\leq \frac{\varepsilon}{2}\int|\nabla_s^{k+1}X|^2\zeta^{2k+2}ds + \frac{C}{\varepsilon}\Big(\delta\int|\nabla^k_sX|^2\zeta^{2k}ds + C \delta^{1-k}\int_{[\zeta>0]}|X|^2ds \Big).
	\end{align}
	Taking $\delta = \frac{\varepsilon}{2C}<1$ here and absorbing, we thus obtain
	\begin{align}
		\int|\nabla_s^kX|^2 \zeta^{2k}ds \leq  \varepsilon\int|\nabla_s^{k+1}X|^2\zeta^{2k+2}ds + C\varepsilon^{-k}\int_{[\zeta>0]}|X|^2ds. 
	\end{align}
	Since $\varepsilon^{\frac{k}{k-(k+1)}}=\varepsilon^{-k}$, the case $i=k$ is proven and the statement follows.
\end{proof}

We now establish a multiplicative version. For $X$ normal along $\gamma$, we define the seminorms (using the weight $\zeta$)
\begin{align}
	|X|_{0,2}&\vcentcolon= \Vert X\Vert_{L^2([\zeta>0],ds)}\\
	\vert X\vert_{k,2} &\vcentcolon = \Vert\zeta^{k} \nabla^k_sX\Vert_{L^2(ds)} +|X|_{0,2} \quad (k\geq1).
\end{align}

\begin{lemma}\label{lem:interpolation_multiplicative}
	%Let $\gamma\colon\R\to\R^n$ be an immersion, 
	Let $k\in \N$ and $i\in\N_0$ with $i\leq k$. There exists $C=C(\Lambda,i,k)>0$ such that for any smooth vector field $X$ normal along $\gamma$ we have
	\begin{align}
		\Vert \zeta^i\nabla_s^iX \Vert_{L^2(ds)} \leq C |X|_{k,2}^{\frac{i}{k}}|X|_{0,2}^{\frac{k-i}{k}}.
	\end{align}
\end{lemma}
\begin{proof}
	Without loss of generality, we may assume $k\geq 2$ and $0<i<k$, otherwise there is nothing to show.  \Cref{lem:interpolation_epsilon} implies that for all $\varepsilon\in (0,1]$, we have
	\begin{align}
		\Vert \zeta^{i} \nabla^i_sX \Vert_{L^2(ds)}\leq C\Big(\varepsilon \vert X\vert_{k,2} +\varepsilon^{\frac{i}{i-k}}|X|_{0,2}\Big).
	\end{align}
	If $\vert X\vert_{0,2}=0$, the statement is trivial. Otherwise, choosing	$\varepsilon = \Big(\frac{|X|_{k,2}}{|X|_{0,2}}\Big)^{\frac{i-k}{k}} \in (0,1]$, the statement follows.
\end{proof}

%% Not needed anymore %%
\begin{comment}
	We now give a geometric version of the Gagliardo--Nirenberg inequality.
	\begin{theorem}\label{thm:gagliardo_nirenberg}
		Let $\gamma\colon\R\to\R^{n}$ be an immersion with $L[\gamma]=\infty$. Let $p>2$ and let $X$ be normal along $\gamma$. Then, we have
		\begin{align}
			\Vert X \Vert_{L^p(ds)}\leq C(p) \Vert \nabla_sX\Vert_{L^2(ds)}^{\theta}\Vert X\Vert_{L^2(ds)}^{1-\theta},
		\end{align}
		where $\theta = \frac{p-2}{2p}$.
	\end{theorem}
	\begin{proof}

		Let $X$ be the arclength parametrization of $\gamma$. With $\theta = \frac{p-2}{2p}$, the classical Gagliardo--Nirenberg inequality for $f:\R\to\R$ \FR{TODO: Good reference, dependence on $n$?} reads
		\begin{align}
			\Vert f\Vert_{L^p(dx)}\leq C(p) \Vert f'\Vert_{L^2(dx)}^{\theta} \Vert f\Vert_{L^2(dx)}.
		\end{align}
		Applying this to $|X|\circ X$ yields 
		\begin{align}
			\Vert X\circ X \Vert_{L^p(dx)} \leq C(p) \Vert \partial_x|X\circ X| \Vert_{L^2(dx)}^{\theta}\Vert X\circ X\Vert_{L^2(dx)}^{1-\theta}.
		\end{align}
		Since $X$ is normal we have $|\partial_s|X|| \leq |\nabla_sX|$ a.e.\ in $\R$. Together with the change of variables formula, this gives the statement.
		%	Using the change of variables formula, we find
		%	\begin{align}
			%		\Vert X\Vert_{L^p(ds)}\leq C(p) \Vert\partial_s|X| \Vert_{L^2(ds)}^{\theta}\Vert X\Vert_{L^2(ds)}^{1-\theta} \leq  C(p) \Vert \nabla_sX\Vert_{L^2(ds)}^{\theta}\Vert X\Vert_{L^2(ds)}^{1-\theta}. &\qedhere
			%	\end{align}
	\end{proof}
\end{comment}

This will enable us to estimate nonlinearities arising in higher order energy evolutions. 
In order to apply $L^2$-type estimates for  $L^p$-type nonlinearities, we now prove the following key Gagliardo--Nirenberg-type interpolation inequality. The crucial observation is that this estimate distributes the weight in a way that is compatible with the seminorms $|X|_{k,2}, |X|_{0,2}$ defined above.

\begin{lemma}\label{lem:gagliardo_nirenberg}
	%Let $\gamma\colon \R\to\R^n$ be an immersion and let $\zeta$ be as in \eqref{eq:def_zeta}. 
	Let $r\geq 0$ and $p\in[2,\infty]$.
	There exists $C=C(\Lambda,p,r)>0$ such that for any smooth normal vector field $X$ along $\gamma$ we have
	\begin{align}
		\Big(\int |X|^p \zeta^{p(r+\theta)}ds\Big)^{\frac{1}{p}} &\leq C \Big(\int|\nabla_sX|^2\zeta^{2r+2}ds \Big)^{\frac{\theta}{2}} \Big(\int|X|^2 \zeta^{2r} ds\Big)^{\frac{1-\theta}{2}}\\
		&\qquad + C\Big( \int|X|^2\zeta^{2r} ds\Big)^{\frac{1}{2}}\label{eq:gagliardo-nirenberg}
	\end{align}
    for $p\in[2,\infty)$,
	where $\theta=\frac{p-2}{2p}\in[0,\frac{1}{2})$, and, for $p=\infty$,
	\begin{align}
		\Vert \zeta^{r+\frac{1}{2}}X \Vert_{\infty}\leq C \Big(\int|\nabla_sX|^2\zeta^{2r+2}ds \Big)^{\frac{1}{4}} \Big(\int|X|^2 \zeta^{2r} ds\Big)^{\frac{1}{4}} + C\Big( \int|X|^2\zeta^{2r} ds\Big)^{\frac{1}{2}}.
	\end{align}
\end{lemma}

\begin{remark}
	Defining, for $p\in[2,\infty]$, the seminorm
	\begin{align}
		|X|_{0,p} &\vcentcolon= \Vert \zeta^{\frac{1}{2}-\frac{1}{p}} X \Vert_{L^p([\zeta>0],ds)} \Big( = \Vert \zeta^{\theta} X \Vert_{L^p([\zeta>0],ds)} \Big),
	\end{align}
	with the interpretation $|X|_{0,\infty} \vcentcolon= \Vert \zeta^{\frac{1}{2}} X \Vert_{\infty}$,
	\Cref{lem:gagliardo_nirenberg} resembles the classical Gagliardo--Nirenberg inequality, since \eqref{eq:gagliardo-nirenberg} yields
	\begin{align}
		|\zeta^r X |_{0,p}\leq C |\zeta^r X|_{1,2}^{\theta} |\zeta^r X|_{0,2}^{1-\theta}. 
	\end{align}
\end{remark}

By taking $\gamma$ in \Cref{lem:gagliardo_nirenberg} to be an arclength parametrization of a straight line in $\R^2$,
we also obtain a non-geometric version for functions on the real line.
Although we will not use it here, we include the statement for the reader's convenience.

\begin{corollary}
	Let $\xi\in C^\infty_c(\R)$ with $\xi\geq0$ and $|\xi'|\leq\Lambda$. Let $r\geq 0$ and $p\in[2,\infty]$. 
	There exists $C=C(\Lambda, p,r)>0$ such that for all $u\in W^{1,2}_{\mathrm{loc}}(\R)$ we have
	\begin{align}
		\Big(\int |u|^p \xi^{p(r+\theta)}dx\Big)^{\frac{1}{p}} &\leq C \Big(\int|u'|^2\xi^{2r+2}dx \Big)^{\frac{\theta}{2}} \Big(\int|u|^2 \xi^{2r} dx\Big)^{\frac{1-\theta}{2}}\\
		&\qquad + C\Big( \int|u|^2\xi^{2r} dx\Big)^{\frac{1}{2}}
	\end{align}
    for $p\in [2,\infty)$,
	where $\theta=\frac{p-2}{2p}$, and, for $p=\infty$,
	\begin{align}
		\Vert \xi^{r+\frac{1}{2}}u\Vert_{\infty}&\leq C \Big(\int|u'|^2\xi^{2r+2}dx \Big)^{\frac{1}{4}} \Big(\int|u|^2 \xi^{2r} dx\Big)^{\frac{1}{4}} + C\Big( \int|u|^2\xi^{2r} dx\Big)^{\frac{1}{2}}.
	\end{align}
\end{corollary}

\begin{proof}[Proof of \Cref{lem:gagliardo_nirenberg}]
	If $p=2$, there is nothing to prove. Let $p>2$. We have $|\partial_s|X||\leq |\nabla_s X|$ a.e.\ since $X$ is normal. Using that $\zeta$ has compact support, we have
	\begin{align}
		\Vert |X|^2\zeta^{2r+1}\Vert_{\infty}
        &\leq \int \big(2|X||\nabla_s X| \zeta^{2r+1} + (2r+1)|X|^2\zeta^{2r}|\partial_s\zeta|\big)ds \\
		&\leq 2\Big(\int|X|^2\zeta^{2r}ds\Big)^{\frac{1}{2}} \Big(\int|\nabla_s X|^2\zeta^{2r+2}ds\Big)^{\frac{1}{2}} + (2r+1)\Lambda \int|X|^2 \zeta^{2r}ds. \label{eq:gagliardo-nirenberg_L_infty}
	\end{align}
	This yields the statement for $p=\infty$. For $p\in(2,\infty)$, we further estimate
	\begin{align}
		&\int |X|^p \zeta^{p(r+\frac{1}{2})-1}ds \\ 
		&\quad \leq  \Vert |X|^{p-2} \zeta^{(r+\frac{1}{2})(p-2)}\Vert_{\infty} \int|X|^2 \zeta^{2r}ds \\
		&\quad = \Vert |X|^{2} \zeta^{2(r+\frac{1}{2})}\Vert_{\infty}^{\frac{p-2}{2}} \int|X|^2 \zeta^{2r}ds\\
		&\quad \leq  C \Big(\int|X|^2\zeta^{2r}ds\Big)^{\frac{p+2}{4}} \Big(\int|\nabla_s X|^2\zeta^{2r+2}ds\Big)^{\frac{p-2}{4}} + C\Big(\int|X|^2 \zeta^{2r}ds\Big)^{\frac{p}{2}}.
	\end{align}
	The statement follows.
\end{proof}

Following the notation in \cite{DallAcqua-Pozzi_2014_Willmore-Helfrich} (cf.\ \cite{Dziuk-Kuwert-Schatzle_2002}) for $a,b,c\in \N_0$, we denote by $P^{a,c}_b$ any linear combination (with universal coefficients) of terms of the form
\begin{align}
	\nabla_s^{i_1}\kappa* \dots * \nabla_s^{i_b}\kappa,
\end{align}
that may be scalar or normal vector fields along $\gamma$. Here $i_j\leq c$ for all $j=1,\dots,b$ and $\sum_{j=1}^b i_j=a$. %We also write $P^{a}_b$ if we need not keep track of $\max_{j=1,\dots,b}i_j$.

\begin{proposition}\label{prop:interpolation_P}
	%Let $\gamma\colon\R\to\R^n$ be an immersion and let $\zeta$ be as in \eqref{eq:def_zeta}. 
	Let $k\in \N$, $a,b,c\in \N_0$, $b\geq 2$, and suppose $c\leq k$ and $a+\frac{b}{2}-1<2k$. %Suppose $\int_{[\zeta>0]}|\kappa|^2ds \leq M<\infty$.	
	Then, for any $\varepsilon>0$, there exists $C=C(\varepsilon,\Lambda,a,b, k)$ such that
    \begin{align}
		\int |P^{a,c}_b| \zeta^{a+\frac{b}{2}-1}ds  &\leq \varepsilon \int |\nabla_s^k \kappa|^2 \zeta^{2k}ds + C\Big(\int_{[\zeta>0]}|\kappa|^2 ds\Big)^{\frac{b-\delta}{2-\delta}} + C\Big(\int_{[\zeta>0]}|\kappa|^2 ds\Big)^{\frac{b}{2}}
	\end{align}
    with $\delta = \frac{1}{k}(a+\frac{b}{2}-1)<2$.
	% \begin{align}
	% 	\int |P^{a,c}_b(\kappa)| \zeta^{a+\frac{b}{2}-1}ds \leq \varepsilon \int |\nabla_s^k \kappa|^2 \zeta^{2k}ds + C(\varepsilon,\Lambda,a,b, k,M).
	% \end{align}
\end{proposition}

\begin{proof}
	If suffices to consider the case $P^{a,c}_b=	\nabla_s^{i_1}\kappa* \dots * \nabla_s^{i_b}\kappa$. Again, we denote by $C$ a constant that is allowed to vary from line to line, depending only on 
	$\varepsilon,\Lambda,a,b, k$.  
    
    We first prove
    \begin{align}\label{eq:P_interpolation_multiplicative}
		\int |P^{a,c}_b| \zeta^{a+\frac{b}{2}-1}ds \leq C |\kappa|_{k,2}^{\delta}|\kappa|_{0,2}^{b-\delta} + C |\kappa|_{k,2}^{\frac{a}{k}}|\kappa|_{0,2}^{b-\frac{a}{k}}.
	\end{align}
	We may assume $|\kappa|_{0,2}>0$, otherwise the statement is trivial. 
    
    For proving \eqref{eq:P_interpolation_multiplicative}, we first assume $c\leq k-1$. Using H\"older's inequality and $a+\frac{b}{2}-1 =\sum_{j=1}^b (i_j+\frac{1}{2}-\frac{1}{b})$ we obtain
	\begin{align}
		\int |\nabla_s^{i_1}\kappa *\dots*\nabla_s^{i_b}\kappa|\zeta^{a+\frac{b}{2}-1}ds &\leq \prod_{j=1}^b \Vert  \zeta^{i_j+\frac{1}{2}-\frac{1}{b}}\nabla_s^{i_j}\kappa\Vert_{L^b(ds)}.\label{eq:P_interpolation_1}
	\end{align}
	We now estimate the factors in \eqref{eq:P_interpolation_1} by using the Gagliardo--Nirenberg-type inequality. Indeed,  \Cref{lem:gagliardo_nirenberg} with $r=i_j$ yields
	\begin{align}
		&\Big(\int  |\nabla^{i_j}_s \kappa|^{b}\zeta^{bi_j+\frac{b}{2}-1}ds\Big)^{\frac{1}{b}} \\
		&\leq C \Big(\int|\nabla^{i_j+1}_s\kappa|^2\zeta^{2i_j+2}\Big)^{\frac{\theta}{2}} \Big(\int|\nabla^{i_j}_s\kappa|^2\zeta^{2i_j}ds\Big)^{\frac{1-\theta}{2}} + C \Big(\int|\nabla^{i_j}_s\kappa|^2\zeta^{2i_j}ds\Big)^{\frac{1}{2}},
	\end{align}
	where $\theta = \frac{b-2}{2b}$. Applying \Cref{lem:interpolation_multiplicative} and combining exponents, we find
	\begin{align}
		\Big(\int  |\nabla^{i_j}_s \kappa|^{b}\zeta^{bi_j+\frac{b}{2}-1}ds\Big)^{\frac{1}{b}}
		%		&\leq C |\kappa|_{k,2}^{\frac{i_j+1}{k} \theta} |\kappa|_{0,2}^{\frac{k-(i_j+1)}{k}\theta} |\kappa|_{k,2}^{\frac{i_j}{k}(1-\theta)} |\kappa|_{0,2}^{\frac{k-i_j}{k}(1-\theta)}
		%		+ C |\kappa|_{k,2}^{\frac{i_j}{k}}|\kappa|_{0,2}^{\frac{k-i_j}{k}} \\
		&\leq  C |\kappa|_{k,2}^{\frac{i_j+\theta}{k}} |\kappa|_{0,2}^{\frac{k-i_j-\theta}{k}}
		+ C |\kappa|_{k,2}^{\frac{i_j}{k}}|\kappa|_{0,2}^{\frac{k-i_j}{k}} \\
		&= C|\kappa|_{k,2}^{\frac{i_j}{k}}|\kappa|_{0,2}^{\frac{k-i_j}{k}}\Big( |\kappa|_{k,2}^{\frac{\theta}{k}}|\kappa|_{0,2}^{-\frac{\theta}{k}}+1\Big),
	\end{align}
	using that $|\kappa|_{0,2}>0$.
	Plugging this into \eqref{eq:P_interpolation_1}, we obtain
	\begin{align}
		\int |\nabla_s^{i_1}\kappa *\dots*\nabla_s^{i_b}\kappa|\zeta^{a+b}ds
		&\leq C \Big(\prod_{j=1}^b |\kappa|_{k,2}^{\frac{i_j}{k}}|\kappa|_{0,2}^{\frac{k-i_j}{k}}\Big) \Big(|\kappa|_{k,2}^{\frac{\theta}{k}}|\kappa|_{0,2}^{-\frac{\theta}{k}}+1\Big)^b\\
		& \leq C |\kappa|_{k,2}^{\frac{a}{k}} |\kappa|_{0,2}^{\frac{bk-a}{k}} \Big(|\kappa|_{k,2}^{\frac{b\theta}{k}}|\kappa|_{0,2}^{-\frac{b\theta}{k}}+1\Big)\\
		& = C |\kappa|_{k,2}^{\frac{a+b\theta}{k}} |\kappa|_{0,2}^{\frac{bk-a - b\theta}{k}} +C |\kappa|_{k,2}^{\frac{a}{k}} |\kappa|_{0,2}^{\frac{bk-a}{k}}.
	\end{align}
	Noting that $\frac{1}{k}(a+b\theta) = \frac{1}{k}(a+\frac{b}{2}-1) = \delta$, we obtain \eqref{eq:P_interpolation_1} and thus \eqref{eq:P_interpolation_multiplicative} in the case $c\leq k-1$.
    
    On the other hand, suppose $c=k$. Then $P^{a,c}_b$ contains a term of the form $\nabla_s^k\kappa$, so $k\leq a \leq a+\frac{b}{2}-1 <2k$ (as $b\geq 2$), and we may write $
        P_b^{a,c} = P_{b-1}^{a-k,k-1} * \nabla_s^k\kappa$. 
        The Cauchy--Schwarz inequality yields
        \begin{align}
            \int|P_b^{a,c}\zeta^{a+\frac{b}{2}-1}|ds \leq \Big(\int|\nabla_s^k\kappa|^2 \zeta^{2k}ds\Big)^{\frac{1}{2}} \Big(\int|P_{\tilde{b}}^{\tilde{a}, k-1}| \zeta^{\tilde{a}+ \frac{\tilde{b}}{2}-1} ds \Big)^{\frac{1}{2}},
        \end{align}
        where $\tilde{a} \vcentcolon = 2(a-k), \tilde{b}\vcentcolon =2(b-1)$. Since $\tilde{a}+\frac{\tilde{b}}{2}-1 = 2(a+\frac{b}{2}-1-k)<2k$ and as we have already proven \eqref{eq:P_interpolation_multiplicative} in the case $c\leq k-1$, we may apply \eqref{eq:P_interpolation_multiplicative} to the second factor, yielding
        \begin{align}
            \int|P_b^{a,c}\zeta^{a+\frac{b}{2}-1}|ds \leq C |\kappa|_{k,2} \Big(|\kappa|_{k,2}^{\tilde\delta}|\kappa|_{0,2}^{\tilde b-\tilde \delta} + C |\kappa|_{k,2}^{\frac{\tilde a}{k}}|\kappa|_{0,2}^{\tilde b-\frac{\tilde a}{k}}\Big)^{\frac{1}{2}}.
        \end{align}
        Since $\delta = 1+\frac{\tilde\delta}{2}$, $\frac{1}{2}(\tilde b-\tilde \delta)=b-\delta$, $1+\frac{\tilde a}{2k} =\frac{a}{k}$, and $\frac{1}{2}(\tilde b-\frac{\tilde a}{k})=b-\frac{a}{k}$, we have thus proven \eqref{eq:P_interpolation_multiplicative} also in the case $c=k$.

    Lastly, we estimate the first term of the right hand side of \eqref{eq:P_interpolation_multiplicative} using Young's inequality and $\delta\in [0,2)$, yielding
    \begin{align}
        |\kappa|_{k,2}^{\delta}|\kappa|_{0,2}^{b-\delta} &\leq C \Big(\int |\nabla_s^k \kappa|^2 \zeta^{2k} ds\Big)^{\frac{\delta}{2}} \Big(\int_{[\zeta>0]}|\kappa|^2 ds\Big)^{\frac{b-\delta}{2}} + C\Big(\int_{[\zeta>0]}|\kappa|^2 ds\Big)^{\frac{b}{2}} \\
        &\leq \varepsilon \int |\nabla_s^k \kappa|^2 \zeta^{2k} ds + C \Big(\int_{[\zeta>0]}|\kappa|^2 ds\Big)^{\frac{b-\delta}{2-\delta}} + C\Big(\int_{[\zeta>0]}|\kappa|^2 ds\Big)^{\frac{b}{2}}.
    \end{align}
    Using the same argument with $\delta$ replaced by $\frac{a}{k}$, we additionally estimate the second term in \eqref{eq:P_interpolation_multiplicative} and get the term
    \begin{align}
        \Big(\int_{[\zeta>0]}|\kappa|^2 ds\Big)^{\frac{b-\frac{a}{k}}{2-\frac{a}{k}}} \leq \begin{cases}
             \Big(\int_{[\zeta>0]}|\kappa|^2 ds\Big)^{\frac{b}{2}}, & \text{ if }\int_{[\zeta>0]}|\kappa|^2 ds\leq 1,\\
             \Big(\int_{[\zeta>0]}|\kappa|^2 ds\Big)^{\frac{b-\delta}{2-\delta}}, & \text{ if }\int_{[\zeta>0]}|\kappa|^2 ds\geq 1,
        \end{cases}
    \end{align}
    where we used that $x\mapsto\frac{b-x}{2-x}$ is monotonically increasing and the estimate $0\leq \frac{a}{k}\leq \delta$ in the last step.
    Thus, \eqref{eq:P_interpolation_multiplicative} implies the statement.
	% with $\delta = \frac{1}{k}(a+\frac{b}{2}-1).$ By the assumptions, $\delta<2$ and $\frac{a}{k}<2$, so the lemma then follows from Young's inequality.
\end{proof}

\section{Global curvature control}\label{sec:curvature_control}

In this section we prove global curvature estimates for a flow with uniformly bounded bending energy $B$. Having established the key interpolation estimate, \Cref{prop:interpolation_P}, we now consider a rather general evolution law and present a unified approach to treat all the flows we consider in this work simultaneously by a combination of the strategies in \cite{Dziuk-Kuwert-Schatzle_2002,KSSI}.

\subsection{Evolution of localized curvature integrals}

We consider a general class of geometric flows in this section, including \eqref{eq:CSF}, \eqref{eq:SDF}, \eqref{eq:CF}, \eqref{eq:lambda-EF}, i.e., we assume that $\gamma\colon[0,T)\times\R\to\R^n$ is a smooth family of \emph{proper} immersions, satisfying 
\begin{align}\label{eq:LTE_general_flow}
	\partial_t \gamma = -\sigma\nabla_s^2\kappa + \lambda\kappa + \mu|\kappa|^2\kappa +\vartheta \langle \kappa, \nabla_s\kappa\rangle\partial_s\gamma.
\end{align}
Here $\lambda, \mu,\sigma,\vartheta\in\R$ are fixed parameters. In the following key lemma, in abuse of notation, we write, for $\alpha,\beta\in\R$,
\begin{align}
    (\alpha+\beta)(P_b^{a,c}+P_{b'}^{a',c'}) \vcentcolon = \alpha P_b^{a,c} + \beta P_b^{a,c} + \alpha P_{b'}^{a',c'}+\beta P_{b'}^{a',c'}.
\end{align}

\begin{lemma}\label{lem:LTE_curvature_evolution}
Let $\tilde{\zeta}\in C_c^\infty(\R^n)$ and set $\zeta(t,\cdot)\vcentcolon =\tilde\zeta\circ\gamma(t,\cdot)$. For all $m\in\N_0$ and $r\in\N$ with $r\geq 2$, we have
    \begin{align}
		&\frac{d}{dt}\frac{1}{2}\int|\nabla^{m}_s\kappa|^2 \zeta^{2m+r} ds +\sigma \int|\nabla_s^{m+2}\kappa|^2\zeta^{2m+r}ds +\lambda \int|\nabla_s^{m+1}\kappa|^2\zeta^{2m+r}ds  \\
		&= \int \Big( (\sigma+\mu+\vartheta)(P^{2m+2,m+2}_4+ P_6^{2m,m})+ \lambda P_4^{2m,m}\Big)\zeta^{2m+r}ds \\
        %&\quad + (2m+r)\int \langle \sigma P^{2m+2,m+2}_3+\lambda P_3^{2m,m} + \mu P_5^{2m,m}+\vartheta P_4^{2m+1,m+1}\partial_s\gamma, D\tilde\zeta\circ\gamma\rangle\zeta^{2m+r-1}ds \\
         &\quad + (2m+r)\int \langle \sigma P^{2m+2,m+2}_3+\lambda P_3^{2m,m}, D\tilde\zeta\circ\gamma\rangle\zeta^{2m+r-1}ds \\
         &\quad + (2m+r)\int \langle \mu P_5^{2m,m}+\vartheta P_4^{2m+1,m+1}\partial_s\gamma, D\tilde\zeta\circ\gamma\rangle\zeta^{2m+r-1}ds \\
        % &\quad + \int \Big(\sigma P^{2m+2,m+2}_4+\lambda P_4^{2m,m} + \mu P_6^{2m,m} + \vartheta P_4^{2m+2,m+2}\Big)\zeta^{2m+r}ds\\
		% &\quad + \int\langle \sigma P^{2m+2,\max\{m,2\}}_3(\kappa)+\lambda P^{2m,m}_3(\kappa),D\tilde{\zeta}\circ\gamma\rangle\zeta^{2m+r-1}ds \\
		% &\quad +\mu \int \langle P^{2m,m}_5(\kappa)+ P^{2m+1,\max\{m,1\}}_4(\kappa)\partial_s\gamma,D\tilde{\zeta}\circ\gamma\rangle\zeta^{2m+r-1}ds\\
		&\quad - 2\sigma (2m+r)\int\langle\nabla_s^{m+2}\kappa,\nabla_s^{m+1}\kappa\rangle \zeta^{2m+r-1}\partial_s\zeta ds\\
		&\quad - \sigma (2m+r) \int\langle \nabla_s^{m+2}\kappa, \nabla_s^m\kappa\rangle \zeta^{2m+r-2}\big((2m+r-1)(\partial_s\zeta)^2+\zeta \partial_s^2\zeta\big)ds\\
        &\quad - \lambda (2m+r) \int\langle \nabla_s^{m+1}\kappa,\nabla_s^m\kappa\rangle \zeta^{2m+r-1}\partial_s\zeta ds.
	\end{align}
	% \begin{align}
	% 	&\frac{d}{dt} \int|\nabla_s^m\kappa|^2 \zeta^{2m+4}ds + \int|\nabla_s^m\kappa|^2\zeta^{2m+4}ds \\
 %        &\quad + \sigma \int|\nabla_s^{m+2}\kappa|^2\zeta^{2m+4}ds + \lambda \int|\nabla_s^{m+1}\kappa|^2\zeta^{2m+4} ds
 %        \leq C(\lambda,\mu,\sigma,\Lambda,m,M)
	% \end{align}
	% for all $m\in \N_0$.
\end{lemma}

\begin{proof}
We first show that the derivatives of the curvature satisfy
	\begin{align}
		&\partial_t^{\perp} \nabla_s^m \kappa + \sigma \nabla_s^{m+4}\kappa - \lambda \nabla_s^{m+2}\kappa  \\
        &\qquad = (\sigma+\mu+\vartheta)(P^{m+2,m+2}_3+ P_5^{m,m})+ \lambda P_3^{m,m}.\label{eq:energy_estimates_0}
	\end{align}
Indeed, for $m=0$, this follows by \eqref{eq:dtkappa}, where $V\vcentcolon=\partial_t^{\perp}\gamma=-\sigma\nabla_s^2 \kappa + \lambda \kappa + \mu P_3^{0,0}$ and $\xi\vcentcolon = \langle\partial_t\gamma,\partial_s\gamma\rangle = \vartheta P_2^{1,1}$.  For $m\geq 1$, recall from \cite[(2.8)]{Dziuk-Kuwert-Schatzle_2002} that we have
	\begin{align}
		(\partial_t^{\perp}\nabla_s - \nabla_s\partial_t^{\perp})\nabla_s^{m}\kappa =
        \kappa*V*\nabla_s^{m+1}\kappa + \partial_s\xi \nabla_s^{m+1}\kappa + \kappa*\nabla_s^{m}\kappa*\nabla_sV,
        %\big(\langle \kappa, V\rangle - \partial_s\xi\big)\nabla_s^{m+1}\kappa + \kappa*\nabla_s^{m}\kappa*\nabla_sV,
	\end{align}
    so \eqref{eq:energy_estimates_0} follows by induction on $m\in\N_0$, using $\partial_s\xi =\vartheta P^{2,2}_2$.
    
    We now compute the time derivative of the localized integral. Note that since $\gamma(t)$ is proper by assumption, $\zeta(t,\cdot)$ has compact support for all $t\in[0,T)$.  
    By \eqref{eq:derivative_formula_1}, we have
    \begin{align}
        \frac{d}{dt}\frac{1}{2}\int|\nabla_s^{m}\kappa|^2\zeta^{2m+r}ds
        &= \int \langle \nabla_s^m\kappa,\partial_t^\perp\nabla_s^m\kappa \rangle \zeta^{2m+r}ds \\
        & \quad + \frac{2m+r}{2}\int |\nabla_s^m\kappa|^2\zeta^{2m+r-1}\langle D\tilde{\zeta}\circ\gamma,\partial_t\gamma \rangle ds \\
        & \quad + \frac12 \int |\nabla_s^m\kappa|^2\zeta^{2m+r} (\partial_s\xi-\langle \kappa, V\rangle) ds.\label{eq:0402-01}
    \end{align}
	By \eqref{eq:energy_estimates_0}, the leading order term arising from the first term satisfies
	\begin{align}
		&\int\langle \nabla_s^{m+4}\kappa, \nabla_s^m\kappa\rangle \zeta^{2m+r}ds\\ &= - \int\langle \nabla_s^{m+3}\kappa, \nabla_s^{m+1}\kappa\rangle \zeta^{2m+r}ds - (2m+r) \int\langle \nabla_s^{m+3}\kappa, \nabla_s^m\kappa\rangle \zeta^{2m+r-1}\partial_s\zeta ds \\
		& = \int|\nabla_s^{m+2}\kappa|^2\zeta^{2m+r}ds  + 2(2m+r) \int\langle\nabla_s^{m+2}\kappa,\nabla_s^{m+1}\kappa\rangle \zeta^{2m+r-1}\partial_s\zeta ds\\
		&\quad + (2m+r) \int\langle \nabla_s^{m+2}\kappa, \nabla_s^m\kappa\rangle \zeta^{2m+r-2}\big((2m+r-1)(\partial_s\zeta)^2+\zeta \partial_s^2\zeta\big)ds,\label{eq:energy_estimates_1}
	\end{align}
	using integration by parts twice. Similarly, by integration by parts
    \begin{align}
        \int \langle \nabla^{m+2}_s\kappa,\nabla_s^m\kappa\rangle \zeta^{2m+r}ds & = -\int|\nabla_s^{m+1}\kappa|^2\zeta^{2m+r}ds \\
        &\quad - (2m+r) \int\langle\nabla_s^{m+1}\kappa, \nabla_s^m\kappa\rangle \zeta^{2m+r-1}\partial_s\zeta ds.\label{eq:0402-02}
    \end{align}
	Moreover, for the second term we have
	\begin{align}
		&\int|\nabla_s^m\kappa|^2 \zeta^{2m+r-1} \langle D\tilde{\zeta}\circ \gamma, \partial_t\gamma\rangle ds  \\
		&= \int \langle  \sigma P^{2m+2,\max\{m,2\}}_3+\lambda P^{2m,m}_3,  D\tilde{\zeta}\circ\gamma \rangle \zeta^{2m+r-1}ds\\
		&\quad + \int \langle \mu P^{2m,m}_5 + \vartheta P^{2m+1,\max\{m,1\}}_4\partial_s\gamma , D\tilde{\zeta}\circ\gamma \rangle \zeta^{2m+r-1}ds.\label{eq:energy_estimates_2}
	\end{align}
	Using that $\partial_s\xi = P_2^{2,2}$, the third term is
	\begin{align}
		&\int|\nabla_s^m\kappa|^2\zeta^{2m+r}(\partial_s\xi - \langle \kappa,V\rangle) ds \\
		&\quad = \int \Big(\sigma P^{2m+2,\max\{m,2\}}_4 + \lambda P_4^{2m,m} + \mu  P^{2m,m}_6 + \vartheta P^{2m+2,\max\{m,2\}}_4\Big)\zeta^{2m+r} ds.\label{eq:energy_estimates_3}
	\end{align}
	Inserting \eqref{eq:energy_estimates_0} into \eqref{eq:0402-01}, using \eqref{eq:energy_estimates_1}, \eqref{eq:0402-02}, \eqref{eq:energy_estimates_2}, \eqref{eq:energy_estimates_3}, and noting that the terms of the form $P_b^{a,\max\{c,c'\}}$ can also be regarded as $P_b^{a,c+c'}$, we conclude the desired identity.
\end{proof}

\subsection{Curvature control for fourth order flows}

We now fix $\tilde{\zeta}\in C_c^\infty(\R^n)$ with $0\leq \tilde{\zeta}\leq 1$, $|D\tilde{\zeta}|\leq \Lambda$, and $|D^2\tilde{\zeta}|\leq \Lambda^2$. If $\gamma\in C^\infty(\R;\R^n)$ is a proper immersion, then $\zeta\vcentcolon = \tilde{\zeta}\circ\gamma\in C_c^\infty(\R)$ satisfies 
\begin{align}\label{eq:ds2_zeta}
	\zeta \in C_c^\infty(\R),\ 0\leq\zeta\leq 1,\ |\partial_s\zeta|\leq \Lambda, \text{ and } |\partial_s^2\zeta|\leq \Lambda^2+|\kappa|\Lambda.
\end{align}
In particular, $\zeta$ satisfies the assumptions in \eqref{eq:def_zeta}, so that \Cref{prop:interpolation_P} is applicable.
Then for the flow in \eqref{eq:LTE_general_flow} we obtain the following key energy estimate.

% Since we consider $\lambda,\mu,\sigma$ to be fixed parameters, we will regard them as universal in the sequel. In particular, we are not keeping track how constants in our estimates or coefficients in nonlinear expressions of the curvature (like $P^{a,c}_b(\kappa)$ or the $*$-notation) depend on $\lambda,\mu,\sigma$.

\begin{lemma}\label{lem:LTE_gronwall_4th_order}
    Let 
    %$\gamma$ be a smooth solution to \eqref{eq:LTE_general_flow} with 
    $\sigma>0$. Suppose that $B[\gamma(t)]\leq M$ holds at some time $t\in [0,T)$. Then, at the time $t$,
    \begin{align}
		&\frac{d}{dt} \int|\nabla_s^m\kappa|^2 \zeta^{2m+4}ds + \int|\nabla_s^m\kappa|^2\zeta^{2m+4}ds + \sigma \int|\nabla_s^{m+2}\kappa|^2\zeta^{2m+4}ds
        \leq C
	\end{align}
	holds for all $m\in \N_0$ where $C= C(\lambda,\mu,\sigma,\vartheta,\Lambda,m,M)$.
\end{lemma}
\begin{proof}
    % We estimate the terms arising in \Cref{lem:LTE_curvature_evolution} for $r=4$.  In the following, we are not keeping track how the coefficients in nonlinear expressions of the curvature (like $P^{a,c}_b$) depend on the parameters $\lambda,\mu,\sigma,\vartheta, \Lambda,m$.
    % %In the following, $C=C(\lambda,\mu,\sigma,\Lambda,m,M)\in(0,\infty)$ denotes a constant that is allowed to change from line to line.

    We estimate the terms on the right hand side of the identity in \Cref{lem:LTE_curvature_evolution} for $r=4$. Using \eqref{eq:ds2_zeta} for the derivatives of $\zeta$, there is $C=C(\lambda,\mu,\sigma,\vartheta,\Lambda,m)$  such that
	\begin{align}
		&\frac{d}{dt} \frac{1}{2}\int|\nabla^{m}_s\kappa|^2 \zeta^{2m+4} ds +\sigma \int|\nabla_s^{m+2}\kappa|^2\zeta^{2m+4}ds \leq C \sum \int |P^{a,c}_b|\zeta^q ds, \label{eq:energy_estimates_5}
		% &\leq \int \big[|P^{2m+2,m+1}_4(\kappa)|+ |P^{2m,m}_4(\kappa)|+|P^{2m,m}_6(\kappa)| +| P^{2m+1,m+1}_4(\kappa)|\big]\zeta^{2m+4}ds\\
		% &\quad +C\int\big[ |P^{2m+2,\max\{m,2\}}_3(\kappa)|+|P^{2m+1, \max\{m,1\}}_4(\kappa)|\big]\zeta^{2m+3}ds\\
		% &\quad +C\int\big[|P^{2m,m}_5(\kappa)|+|P^{2m,m}_3(\kappa)|\big]
		% \zeta^{2m+3}ds\\
		% &\quad +C \int|\nabla_s^{m+1}\kappa|^2\zeta^{2m+2}ds+C \int|\nabla_s^m\kappa|^2\zeta^{2m}ds +C \int|\nabla_s^m\kappa|^2|\kappa|^2\zeta^{2m+2}ds. 
	\end{align}
    where the sum runs over the finitely many terms with $c\leq m+2$ and $a+\frac{b}{2}-1<2(m+2)$ as well as $q\geq a+\frac{b}{2}-1$. Noting also $0\leq\zeta\leq1$, which yields $0\leq \zeta^q\leq \zeta^{a+\frac{b}{2}-1}$, we may thus apply \Cref{prop:interpolation_P} with $k=m+2$ to conclude that
	\begin{align}
		\frac{d}{dt} \frac{1}{2}\int|\nabla^{m}_s\kappa|^2 \zeta^{2m+4} ds +\frac{2\sigma}{3}\int|\nabla_s^{m+2}\kappa|^2\zeta^{2m+4}ds \leq C,
	\end{align}
    where $C= C(\lambda,\mu,\sigma,\vartheta,\Lambda,m,M)$.
	Since \Cref{prop:interpolation_P} also applies to
	\begin{align}
		\int|\nabla_s^m\kappa|^2\zeta^{2m+4}ds = \int P^{2m,m}_2 \zeta^{2m+4}ds,
	\end{align}
	the statement follows.	
	% In the particular case $m=0$, we need to treat the term 
	% \begin{align}\label{eq:energy_estimates_6}
	% 	\int |P^{2m+2,\max\{m,2\}}_3(\kappa)|\zeta^{2m+3}ds = \int |P^{2,2}_3(\kappa)|\zeta^{3}ds
	% \end{align}	
	% in \eqref{eq:energy_estimates_5} separately. If it contains a term of the form $\int |\nabla_s^{2}\kappa * \kappa*\kappa|^2\zeta^{3}ds$, we estimate, as in \eqref{eq:energy_estimates_4.5} above (recalling $m=0$),
	% \begin{align}
	% 	\int |\nabla_s^{2}\kappa * \kappa*\kappa|^2\zeta^{3}ds \leq \varepsilon\int|\nabla_s^{2}\kappa|^2 \zeta^{4}ds + \frac{1}{\varepsilon} \int |\kappa|^4\zeta^2ds.\label{eq:energy_estimates_7}
	% \end{align}
	% The first term may be absorbed into the left hand side of \eqref{eq:energy_estimates_5}, whereas the second term already appears on the last term of the right hand side of \eqref{eq:energy_estimates_5}. Hence, for $m=0$, we may replace $\int P^{2,2}_3(\kappa)\zeta^{3}ds $ by $\int P^{2,1}_3(\kappa)\zeta^3ds$ in \eqref{eq:energy_estimates_5}. After this substitution, one may proceed as in the case $m\geq 1$, verifying that also for $m=0$ all terms on the right hand side of \eqref{eq:energy_estimates_5} satisfy the assumptions of \Cref{prop:interpolation_P}.
\end{proof}

While a smooth initial datum with finite bending energy $B$ only needs to satisfy $\kappa\in L^2(ds)$, following the argument of \cite[Theorem 3.5]{KSSI} we obtain instantaneous higher Sobolev integrability for positive time.

\begin{lemma}\label{lem:LTE_time_cutoff}
    Let %$\gamma$ be a smooth solution to \eqref{eq:LTE_general_flow} with 
    $\sigma>0$ and let $T\leq T^*<\infty$.
    Suppose that $B[\gamma(t)] \leq M$ for all $t\in [0,T)$.
    Then, for all $m\in\N$ and $t\in(0,T)$,
	\begin{align}
		\int |\nabla^m_s\kappa|^2 \zeta^{2m+4} ds \leq \frac{C(\lambda,\mu, \sigma, \vartheta,\Lambda,m,M,T^*)}{t^{\frac{m}{2}}}.
	\end{align}
\end{lemma}

\begin{proof}
	Fix any $t_0\in(0,T)$. Define Lipschitz cutoff functions in time $\xi_j:\R\to[0,1]$ via
	\begin{align}\label{eq:def_xi_j}
		\xi_j(t) \vcentcolon= \left\lbrace\begin{array}{ll}
			0,& \text{for } t\leq (j-1)\frac{t_0}{m}, \\
			\frac{m}{t_0}\left(t-(j-1)\frac{t_0}{m}\right),& \text{for } (j-1)\frac{t_0}{m}\leq t\leq j\frac{t_0}{m}, \\
			1, &\text{for } t\geq j\frac{t_0}{m},
		\end{array}\right.
	\end{align}
    where $0\leq j \leq m$.
    %, and $\xi_0(t)\vcentcolon = 1$ for all $t\in \R$ if $m=0$.
    We also define $\xi_{-1}(t)\vcentcolon= 0$.
    We note that
	\begin{align}\label{eq:xi'bound}
		0\leq  \frac{d}{dt}{\xi}_j \leq \frac{m}{t_0}\xi_{j-1},\quad \text{for all $0\leq j\leq m$ and a.e.\ $t\geq 0$.}
	\end{align}
	We further define
	\begin{align}
		E_j(t)\vcentcolon = \int|\nabla_s^{2j}\kappa|^2\zeta^{4j+4}ds, \quad e_j(t)\vcentcolon= \xi_j(t)E_j(t).
	\end{align}
    Now \Cref{lem:LTE_gronwall_4th_order}, \eqref{eq:xi'bound}, and $0\leq \zeta\leq 1$ imply that for $0\leq j\leq m$ and a.e.\ $t\in [0,T)$, 
	\begin{align}
		\frac{d}{dt} e_j + \sigma \xi_j E_{j+1}\leq \frac{m}{t_0}\xi_{j-1}E_j + C(\lambda, \mu, \sigma,\vartheta,\Lambda,j,M).\label{eq:dt_ej}
	\end{align}
	We now show that for all $0\leq j\leq m$ and $t\in [0,t_0]$ we have
	\begin{align}
		e_j(t) + \sigma \int_0^t \xi_j E_{j+1}d\tau \leq \frac{C(\lambda,\mu, \sigma,\vartheta,\Lambda,j,m,M,T^*)}{t_0^{j}}.\label{eq:LTE_induction}
	\end{align}
	We prove \eqref{eq:LTE_induction} by induction on $j$. For $j=0$, \eqref{eq:dt_ej} yields
	\begin{align}
		e_0(t) + \sigma \int_0^t \xi_0E_1d\tau &\leq C(\lambda,\mu,\sigma,\vartheta,\Lambda,M)t + e_0(0) \\
        &\leq C(\lambda,\mu,\sigma,\vartheta,\Lambda,M,T^*).
	\end{align}
	Let $j\geq 1$. Then, integrating \eqref{eq:dt_ej} and using $\xi_j(0)=0$ (for $j\geq 1$), we obtain
	\begin{align}
		&e_j(t) + \sigma \int_0^t\xi_jE_{j+1}d\tau \\
        &\quad \leq \int_0^t \frac{m}{t_0}\xi_{j-1}E_j d\tau +\frac{C(\lambda,\mu, \sigma,\vartheta,\Lambda,j,M)t_0^{j+1}}{t_0^j} \\
		&\quad \leq \frac{m}{t_0}\frac{C(\lambda,\mu, \sigma,\vartheta,\Lambda,j,m,M,T^*)}{t_0^{j-1}} + \frac{C(\lambda,\mu, \sigma,\vartheta,\Lambda,j,M,T^*)}{t_0^j},
	\end{align}
	where we also used the induction hypothesis in the last step. This proves \eqref{eq:LTE_induction}. Evaluating \eqref{eq:LTE_induction} at $t=t_0$ for $j=m$ and using $\xi_m(t_0)=1$, we conclude
	\begin{align}
		\left.\int|\nabla_s^{2m}\kappa|^2\zeta^{4m+4}ds\right\vert_{t=t_0} \leq \frac{C(\lambda,\mu, \sigma,\vartheta,\Lambda,m,M,T^*)}{t_0^m}.\label{eq:LTE_induction_even}
	\end{align}
	To get the same for $2m$ replaced by $2m+1$, we argue as in \Cref{lem:interpolation_IBP}. By integration by parts, H\"older, and Young's inequality
	\begin{align}
		\int|\nabla_s^{2m+1}\kappa|^2\zeta^{4m+6}ds &\leq \Big(\int|\nabla_s^{2m+2}\kappa|^2\zeta^{4m+8}ds\Big)^{\frac{1}{2}}\Big( |\nabla_s^{2m}\kappa|^2\zeta^{4m+4}ds\Big)^{\frac{1}{2}}  \\
		&\qquad + \frac{1}{2}\int|\nabla_s^{2m+1}\kappa|\zeta^{4m+6}ds \\
        &\qquad\qquad + C(\Lambda,m)\int|\nabla_s^{2m}\kappa|\zeta^{4m+4}ds.
	\end{align}
	Absorbing, evaluating at $t=t_0$, and using \eqref{eq:LTE_induction_even}, we find
	\begin{align}							\left.\int|\nabla_s^{2m+1}\kappa|^2\zeta^{4m+6}ds \right\vert_{t=t_0}&\leq \frac{C(\lambda,\mu, \sigma,\vartheta,\Lambda,m,T^*)}{t_0^{\frac{m+1}{2}}t_0^{\frac{m}{2}}} + \frac{C(\lambda,\mu, \sigma,\vartheta,\Lambda,m,T^*)}{t_0^{m}} \\
    &\leq \frac{C(\lambda,\mu, \sigma,\vartheta,\Lambda,m,T^*)}{t_0^{\frac{2m+1}{2}}}. \label{eq:LTE_induction_odd}
	\end{align}
	Renaming $t_0$ into $t$, the statement now follows from \eqref{eq:LTE_induction_even} and \eqref{eq:LTE_induction_odd}.
\end{proof}

\begin{lemma}\label{lem:LTE_L2_bounds}
      Let %$\gamma$ be a smooth solution to \eqref{eq:LTE_general_flow} with 
      $\sigma>0$.
      Suppose that $B[\gamma(t)] \leq M$ for all $t\in [0,T)$. Then for all $m\in\N_0$ and $t\in (0,T)$,
	\begin{align}
		\Vert\nabla_s^{m}\kappa\Vert_{L^2(ds)} &\leq C(\lambda,\mu,\sigma,\vartheta,m,M)(1+t^{-\frac{m}{4}}),\label{eq:LTE_L2_bounds_global}\\
        \Vert \nabla_s^m\kappa\Vert_{\infty}&\leq C(\lambda,\mu,\sigma,\vartheta,m,M)(1+t^{-\frac{2m+1}{8}}).\label{eq:LTE_Linfty_bounds_1}
	\end{align}
\end{lemma}
\begin{proof}
	We first prove \eqref{eq:LTE_L2_bounds_global}. For any $R>0$, there exists $\tilde{\zeta}\in C_c^\infty(\R^n)$ with $\chi_{B_{R}(0)}\leq \tilde{\zeta}\leq \chi_{B_{2R}(0)}$, and such that $|D\tilde{\zeta}|\leq cR^{-1}$ and $|D^2\tilde{\zeta}|\leq c R^{-2}$ for some universal $c>0$, in particular independent of $n$. Hence, taking $R\geq 1$ sufficiently large (depending on the universal $c>0$), we may thus assume $\Lambda=1$ for $\zeta=\tilde{\zeta}\circ\gamma$ in \eqref{eq:ds2_zeta} in the sequel.
			
	If $T\leq 1$ or $t\leq 1 \leq T$, then the statement follows directly from \Cref{lem:LTE_time_cutoff} choosing $T^*=1$, sending $R\to \infty$, and using the monotone convergence theorem. 
	
	Suppose $T> 1$ and $t\geq 1$.  \Cref{lem:LTE_time_cutoff} with $R\to\infty$ yields
	\begin{align}
		\left.\int |\nabla^{m}_s\kappa|^2ds \right\vert_{t=1} \leq C(\lambda,\mu,\sigma,\vartheta,m,M). \label{eq:L2_bounds_1}
	\end{align}
	For $R$ sufficiently large,  \Cref{lem:LTE_gronwall_4th_order} and Gronwall's inequality imply
	\begin{align}
		\int|\nabla_s^m\kappa|^2\zeta^{2m+4}ds &\leq C(\lambda,\mu,\sigma,\vartheta,m,M)+ \left.\int|\nabla_s^m\kappa|^2\zeta^{2m+4}ds\right\vert_{t=1} e^{1-t}\\
		& \leq C(\lambda,\mu,\sigma,\vartheta,m,M)(1+t^{-\frac{2m}{4}}),
	\end{align}
	using \eqref{eq:L2_bounds_1} and $t\geq 1$. Sending $R\to\infty$ and using $\sqrt{a+b}\leq\sqrt{a}+\sqrt{b}$ for $a,b\geq 0$ give \eqref{eq:LTE_L2_bounds_global}.

    For \eqref{eq:LTE_Linfty_bounds_1}, %the second estimate, 
    by \Cref{lem:gagliardo_nirenberg} (with $\Lambda=1$) we have
	\begin{align}
		\Vert \zeta^{m+\frac{1}{2}}\nabla_s^{m}\kappa \Vert_{\infty} &\leq C(m) \Big(\int|\nabla_s^{m+1}\kappa|^2\zeta^{2m+2}ds \Big)^{\frac{1}{4}} \Big(\int|\nabla_s^m\kappa|^2 \zeta^{2m} ds\Big)^{\frac{1}{4}} \\
		&\quad + C(m)\Big( \int|\nabla_s^m\kappa|^2\zeta^{2m} ds\Big)^{\frac{1}{2}} \\
		&\leq C(\lambda,\mu,\sigma,\vartheta,m,M)\Big[(1+t^{-\frac{m+1}{4}})^{\frac{1}{2}}(1+t^{-\frac{m}{4}})^{\frac{1}{2}} + 1+t^{-\frac{m}{4}}\Big],
	\end{align}
    where we used \eqref{eq:LTE_L2_bounds_global} in the last step.
	Sending $R\to\infty$ yields \eqref{eq:LTE_Linfty_bounds_1}.
\end{proof}

We next establish bounds on the derivatives $\partial_s^m\kappa$ (without taking the normal projection).

\begin{lemma}\label{lem:EF_smooth_bound}
	Under the assumptions of \Cref{lem:LTE_L2_bounds}, for $m\in\N_0$ and $t\in (0,T)$,
	\begin{align}
		\Vert \partial_s^m\kappa\Vert_{\infty}&\leq C(\lambda,\mu,\sigma,\vartheta,m,M)(1+t^{-\frac{2m+1}{8}})\label{eq:ds_kappa_infty},\\
		\Vert \partial_s^m\kappa\Vert_{L^2(ds)}&\leq C(\lambda,\mu,\sigma,\vartheta,m,M)(1+t^{-\frac{m}{4}}).\label{eq:ds_kappa_L2}
	\end{align}
\end{lemma}
\begin{proof}
	% We first establish $L^\infty$-bounds on $\nabla_s^m\kappa$ for all $m\geq 0$. Let $\zeta$ be as in the proof of \Cref{lem:LTE_time_cutoff}.	
	% By \Cref{lem:gagliardo_nirenberg} (with $\Lambda=1$) we have
	% \begin{align}
	% 	\Vert \nabla_s^{m}\kappa \zeta^{m+\frac{1}{2}}\Vert_{\infty} &\leq C \Big(\int|\nabla_s^{m+1}\kappa|^2\zeta^{2m+2}ds \Big)^{\frac{1}{4}} \Big(\int|\nabla_s^m\kappa|^2 \zeta^{2m} ds\Big)^{\frac{1}{4}} \\
	% 	&\quad + C\Big( \int|\nabla_s^m\kappa|^2\zeta^{2m} ds\Big)^{\frac{1}{2}} \\
	% 	&\leq C(m,M) (1+t^{-\frac{m+1}{2}})^{\frac{1}{2}}(1+t^{-\frac{m}{2}})^{\frac{1}{2}} + C(m,M)(1+t^{-\frac{m}{2}})^{\frac{1}{2}}.
	% \end{align}
	% Sending $R\to\infty$, we obtain
	% \begin{align}
	% 	\Vert \nabla_s^m\kappa\Vert_{\infty}\leq C(m,M)(1+t^{-\frac{m+1}{4}}).\label{eq:LTE_Linfty_bounds_1}
	% \end{align}
	We recall from \cite[Lemma 2.6]{Dziuk-Kuwert-Schatzle_2002} that
	\begin{align}
		% \nabla_s\kappa-\partial_s\kappa &= |\kappa|^2\partial_s\gamma,\label{eq:LTE_Linfty_bounds_1.5}\\
		\nabla_s^m\kappa-\partial_s^m\kappa&= \sum_{i=1}^{[\frac{m}{2}]}Q^{m-2i}_{2i+1}(\kappa) + \sum_{i=1}^{[\frac{m+1}{2}]}Q^{m+1-2i}_{2i}(\kappa)\partial_s\gamma.\label{eq:LTE_Linfty_bounds_2}
	\end{align}
    Here $[\cdot]$ is the floor function, and
	$Q^a_b(\kappa)$ is a linear combination of terms of the form
	\begin{align}
		\partial_{s}^{i_1}\kappa * \dots * \partial_s^{i_b}\kappa,
	\end{align}
	with $\sum_{j=1}^b i_j=a$. 
    Estimate \eqref{eq:ds_kappa_infty} thus follows by induction from \eqref{eq:LTE_Linfty_bounds_1} and \eqref{eq:LTE_Linfty_bounds_2}. For \eqref{eq:ds_kappa_L2}, one may also proceed by induction using \eqref{eq:LTE_L2_bounds_global}, \eqref{eq:ds_kappa_infty}, and \eqref{eq:LTE_Linfty_bounds_2}, as well as  $\int|\kappa|^2 ds\leq M$.
\end{proof}

\subsection{Curvature control for the curve shortening flow}
We now consider the case $\sigma=\mu=\vartheta=0$ and $\lambda>0$ in \eqref{eq:LTE_general_flow}, yielding essentially the curve shortening flow.
\begin{lemma}\label{lem:LTE_gronwall_2nd_order}
    Let 
    %$\gamma$ be a smooth solution to \eqref{eq:LTE_general_flow} with 
    $\sigma=\mu=\vartheta=0$ and $\lambda>0$. Suppose that $B[\gamma(t)]\leq M$ holds at some time $t\in [0,T)$.  Then, at the time $t$,
    \begin{align}
        \frac{d}{dt}\int|\nabla_s^m\kappa|^2 \zeta^{2m+2} ds + \int|\nabla_s^m\kappa|^2\zeta^{2m+2}ds + \lambda \int|\nabla_s^{m+1}\kappa|^2\zeta^{2m+2}ds \leq C
    \end{align}
    holds for all $m\in\N_0$ where $C=C(\lambda,\Lambda,m,M)$.
\end{lemma}

\begin{proof}
    For this choice of parameters, \Cref{lem:LTE_curvature_evolution} with $r=2$ reads
    \begin{align}
        &\frac{d}{dt} \frac{1}{2}\int|\nabla^m_s\kappa|^2\zeta^{2m+2}ds +\lambda \int|\nabla_s^{m+1}\kappa|^2 \zeta^{2m+2}ds \\
        &=  \lambda \int P^{2m,m}_4\zeta^{2m+2}ds + (2m+2) \lambda \int \langle P^{2m,m}_3,D\tilde\zeta\circ\gamma\rangle \zeta^{2m+1} ds\\
        &\quad - \lambda (2m+2)\int\langle \nabla_s^m\kappa, \nabla_s^{m+1}\kappa\rangle\zeta^{2m+1}\partial_s\zeta ds,
    \end{align}
    where $C=C(\lambda, \Lambda,m)$.
    As in the proof of \Cref{lem:LTE_gronwall_4th_order}, the right hand side consists of terms of the form $\int|P^{a,c}_b|\zeta^q ds$, where the parameters satisfy $a+\frac{b}{2}-1<2(m+1)$, $q\geq a+\frac{b}{2}-1$. Hence, \Cref{prop:interpolation_P} with $k=m+1$ applies and, after absorbing, we have
    \begin{align}
        &\frac{d}{dt}\frac12 \int|\nabla_s^m\kappa|^2\zeta^{2m+2}ds + \frac{2\lambda}{3} \int|\nabla_s^{m+1}\kappa|^2\zeta^{2m+2}ds %+\frac{1}{2}\int|\nabla_s^m\kappa|^2 \zeta^{2m+2}|\kappa|^2 ds 
        \leq C.
    \end{align}
    Applying \Cref{prop:interpolation_P} (also with $k=m+1$) to
    \begin{align}
        \int|\nabla_s^m\kappa|^2\zeta^{2m+2}ds = \int P_2^{2m,m}\zeta^{2m+2} ds
    \end{align}
    yields the statement.
\end{proof}

We have the following analogue of \Cref{lem:LTE_time_cutoff}.

\begin{lemma}\label{lem:LTE_time_cutoff_CSF}
    Let %$\gamma$ be a smooth solution to \eqref{eq:LTE_general_flow} with 
    $\sigma=\mu=\vartheta=0$ and $\lambda>0$, and let $0<T\leq T^*<\infty$.
    Suppose that $B[\gamma(t)]\leq M$ for all $t\in [0,T)$. Then for all $m\in\N$ and $t\in (0,T)$,
	\begin{align}
		\int |\nabla^m_s\kappa|^2 \zeta^{2m+2} ds \leq \frac{C(\lambda,\Lambda,m,M,T^*)}{t^{m}}.
	\end{align}
\end{lemma}

\begin{proof}
    We proceed exactly as in the proof of \Cref{lem:LTE_time_cutoff}, redefining $E_j$ by
    \begin{align}
        E_j(t)\vcentcolon = \int |\nabla_s^j\kappa|^2\zeta^{2j+2}ds,\quad e_j(t)\vcentcolon = \xi_j(t)E_j(t),
    \end{align}
    where $\xi_j$ is as in \eqref{eq:def_xi_j}. Now, we may use \Cref{lem:LTE_gronwall_2nd_order} to conclude that \eqref{eq:dt_ej} holds with $\sigma>0$ replaced by $\lambda>0$. We may thus prove \eqref{eq:LTE_induction} by induction and evaluate at $t=t_0$ and for $j=m$ to conclude the statement.
\end{proof}

Arguing exactly as in \Cref{lem:LTE_L2_bounds,lem:EF_smooth_bound}, we obtain the following.

\begin{lemma}\label{lem:CSF_smooth_bound}
 Let %$\gamma$ be a smooth solution to \eqref{eq:LTE_general_flow} with 
 $\sigma=\mu=\vartheta=0$ and $\lambda>0$.
 Suppose that $B[\gamma(t)]\leq M$ for all $t\in [0,T)$. Then for all $m\in\N_0$ and $t\in (0,T)$,
    \begin{align}
        \Vert \nabla_s^m\kappa\Vert_{L^2(ds)} + \Vert\partial_s^m\kappa\Vert_{L^2(ds)} &\leq C(\lambda,\Lambda,m,M)(1+t^{-\frac{m}{2}}),\\
        \Vert \nabla_s^m\kappa\Vert_\infty+\Vert\partial_s^m\kappa\Vert_\infty&\leq C(\lambda,\Lambda,m,M)(1+t^{-\frac{2m+1}{4}}).
    \end{align}
\end{lemma}

\section{Long-time behavior}\label{sec:long-time_behavior}

In this section we discuss the long-time behavior, proving the main theorems in \Cref{sec:main_results}.
Throughout this section we will repeatedly use, without explicitly mentioning, the fact that under the finiteness of the direction energy we always have properness (\Cref{lem:direction_horizontal}) and hence we can use the results from \Cref{sec:curvature_control}.

\subsection{Blow-up rates}
Consider a general flow \eqref{eq:LTE_general_flow} with $\sigma>0$.
Given $\gamma_0\in \dot{C}^{\infty}(\R)$ with $\inf_\R |\partial_x\gamma_0|>0$, we define the \emph{maximal existence time} $T\in(0,\infty]$ by the supremum of $T'\in(0,\infty)$ such that there is a unique smooth solution $\gamma:[0,T']\times\R\to\R^n$ to \eqref{eq:LTE_general_flow} starting from $\gamma_0$ such that $\gamma-\gamma_0\in C^\infty([0,T']\times\R)$ and $\inf_{[0,T']\times\R}|\partial_x\gamma|>0$. The maximal existence time is well-defined thanks to \Cref{thm:GF_well_posed}.
We call a solution with maximal existence time $T$ \emph{maximal solution}.

Of course, not all flows as in \eqref{eq:LTE_general_flow} will develop a singularity. However, our curvature estimates enable us to characterize the behavior of the bending energy whenever this happens in finite time.

\begin{theorem}\label{thm:blow-up_rate_4th_order}
	Suppose $\sigma>0$.
	Let $\gamma_0\in \dot{C}^{\infty}(\R)$ be a properly immersed curve with $\inf_\R |\partial_x\gamma_0|>0$ and $B[\gamma_0]<\infty$. Suppose that the maximal solution $\gamma:[0,T)\times\R\to\R^n$ to \eqref{eq:LTE_general_flow} with initial datum $\gamma_0$ exists for a finite time $T<\infty$. 
    Then, we have
	\begin{align}\label{eq:blow-up-rate}
		\liminf_{t\to T} (T-t)^{1/4}B[\gamma(t)]>0.
	\end{align}
\end{theorem}

To prove this, we first show the continuity of the bending energy along the flow.

\begin{lemma}\label{lem:blow-up-rate-finite-bending}
    Let $\gamma:[0,T)\times\R\to\R^n$ be a properly immersed maximal solution to \eqref{eq:LTE_general_flow} with $\sigma>0$ such that $B[\gamma_0]<\infty$. Then the map $t\mapsto B[\gamma(t)]$ is finite-valued, continuous on $[0,T)$, and locally Lipschitz continuous on $(0,T)$.
\end{lemma}

\begin{proof}
    Let $T'\in [0,T)$. Let $R>1$ and let $\eta = \varphi^4$ with $\varphi \in C_c^\infty(\R)$ as in the proof of \Cref{prop:energy_decay}.
    For all $t\in [0,T']$, \Cref{lem:bending_derivative_cutoff} yields
    \begin{align}
        &\frac{d}{dt} B[\gamma|\eta] + \sigma \int|\nabla_s^2\kappa|^2\eta ds \\
        & \quad = \lambda\int\langle\nabla_s^2\kappa,\kappa\rangle \eta ds + \mu \int |\kappa|^2\langle \nabla_s^2\kappa, \kappa\rangle\eta ds + \frac{1}{2}\int|\kappa|^2\langle\kappa, V\rangle \eta ds \\
        & \quad\quad - \frac{1}{2}\int|\kappa|^2\xi\partial_s\eta ds + 2 \int\langle \nabla_s\kappa,V\rangle \partial_s\eta ds + \int\langle \kappa, V\rangle \partial_s^2\eta ds. \label{eq:lemma8-1-1}
    \end{align}
    Estimating
    \begin{align}\label{eq:lemma8-1-2}
    \Big|\lambda\int\langle\nabla_s^2\kappa,\kappa\rangle \eta ds\Big| \leq \frac{\sigma}{2}\int|\nabla_s^2\kappa|^2\eta ds + \frac{\lambda^2}{2\sigma} \int|\kappa|^2\eta ds
    \end{align}
    and rearranging, we deduce
    \begin{align}
        &\frac{d}{dt} B[\gamma|\eta] + \frac{\sigma}{2} \int|\nabla_s^2\kappa|^2\eta ds \\
        & \quad \leq \frac{1}{2}\int|\kappa|^2 \Big(\lambda^2\sigma^{-1} + 2\mu\langle \nabla_s^2\kappa, \kappa\rangle + \langle\kappa, V\rangle \Big)\eta ds \\
        & \quad\quad - \frac{1}{2}\int|\kappa|^2\xi\partial_s\eta ds + 2 \int\langle \nabla_s\kappa,V\rangle \partial_s\eta ds + \int\langle \kappa, V\rangle \partial_s^2\eta ds.
    \end{align} 
    Using that the derivatives of $\gamma$ are uniformly bounded by $C=C(\gamma,T')$ (independently of $R$) on $[0,T']\times\R$, and that by \eqref{eq:Energy_decay_control_eta_derivatives} and \eqref{eq:energy_decay_local_length_control} we have $\int(|\partial_s\eta|+|\partial_s^2\eta|)ds\leq C$, we estimate
    \begin{align}\label{eq:0212-01}
         &\frac{d}{dt} B[\gamma|\eta] + \frac{\sigma}{2} \int|\nabla_s^2\kappa|^2\eta ds  \leq C(\gamma,\lambda,\mu,\sigma,T')B[\gamma|\eta] + C(\gamma,T').
    \end{align}
    Dropping the term involving $\sigma>0$, applying Gronwall's lemma to the smooth function $t\mapsto B[\gamma(t)|\eta]$, and sending $R\to\infty$ yield $\sup_{[0,T']}B[\gamma]<\infty$.     
    We now integrate \eqref{eq:lemma8-1-1} in time over $[t_1,t_2]\subset[0,T']$, rearrange so that the left hand side is only 
    $B[\gamma(t_2)|\eta]-B[\gamma(t_1)|\eta]$, take absolute values of both sides, use \eqref{eq:lemma8-1-2}, and send $R\to\infty$
    to get
    \begin{align}
        |B[\gamma(t_2)]-B[\gamma(t_1)]| &\leq \int_{t_1}^{t_2}\Big( \frac{3\sigma}{2} \int|\nabla_s^2\kappa|^2ds+ C B[\gamma(t)] + C\Big) dt \\
        &\leq \frac{3\sigma}{2} \int_{t_1}^{t_2}\int|\nabla_s^2\kappa|^2ds dt +C|t_2-t_1|,\label{eq:lemma8-1-3}
    \end{align}
    where $C=C(\gamma,\lambda,\mu,\sigma,T')$.
    % , move the $\sigma$-term to the right hand side, send $R\to\infty$, and take absolute values to conclude
    %    % taking absolute values in \eqref{eq:lemma8-1-1}, integrating over $[t_1,t_2]\subset [0,T']$, and sending $R\to\infty$ now yields
    % \begin{align}\label{eq:lemma8-1-2}
    %     |B[\gamma(t_2)] - B[\gamma(t_1)]|\leq C |t_2-t_1|  + \frac{\sigma}{2}\int_{t_1}^{t_2} \int |\nabla_s^2\kappa|^2 ds dt .
    % \end{align}
    The $L^2$-bounds in \Cref{lem:LTE_L2_bounds} imply that the first term on the right is controlled if $[t_1,t_2]\subset [t_0,T']\subset(0,T']$, in which case we conclude the Lipschitz estimate 
    \begin{align}
        |B[\gamma(t_2)]-B[\gamma(t_1)]|\leq C(\gamma,\lambda,\mu,\sigma,t_0,T') |t_2-t_1|.
    \end{align}
    On the other hand, integrating \eqref{eq:0212-01} on $[0,T']$, we have
    % \eqref{eq:lemma8-1-1} in $[0,T']$, rearranging, using \eqref{eq:lemma8-1-2}, and taking $R\to\infty$, we find
    \begin{align}
        B[\gamma(T')] - B[\gamma(0)] + \frac{\sigma}{2}\int_0^{T'} \int |\nabla_s^2\kappa|^2 ds dt \leq C T',
    \end{align}
    which implies that $\int_0^{T'} \int |\nabla_s^2\kappa|^2 ds d t<\infty$. 
    % \TM{For this it is sufficient to integrate \eqref{eq:0212-01}?}
    Hence, taking $t_1=0$ and $t_2=t \to 0$ in \eqref{eq:lemma8-1-3}, we conclude continuity at $t=0$.
\end{proof}

\begin{proof}[Proof of \Cref{thm:blow-up_rate_4th_order}]
    We first note that (although the finiteness of the bending energy does not imply properness), since here we assume the properness of $\gamma_0$, the maximal solution $\gamma$ is also proper on $[0,T)$ by \Cref{thm:GF_well_posed}.
    
    By \Cref{lem:blow-up-rate-finite-bending}  we have $B[\gamma(t)]<\infty$ for all $t\in [0,T)$.
	We first prove that $\limsup_{t\to T}B[\gamma(t)]=\infty$. Indeed, if this was not the case, then there is $M\in(0,\infty)$ such that $B[\gamma(t)]\leq M$ for all $t\in [0,T)$. By \Cref{lem:EF_smooth_bound}, for all $m\in\N_0$, we have
    \[
    \sup_{t\in[T/2,T)}\|\partial_s^m\kappa\|_\infty \leq C(\lambda,\mu,\sigma,\vartheta,m,M,T).
    \]
	%Letting $V:=\partial_t\gamma$, we
    From \eqref{eq:LTE_general_flow}, we conclude that for $t\in [T/2,T)$ we have
	\begin{equation}
		\|\partial_s^m \partial_t \gamma\|_\infty \leq C(\lambda,\mu,\sigma,\vartheta,m,M,T),
	\end{equation}
	and further deduce %from $|\gamma-\gamma_0|\leq\sup_{t}|\partial_t\gamma|$ that
	\begin{equation}
		\sup_{t\in[T/2,T)} \|\gamma(t)-\gamma(T/2)\|_{L^\infty(\R)} \leq \frac{T}{2}\sup_{t\in[T/2,T)}\Vert \partial_t\gamma\Vert_{L^\infty(\R)}
		\leq C(\lambda,\mu,\sigma,\vartheta,M,T).
	\end{equation}
	In addition, by the differential equation $\partial_t|\partial_x\gamma|=-\langle \kappa, V \rangle |\partial_x\gamma|$ and the boundedness of $\langle \kappa,V\rangle$, on $[T/2,T)$ we also have the uniform estimate
	\begin{equation}\label{eq:control_arclength}
		C^{-1}\leq |\partial_x\gamma| \leq C, \quad \text{where $C=C(\partial_x\gamma_0,\lambda,\mu,\sigma,\vartheta,M,T)>0$}.
	\end{equation}
	Differentiation and induction as in the proof of \cite[Theorem 3.1]{Dziuk-Kuwert-Schatzle_2002} yield that for $t\in[T/2,T)$ and all $m\in\N_0$, we have
	\begin{equation}
		\|\partial_x^m|\partial_x\gamma|\|_\infty \leq C(\partial_x\gamma_0,\lambda,\mu,\sigma,\vartheta,m,M,T). %\quad\text{ for all }t\in [0,T),\ m\in\N_0.
	\end{equation}
	Combining the above estimates, we finally deduce that for all $t\in[T/2,T)$ and $m\in\N_0$,
	\begin{equation}
		\|\partial_x^m(\gamma-\gamma(T/2))\|_\infty \leq C(\partial_x\gamma_0,\lambda,\mu,\sigma,\vartheta,m,M,T).
	\end{equation}
	For any $T/2\leq t_1\leq t_2<T$, we have
	\begin{align}
		\sup_{x\in\R}|\partial_x^m\gamma(t_1,x)-\partial_x^m\gamma(t_2,x)|&\leq  \sup_{x\in\R} \max_{t\in[t_1,t_2]}|\partial_t\partial_x^m\gamma(t,x)||t_1-t_2| \\
		&\leq \sup_{t\in[0,T)}\Vert \partial_x^m \partial_t\gamma\Vert_{\infty}|t_1-t_2| \\
        &\leq C(\partial_x\gamma_0,\lambda,\mu,\sigma,\vartheta,m,M,T)|t_1-t_2|.
	\end{align}
	This together with \eqref{eq:control_arclength} implies that $\gamma(t)-\gamma(T/2)$ is Cauchy in $C^m(\R)$ as $t\nearrow T$ and thus there exists a smooth immersion $\gamma(T) \in \dot{C}^\infty(\R)$ with $\inf_\R|\partial_x\gamma(T)|>0$ such that $\gamma(t)-\gamma(T/2)\to \gamma(T)-\gamma(T/2)$ as $t\nearrow T$ in $C^m(\R)$ for all $m\in \N_0$. 
	Now we may apply \Cref{thm:GF_well_posed} with initial datum $\gamma(T)$ and glue the resulting solution together with $\gamma\colon[0,T]\times\R\to\R^n$ to extend $\gamma$ beyond $T$, contradicting the maximality of $T$.

    Hence, there exist $t_j\to T$ with $B[\gamma(t_j)]\to\infty$. By \Cref{lem:LTE_curvature_evolution} with $m=0$ and $r=4$, we have
	\begin{align}
		\frac{d}{dt} \frac{1}{2}\int|\kappa|^2 \zeta^4 ds + \sigma \int|\nabla_s^2\kappa|^2 \zeta^4 ds \leq \sum \int |P^{a,c}_b| \zeta^{q} ds,
	\end{align}
	where the sum is finite. The indices satisfy $c\leq 2$ and the triple  
    $(a,b,q)$ runs over the set of $(2,4,4)$, $(0,6,4)$, $(0,4,4)$, $(2,3,3)$, $(0,3,3)$, $(0,5,3)$, $(1,4,3)$, $(3,2,3)$, $(2,2,2)$, $(2,3,3)$, $(1,2,3)$.
    Applying \Cref{prop:interpolation_P} (with $k=2$), we may estimate
	\begin{align}\label{eq:blow-up-rate-1}
		\frac{d}{dt} \frac{1}{2}\int|\kappa|^2 \zeta^4 ds + \frac{\sigma}{2} \int|\nabla_s^2\kappa|^2 \zeta^4 ds \leq C \sum (B[\gamma]^{\frac{b-\delta}{2-\delta}} + B[\gamma]^{\frac{b}{2}}),
	\end{align}
	where $\delta= \frac{1}{2}(a+\frac{b}{2}-1)$. Using the inequality $x^p \leq C(p,p_0,p_1)(x^{p_0}+x^{p_1})$ for all $x\geq 0$ and $0<p_0\leq p \leq p_1<\infty$, we may estimate the right hand side of \eqref{eq:blow-up-rate-1} by finding the smallest and largest exponent, respectively. We thus obtain
	\begin{align}
		\frac{d}{dt} \frac{1}{2}\int|\kappa|^2 \zeta^4 ds + \frac{\sigma}{2} \int|\nabla_s^2\kappa|^2 \zeta^4 ds \leq C (B[\gamma] + B[\gamma]^5).
	\end{align}
	Integrating and taking $\eta \nearrow 1$, we obtain for all $0\leq t_1\leq t_2<T$
	\begin{align}\label{eq:blow-up-rate-2}
		B[\gamma(t_2)] \leq B[\gamma(t_1)]+ C\int_{t_1}^{t_2} (B[\gamma(\tau)]+B[\gamma(\tau)]^5) d\tau.
	\end{align}
    The statement then follows from \Cref{lem:blow-up_convergence_inequality}~\eqref{item:blow_up}.
\end{proof}

In the case of the curve shortening flow ($\sigma=\mu=\vartheta=0$ and $\lambda=1$), the situation is similar but algebraically less involved.
The maximal existence time is defined analogously as in the case $\sigma>0$.
The following result is a part of \Cref{thm:main_CSF_dichotomy}.

\begin{lemma}\label{lem:CSF_blow-up}
    Let $\gamma_0\in\dot{C}^\infty(\R)$ with $\inf_\R|\partial_x\gamma_0|>0$ and $D[\gamma_0]<\infty$.
    Suppose that the maximal solution $\gamma\colon[0,T)\times\R\to\R^n$ to \eqref{eq:CSF} with initial datum $\gamma_0$ exists for a finite time $T<\infty$.
    Then
    \begin{equation}
        \liminf_{t\to T}(T-t)^{1/2}B[\gamma(t)]>0.
    \end{equation}
\end{lemma}

As in the fourth-order case, we first show the continuity of $B[\gamma(t)]$.
Here the initial time $t=0$ has to be excluded since the initial bending energy may be infinite.

\begin{lemma}\label{lem:CSF_bending_conti}
    Let $\gamma\colon[0,T)\times\R\to\R^n$ be a maximal solution to \eqref{eq:CSF} with $D[\gamma_0]<\infty$. Then the map $(0,T)\ni t\mapsto B[\gamma(t)]$ is finite-valued and locally Lipschitz continuous.
\end{lemma}

\begin{proof}
    Using \Cref{lem:bending_derivative_cutoff} with $V=\kappa$ and $\xi=0$, and integrating by parts, we obtain
    \begin{align}
		\frac{d}{dt} B[\gamma|\eta] + \int |\nabla_s\kappa|^2 \eta ds %&=  \frac{1}{2}\int|\kappa|^4\eta ds + \int\langle \nabla_s\kappa,\kappa\rangle\partial_s\eta  ds + \int|\kappa|^2 \partial_s^2 \eta ds - \int \frac{1}{2}|\kappa|^2\xi \partial_s\eta ds\\
        &= \frac{1}{2}\int|\kappa|^4\eta ds -\int\langle \nabla_s\kappa,\kappa\rangle\partial_s\eta  ds.\label{eq:0118-11-2}
    \end{align}
    Let $T'\in(0,T)$.
    Let $R>1$ and $\eta = \varphi^4$ with $\varphi \in C_c^\infty(\R)$ as in the proof of \Cref{prop:energy_decay}.
    % Consider $\eta=\varphi^2$ where, for $R>1$, we take $\varphi(x)=\varphi_0(R^{-1}x)$ with $\chi_{[-\frac{1}{2},\frac{1}{2}]}\leq \varphi_0 \leq \chi_{[-1,1]}$.
    
   Dropping the second term on the left hand side of \eqref{eq:0118-11-2} and using that the derivatives of $\gamma$ are bounded by 
   $C=C(\gamma,T')$ on $[0,T']\times\R$, and that by \eqref{eq:Energy_decay_control_eta_derivatives} and \eqref{eq:energy_decay_local_length_control} we have
   $\int|\partial_s\eta|ds\leq C$, we deduce that for all $t\in[0,T']$,
    \begin{align}
		\frac{d}{dt} B[\gamma|\eta] \leq C B[\gamma|\eta] + C. \label{eq:0118-11}
    \end{align}
    Integrating over $[t_1,t_2]\subset[0,T']$ and sending $R\to\infty$ yield for all $0\leq t_1\leq t_2\leq T'$,
    \begin{align}\label{eq:CSF_bending_inequality}
        B[\gamma(t_2)] \leq B[\gamma(t_1)] + C\int_{t_1}^{t_2}\big( B[\gamma(t)] + 1 \big) dt.
    \end{align}
    At this moment $B[\gamma(t_2)]$ and $B[\gamma(t_1)]$ are possibly infinite, but by \Cref{prop:energy_decay} the last integral term is always finite.
    By this integrability, we have $B[\gamma(t_1)]<\infty$ for a.e.\ $t_1\in[0,T']$, and for any such $t_1$ we deduce from \eqref{eq:CSF_bending_inequality} that $B[\gamma(t_2)]<\infty$ for all $t_2\in [t_1,T']$.
    This implies that $B[\gamma(t)]<\infty$ holds for all $t\in(0,T']$.
    
    Now going back to \eqref{eq:0118-11}, using Gronwall's lemma for the smooth function $t\mapsto B[\gamma(t)|\eta]$ and sending $R\to\infty$ yield, for any $t_0\in(0,T']$,
    \begin{align}
        \sup_{t\in[t_0,T']} B[\gamma(t)] \leq C(\gamma,T',t_0).
    \end{align}
    This allows us to use the $L^2$- and $L^\infty$-bounds in \Cref{lem:CSF_smooth_bound} on the time interval $[t_0,T']$.
    Thanks to this, integrating \eqref{eq:0118-11-2} over any interval $[t_1,t_2]\subset[t_0,T']$, sending $R\to\infty$, and using \eqref{eq:Energy_decay_control_eta_derivatives}, we obtain 
    \begin{align}\label{eq:0128-1}
        B[\gamma(t_1)]-B[\gamma(t_2)] + \int_{t_1}^{t_2}\int|\nabla_s\kappa|^2dsdt = \frac{1}{2}\int_{t_1}^{t_2}\int|\kappa|^4dsdt,
    \end{align}
    and in particular the desired Lipschitz continuity
    \begin{align}
        |B[\gamma(t_1)]-B[\gamma(t_2)]| \leq C(\gamma,t_0,T')|t_1-t_2|. &\qedhere
    \end{align}
\end{proof}

\begin{proof}[Proof of \Cref{lem:CSF_blow-up}]
    We first establish an integral inequality for the bending energy (independent of the finiteness of $T$). With $R>1$ and $\varphi\in C_c^\infty(\R)$ as in the proof of \Cref{prop:energy_decay},
    \Cref{prop:interpolation_P} implies
    \begin{align}
        \frac{1}{2}\int|\kappa|^4\varphi^2ds &\leq \int|\nabla_s\kappa|^2\varphi^4ds + C_0\left( \Big(\int_{-R}^R|\kappa|^2ds\Big)^3+\Big(\int_{-R}^R|\kappa|^2ds\Big)^2 \right),
    \end{align}
    for some universal $C_0>0$.
    Sending $R\to\infty$ and inserting the resulting estimate into \eqref{eq:0128-1}, we conclude that that for a.e.\ $0<t_1\leq t_2<T$ we have
    \begin{equation}\label{eq:0119-01}
        B[\gamma(t_2,\cdot)] \leq B[\gamma(t_1,\cdot)] + C_0 \int_{t_1}^{t_2} \big( B[\gamma(t)]^3+B[\gamma(t)]^2 \big)dt.
    \end{equation}
    
    Now, suppose that $T<\infty$.
    If $\limsup_{t\to T}B[\gamma(t)]<\infty$, then, using \Cref{lem:CSF_smooth_bound} and proceeding exactly as in \Cref{thm:blow-up_rate_4th_order}, we can extend the flow past $T$, contradicting maximality. 
    Hence $\limsup_{t\to T}B[\gamma(t)]=\infty$. 
    Then the assertion follows from \eqref{eq:0119-01} and \Cref{lem:blow-up_convergence_inequality}~\eqref{item:blow_up}.
\end{proof}

Hereafter, we examine precise convergence properties for the flows discussed in Section 2. 
Since in general the parametrization speed $|\partial_x\gamma|$ of the solution is hard to control as $t\to\infty$, we work with the arclength reparametrization as usual.
Note carefully that, in our non-compact case, the choice of the base point of the reparametrization is important, because depending on the choice the limit may change. 

For a smooth family of immersions $\gamma\colon[0,T)\times\R\to\R^n$ with $L[\gamma(t)]=\infty$ for all $t\in [0,T)$, we define the \emph{canonical reparametrization} $\tilde{\gamma}:[0,T)\times\R\to\R^n$ of $\gamma$ by
\begin{equation}
    \tilde{\gamma}(t,s) := \gamma(t, \Phi_t(s)),
\end{equation}
where $\Phi_t:\R\to\R$ denotes the inverse map of $x\mapsto \int_0^x|\partial_x\gamma(t,x')|dx'$, meaning that the base point of the reparametrization is $x=0$.
By construction, $\Phi_t$ gives an orientation-preserving unit-speed reparametrization of $\gamma$.
In particular, both the direction energy and the bending energy are invariant under this reparametrization.

\subsection{Convergence for the curve shortening flow}

We first discuss convergence for the curve shortening flow, completing the proof of \Cref{thm:main_CSF_dichotomy}.

\begin{proof}[Proof of \Cref{thm:main_CSF_dichotomy}]
    We have already shown the blow-up rate (\Cref{lem:CSF_blow-up}) and the decreasing property of the direction energy (\Cref{prop:energy_decay}).
    Hereafter we show the remaining convergence properties, so suppose that $T=\infty$.

    We first prove that $B[\gamma(t)]\to0$.
    By \Cref{prop:energy_decay} we have $B[\gamma]\in L^1(0,\infty)$. Hence,
    \begin{align}\label{eq:DCT_sequence}
        \limsup_{j\to\infty} \inf_{t\in[j,j+1]} B[\gamma(t)] \leq \lim_{j\to\infty} \int_j^{j+1} B[\gamma(t)]dt = 0,
    \end{align}
    and thus there exists $t_j\in[j,j+1]$ such that $B[\gamma(t_j)]\to 0$ as $j\to\infty$. Moreover, by \Cref{lem:CSF_bending_conti}, $B[\gamma]$ is locally Lipschitz continuous on $(0,\infty)$. The desired convergence now follows from \eqref{eq:0119-01} and \Cref{lem:blow-up_convergence_inequality}~\eqref{item:convergence}.
    
    For smooth convergence, we observe that \Cref{lem:CSF_bending_conti} and the above convergence $B[\gamma(t)]\to0$ yield $\sup_{t\in[1,\infty)}B[\gamma(t)]<\infty$, and hence we can use the bounds in \Cref{lem:CSF_smooth_bound} as $t\to\infty$.
    Fix any family of base points $\{s_t\}_{t\geq0}\subset\R$ and take any time sequence $t_j\to\infty$.
    By \Cref{lem:CSF_smooth_bound} and a standard compactness argument---such as the Arzel\`a--Ascoli theorem applied in $\{|s|\leq R\}$ for all \( m \geq 0 \) and a diagonal argument---we deduce that there is a time subsequence $t_{j'}\to\infty$ such that the sequence of translated arclength reparametrizations  
    \[
    \tilde{\gamma}(t_{j'}, \cdot + s_{t_{j'}} ) - \tilde{\gamma}(t_{j'},s_{t_{j'}})
    \]
    locally smoothly converges to an arclength parametrized curve \( \gamma_\infty \in \dot{C}^\infty(\R; \R^n) \) with \( \gamma_\infty(0) = 0 \).
    By Fatou's lemma and $B[\gamma(t_{j'})]\to0$, we have $B[\gamma_\infty]=0$, i.e., $\gamma_\infty$ is a straight line.
    Also by Fatou's lemma and \Cref{prop:energy_decay}, we deduce that
    \[
    D[\gamma_\infty] \leq \liminf_{j'\to\infty} D[\gamma(t_{j'})] \leq D[\gamma_0]<\infty,
    \]
    so the line is uniquely given by $\gamma_\infty(s)=se_1$.
    The uniqueness of the limit implies that
    \begin{align}\label{eq:main_CSF_convergence}
    \lim_{t\to\infty} \big( \tilde\gamma (t,\cdot + s_t) - \tilde{\gamma}(t,s_t) \big) = \gamma_\infty\quad \text{ in $C^m_{\mathrm{loc}}(\R)$ for all $m\in\N_0$.}
    \end{align}
    Moreover, since $\{s_t\}$ is arbitrary, we have shown \eqref{eq:0220-1} for any $m\in\N_0$.
\end{proof}

In order to complement our result, we mention two nontrivial examples.
Both are given in the framework of planar graphs, and thus we use the fact that a planar graph remains so under the curve shortening flow.

\begin{example}[Oscillation]\label{ex:oscillating_CSF}
    There exists a graphical planar initial curve $\gamma_0(x)=(x,u_0(x))$ with $u_0\in C^\infty(\R)$ and $D[\gamma_0]<\infty$ from which the curve shortening flow $\gamma$ keeps oscillating.
    More precisely, if $x_t\in\R$ denotes the unique point such that $\langle\gamma(t,x_t),e_1\rangle=0$, then
    \begin{equation}
        \liminf_{t\to\infty}\langle\gamma(t,x_t),e_2\rangle=-1, \quad \limsup_{t\to\infty}\langle\gamma(t,x_t),e_2\rangle=1.
    \end{equation}
    
    This follows by Nara--Taniguchi's result \cite[Theorem 1.4]{Nara_Taniguchi_2007_convergence_line} on a certain equivalence of the asymptotic behaviors of solutions to the heat equation and to the graphical curve shortening equation, combined with oscillating solutions to the heat equation by Collet--Eckmann \cite[Lemma 8.6]{Collet_Eckmann_1992_case_study}.
    The function $u_0$ can be taken to have first derivative in $L^p(\R)$ for all $p>1$, thus having finite direction energy by \Cref{lem:direction_graphical}.
    
    In this case it is also easy to show that $\gamma(t)$ cannot uniformly converge to a line, even after translation.
    % Indeed, if it converges, then at some time the solution must be contained in a slab of small width, say $1$, which contradicts the above oscillation of width $2$.
\end{example}

\begin{example}[Escaping to infinity]\label{ex:escaping_CSF}
    There exists an unbounded graphical planar initial curve $\gamma_0$ with $D[\gamma_0]<\infty$ from which the curve shortening flow $\gamma$ remains unbounded at each time, and escapes as time goes to infinity, in the sense that 
    \begin{equation}
        \lim_{x\to\pm\infty}\langle\gamma(t,x),e_2\rangle = \infty \quad \text{for any}\ t\geq0, \qquad \lim_{t\to\infty}\inf_{x\in\R}\langle\gamma(t,x),e_2\rangle = \infty.
    \end{equation}
    
    In fact, we may take any embedded curve $\gamma_0\in\dot{C}^\infty(\R;\R^2)$ with graphical ends given by $|x|^\alpha$ as $x\to\pm\infty$, where $\alpha\in(0,\frac{1}{2})$, so that $D[\gamma_0]<\infty$ by \Cref{lem:direction_graphical}.
    Then we can argue by comparing the solution with suitable barriers.
    
    For the first assertion, given any $t_0>0$ and any $M>0$, we site a right-going grim reaper in a slab $\{ p\in \R^2 \mid M-1 \leq \langle p,e_2\rangle \leq M \}$, sufficiently far from the $e_2$-axis (depending on $\gamma_0,t_0,M$) so that for all $t\geq0$ the grim reaper lies below $\gamma_0$.
    Then the comparison between those two curves implies that $\liminf_{x\to\infty}\langle\gamma(t_0,x),e_2\rangle \geq M$.
    By arbitrariness of $M$ and $t_0$, we deduce the first assertion as $x\to\infty$ (the case $x\to-\infty$ is parallel).
    
    For the second assertion, let $h_*>0$ be an arbitrary height.
    Then we can prepare a smooth graphical curve $\gamma_*$ lying below the initial curve $\gamma_0$ and agreeing with the horizontal line $\ell_*:=\{\langle p,e_2\rangle =h_*\}$ outside a compact set.
    Classical heat kernel estimates imply that the solution to the heat equation from $\gamma_*$ uniformly converges to the line $\ell_*$, which together with \cite[Theorem 1.4]{Nara_Taniguchi_2007_convergence_line} implies the same convergence of the curve shortening flow from $\gamma_*$.
    By the comparison principle, we have $\liminf_{t\to\infty}\inf_{x\in\R}\langle\gamma(t,x),e_2\rangle \geq h_*$.
    The arbitrariness of $h_*$ implies the assertion.
\end{example}

\begin{remark}[Convergence proof using maximum principles]\label{rem:convergence_to_line_comparison}
    Here, for comparison purposes, we sketch how maximum principles ensure convergence of the planar curve shortening flow to a line.
    Consider an initial planar embedded curve with ends $C^1$-asymptotic to horizontal semi-lines with opposite directions.
    Then by the comparison principle the solution remains contained in a horizontal slab.
    Also, outside a vertical slab the initial curve is graphical.
    Applying Angenent's intersection number principle \cite{Angenent_1991_intersection} to the solution and vertical lines implies that the graphical part remains graphical, so the non-graphical part stays in a compact set.
    Then an extension of Grayson's argument \cite{Grayson_1987_round} (cf.\ \cite{Polden_1991_evolving_curve}) implies that at some positive time the solution becomes entirely graphical.
    Note also that the $C^1$-asymptotic property of the ends is preserved along the flow, by simple comparisons with grim reapers and interior gradient estimates \cite{Ecker_Huisken_1991_interior} (see also \cite[Lemma 0.1]{Chou_Zhu_1998_complete}).
    Hence, by Ecker--Huisken's convergence result \cite[Theorem 5.1]{Ecker_Huisken_1989_entiregraph}, the solution locally smoothly converges to a self-similar solution (in a rescaled sense).
    Therefore, this limit must be a horizontal line, which is the only graphical self-similar solution contained in the slab. 
    In particular, identifying the limit in a suitable sense using maximum principles seems substantially hard, especially without the slab property or planarity.
\end{remark}

\subsection{Convergence for the surface diffusion and Chen's flow}

We now turn to fourth-order flows of curve shortening type including \eqref{eq:SDF} and \eqref{eq:CF}, and obtain an analogous result to \Cref{thm:main_CSF_dichotomy}.

\begin{theorem}\label{thm:main_SDF_Chen}
    Let $\mu\geq0$, $\vartheta\in\R$, and $\gamma\colon[0,T)\times \R\to\R^n$ be a maximal solution to
    \begin{equation}\label{eq:SDF_Chen_equation}
        \partial_t\gamma=-\nabla_s^2\kappa+\mu|\kappa|^2\kappa+ \vartheta\langle \kappa,\nabla_s\kappa\rangle \partial_s\gamma,
    \end{equation}
    with initial datum $\gamma_0\in \dot{C}^\infty(\R;\R^n)$ such that $\inf_{\R}|\partial_x\gamma_0|>0$.
    Suppose that $D[\gamma_0]<\infty$.
    Then the energy $D[\gamma(t,\cdot)]$ continuously decreases in $t\geq0$.
    In addition, the following dichotomy holds:
    \begin{enumerate}
        \item\label{item:SDF_Chen_blow-up} If $T<\infty$, then
        \begin{equation}
            \liminf_{t\to T}(T-t)^{1/4}\int_{\gamma(t,\cdot)}|\kappa|^2ds>0.
        \end{equation}
        \item\label{item:SDF_Chen_convergence} If $T=\infty$, then
        \begin{equation}
            \lim_{t\to\infty}\int_{\gamma(t,\cdot)}|\kappa|^2ds =0.
        \end{equation} 
        and moreover,
        \begin{equation}
            \lim_{t\to\infty}\sup_{\R}|\partial_s\gamma(t,\cdot)-e_1|=0, \quad \lim_{t\to\infty}\sup_{\R}|\partial_s^m\kappa(t,\cdot)|=0,
        \end{equation}
        for all $m\in\N_0$.
        In particular, after reparametrization by arclength and translation, the solution locally smoothly converges to a horizontal line.
    \end{enumerate}
\end{theorem}

\begin{proof}
    The decay property for $D$ is already shown in \Cref{prop:SDF_Chen_energy_decay}.
    Case \eqref{item:SDF_Chen_blow-up} is a direct consequence of \Cref{thm:blow-up_rate_4th_order}.
    We prove case \eqref{item:SDF_Chen_convergence}, so suppose that $T=\infty$.
    Once we establish the convergence
    \begin{equation}\label{eq:0213-02}
        \lim_{t\to\infty} B[\gamma(t)] = 0,
    \end{equation}
    then the remaining argument is completely parallel to the proof of \Cref{thm:main_CSF_dichotomy} \eqref{item:main_CSF_convergence}, where we use \Cref{lem:blow-up-rate-finite-bending} instead of \Cref{lem:CSF_bending_conti}, as well as \Cref{lem:EF_smooth_bound} instead of \Cref{lem:CSF_smooth_bound}.

    We prove \eqref{eq:0213-02}.  By \Cref{prop:SDF_Chen_energy_decay}, we have
    $D[\gamma(t)]\leq D[\gamma_0]$ for all $t\in [0,\infty)$ and $\int|\nabla_s\kappa|^2ds \in L^1(0,\infty)$.
    Let $R>1$ and let $\eta\in C_c^\infty(\R)$ be as in the proof of \Cref{prop:energy_decay}. Using $\langle \partial_s\kappa, \partial_s\gamma\rangle = -|\kappa|^2$ and integration by parts, we find
    \begin{align}
        \int \langle \partial_s\gamma-e_1,\nabla_s\kappa\rangle \eta ds &= \int \langle \partial_s\gamma-e_1, \partial_s\kappa + |\kappa|^2\partial_s\gamma\rangle \eta ds \\
        &= - \int |\kappa|^2 \langle \partial_s\gamma,e_1\rangle \eta ds - \int \langle \partial_s\gamma-e_1, \kappa\rangle \partial_s\eta ds. \label{eq:SDF_conv_1}
    \end{align}
    Moreover, since $\eta$ has compact support, we have
    \begin{align}
    &\int |\kappa|^2 |\partial_s\gamma-e_1|^2 \eta ds\leq \int|\partial_s\gamma-e_1|^2 ds\,\Vert |\kappa|^2\eta\Vert_\infty \\
    &\qquad\leq 2D[\gamma_0]\Big(2\int |\kappa||\nabla_s\kappa|\eta ds + \int|\kappa|^2 \partial_s\eta ds \Big)  \\
    &\qquad\leq 4D[\gamma_0]^2 \int |\nabla_s\kappa|^2\eta ds + \int|\kappa|^2\eta ds + 2D[\gamma_0]\int|\kappa|^2 \partial_s\eta ds,\label{eq:SDF_conv_2}
\end{align}
Combining \eqref{eq:SDF_conv_1} and \eqref{eq:SDF_conv_2}, and using H\"older, we find
\begin{align}
    &2\int|\kappa|^2\eta ds = \int |\kappa|^2 \left(|\partial_s\gamma-e_1|^2 + 2\langle \partial_s\gamma,e_1\rangle\right) \eta ds \\
    &\leq 4D[\gamma_0]^2 \int|\nabla_s\kappa|^2 \eta ds + \int|\kappa|^2 \eta ds + 2 D[\gamma_0]\int|\kappa|^2\partial_s\eta ds\\
    & +  2\left(2 D[\gamma_0]\right)^{1/2}\Big(\int|\nabla_s\kappa|^2\eta ds\Big)^{1/2}
    +2(2D[\gamma_0])^{1/2}\Big(\int|\kappa|^2|\partial_s\eta|^2 ds\Big)^{1/2}.    
    \label{eq:lemma8-8-1}
\end{align}
For almost all $t_0>0$, we have $\kappa(t_0)\in L^\infty(\R)$ and $\nabla_s\kappa(t_0)\in L^2(ds(t_0))$. Using \eqref{eq:Energy_decay_control_eta_derivatives} and \eqref{eq:energy_decay_local_length_control}, after absorbing it thus follows from sending $R\to\infty$, $\eta\to 1$ in \eqref{eq:lemma8-8-1} that $B[\gamma(t_0)]<\infty$. \Cref{lem:blow-up-rate-finite-bending} implies that $B[\gamma(t)]<\infty$ for all $t\geq t_0$, and hence for all $t>0$. Therefore, for $t>0$, sending $R\to\infty$ in \eqref{eq:lemma8-8-1} yields
\begin{align}
    \int|\kappa|^2 \leq 4D[\gamma_0]^2 \int|\nabla_s\kappa|^2 + 2\left(2 D[\gamma_0]\right)^{1/2}\left(\int|\nabla_s\kappa|^2ds\right)^{1/2}.\label{eq:bending-11-04}
\end{align}
    
%     For any finite time $t>0$, \Cref{lem:LTE_L2_bounds} yields that $\kappa,\nabla_s\kappa\in L^2(ds(t))$. Using $\langle \partial_s\kappa, \partial_s\gamma\rangle = -|\kappa|^2$ and integration by parts, we find
%     \begin{align}
%         \int \langle \partial_s\gamma-e_1,\nabla_s\kappa\rangle ds &= \int \langle \partial_s\gamma-e_1, \partial_s\kappa + |\kappa|^2\partial_s\gamma\rangle ds \\
%         &= - \int |\kappa|^2 \langle \partial_s\gamma,e_1\rangle ds. \label{eq:SDF_conv_1}
%     \end{align}
% Moreover, since $\kappa,\nabla_s\kappa \in L^2(ds)$, we may use \eqref{eq:gagliardo-nirenberg_L_infty} with $\zeta\nearrow 1$ and $|\partial_s\zeta| \leq \Lambda \to 0$, which yields $\Vert |\kappa|^2\Vert_\infty \leq 2 \Vert \kappa\Vert_{L^2(ds)} \Vert \nabla_s\kappa\Vert_{L^2(ds)}$ %\int 2 |\kappa| |\nabla_s\kappa|ds$, 
% so that
% \begin{align}
%     \int |\kappa|^2 |\partial_s\gamma-e_1|^2 ds&\leq \int|\partial_s\gamma-e_1|^2 ds\,\Vert |\kappa|^2\Vert_\infty \leq 4D[\gamma_0] \Vert \kappa\Vert_{L^2(ds)} \Vert \nabla_s\kappa\Vert_{L^2(ds)} \\
%     &\leq 4D[\gamma_0]^2 \int |\nabla_s\kappa|^2 ds + \int|\kappa|^2ds,\label{eq:SDF_conv_2}
% \end{align}
% Combining \eqref{eq:SDF_conv_1} and \eqref{eq:SDF_conv_2}, and using H\"older, we find
% \begin{align}
%     2\int|\kappa|^2 &= \int |\kappa|^2 \left(|\partial_s\gamma-e_1|^2 + 2\langle \partial_s\gamma,e_1\rangle\right) \\
%     &\leq 4D[\gamma_0]^2 \int|\nabla_s\kappa|^2 + \int|\kappa|^2+ 2\left(2 D[\gamma_0]\right)^{1/2}\left(\int|\nabla_s\kappa|^2ds\right)^{1/2}.\label{eq:lemma8-8-1}
% \end{align}
Since $\int|\nabla_s\kappa|^2ds \in L^1(0,\infty)$, we may argue as in \eqref{eq:DCT_sequence} to find a sequence $t_j\in [j,j+1]$ with $\|\nabla_s\kappa\|_{L^2(ds(t_j))}\to 0$. It follows from \eqref{eq:bending-11-04} that also $B[\gamma(t_j)]\to 0$. By \eqref{eq:blow-up-rate-2} (which is valid even if $T=\infty$) and \Cref{lem:blow-up-rate-finite-bending}, we may use \Cref{lem:blow-up_convergence_inequality}~\eqref{item:convergence} to get $B[\gamma(t)]\to 0$ as $t\to\infty$.
\end{proof}

Analogously to \Cref{cor:CSF_blowup}, we obtain the following result by the same proof.

\begin{corollary}\label{cor:SDF_Chen_blowup}
    Let $\gamma_0\in\dot{C}^\infty(\R;\R^2)$ with $\inf_{\R}|\partial_x\gamma|>0$ and $D[\gamma_0]<\infty$.
    Suppose that $\gamma_0$ has nonzero rotation number.
    Then the unique solution to \eqref{eq:SDF_Chen_equation} with $\mu\geq0$ starting from $\gamma_0$ blows up in finite time as in \Cref{thm:main_SDF_Chen} \eqref{item:SDF_Chen_blow-up}.
\end{corollary}

\subsection{Global existence and convergence for the $\lambda$-elastic flow}\label{subsec:EF_convergence}

Now we consider the $\lambda$-elastic flow \eqref{eq:lambda-EF}, which may be regarded as a gradient flow of
\[
E_\lambda[\gamma] := B[\gamma] + \lambda D[\gamma].
\]
We first establish a general global existence theorem.

\begin{theorem}\label{thm:global_lambda_elastic_flow}
    Let $\lambda\geq0$ and $\gamma_0\in \dot{C}^\infty(\R;\R^n)$ with $\inf_{\R}|\partial_x\gamma_0|>0$ and $E_\lambda[\gamma_0]<\infty$.
    If $\lambda=0$, suppose in addition that $\gamma_0$ is proper.
    Then there exists a unique, properly immersed, global-in-time solution $\gamma:[0,\infty)\times\R\to\R^n$ to \eqref{eq:lambda-EF} with initial datum $\gamma_0$ such that $\gamma-\gamma_0\in C^\infty([0,T']\times\R)$ and $\inf_{[0,T']\times\R}|\partial_x\gamma|>0$ for any $T'\in(0,\infty)$.
\end{theorem}

\begin{proof}
    By \Cref{thm:blow-up_rate_4th_order} it suffices to show that the unique maximal solution $\gamma:[0,T)\times\R\to\R^n$ from $\gamma_0$ satisfies $\sup_{t\in[0,T)}B[\gamma(t)]<\infty$.
    This follows since by \Cref{prop:energy_decay} we have $B[\gamma(t)]\leq E_\lambda[\gamma(t)]\leq E_\lambda[\gamma_0]<\infty$ for all $t\in[0,T)$.
\end{proof}

We now prove that any global solution to the $\lambda$-elastic flow converges to a stationary solution, thus being an elastica.
Recall that a curve is called an \emph{elastica} if it solves the equation $\nabla_s^2\kappa +\frac{1}{2}|\kappa|^2\kappa-\lambda\kappa=0$ for some $\lambda\in\R$.
We also call it a \emph{$\lambda$-elastica} to specify the value of $\lambda$.
A $\lambda$-elastica corresponds to a stationary solution to the $\lambda$-elastic flow.
The $\lambda=0$ case is also called a \emph{free elastica}.

\begin{theorem}\label{thm:convergence_lambda_elastic_flow}
    Let $\lambda\geq0$ and $\gamma$ be the global-in-time solution to \eqref{eq:lambda-EF} obtained in Theorem \ref{thm:global_lambda_elastic_flow}.
    Then $\gamma$ satisfies the following properties:
    \begin{enumerate}
    % \item\label{item:convergence_lambdaEF_1} The energy $E_\lambda[\gamma(t)]$ is continuous and non-increasing in $t\geq0$, and moreover
    %     \begin{equation}
    %         E_\lambda[\gamma(t)]-E_\lambda[\gamma_0] = - \int_0^t\int_{\gamma(t',\cdot)}\Big| \nabla_s^2\kappa +\frac{1}{2}|\kappa|^2\kappa-\lambda\kappa \Big|^2dsdt'.
    %     \end{equation}
        % \item\label{item:convergence_lambdaEF_2} For any $t\in(0,\infty)$ and $m\in\N_0$, we have
        % \begin{align}
        %     \|\nabla_s^m\kappa\|_\infty\leq C(1+t^{-\frac{m+1}{4}}), \quad \|\nabla_s^m\kappa\|_{L^2(ds)}\leq C(1+t^{-\frac{m}{4}}),
        % \end{align}
        % and
        % \begin{align}
        %     \|\partial_s^m\kappa\|_\infty\leq C(1+t^{-\frac{m+1}{4}}), \quad \|\partial_s^m\kappa\|_{L^2(ds)}\leq C(1+t^{-\frac{m}{4}}),
        % \end{align}
        % where $C>0$ depends only on $m$ and $E_\lambda[\gamma_0]$.
        \item\label{item:convergence_lambdaEF_3} For any sequence of times $t_j\to\infty$ and of base points $\{s_j\}_j\subset\R$ there exist subsequences $\{t_{j'}\}\subset\{t_j\}$ and $\{s_{j'}\}\subset\{s_j\}$ such that the sequence of translated reparametrizations $\tilde{\gamma}(t_{j'},\cdot+s_{j'})-\tilde{\gamma}(t_{j'},s_{j'})$ locally smoothly converges as $j'\to\infty$ to a $\lambda$-elastica $\gamma_\infty\in\dot{C}^\infty(\R;\R^n)$ with $\gamma_\infty(0)=0$.
        \item\label{item:convergence_lambdaEF_4} The above limit satisfies the following lower semi-continuity:
        \[
        \lim_{t\to\infty }E_\lambda[\tilde{\gamma}(t)] \geq E_\lambda[\gamma_\infty].
        \]
    \end{enumerate}
\end{theorem}

\begin{proof}
    By \Cref{prop:energy_decay}, we have
     \begin{equation}
            E_\lambda[\gamma(t)]-E_\lambda[\gamma_0] = - \int_0^t\int \Big| \nabla_s^2\kappa +\frac{1}{2}|\kappa|^2\kappa-\lambda\kappa \Big|^2dsdt'.\label{eq:convergence_lambdaEF_1}
    \end{equation}
    In particular, $\sup_{t\in[0,\infty)}B[\gamma(t)]\leq E_\lambda[\gamma_0]$, and thus, \Cref{lem:LTE_L2_bounds,lem:EF_smooth_bound} yield
    \begin{align}
           \|\nabla_s^m\kappa\|_\infty+ \|\partial_s^m\kappa\|_\infty&\leq C(\lambda, m, E_\lambda[\gamma_0])(1+t^{-\frac{2m+1}{8}}),\label{eq:convergence_lambdaEF-04-04}\\ \|\nabla_s^m\kappa\|_{L^2(ds)}+ \|\partial_s^m\kappa\|_{L^2(ds)}&\leq C(\lambda, m, E_\lambda[\gamma_0])(1+t^{-\frac{m}{4}}).\label{eq:convergence_lambdaEF_2}
        \end{align}
        %where $C>0$ depends only on $m$, $\lambda$, and $E_\lambda[\gamma_0]$.

    We now prove property \eqref{item:convergence_lambdaEF_3}.  
    Fix any sequences $t_j\to\infty$ and $\{s_j\}\subset\R$.  
    By \eqref{eq:convergence_lambdaEF-04-04} and a standard compactness argument (as in the proof of \Cref{thm:main_CSF_dichotomy}~\eqref{item:main_CSF_convergence}), up to taking a further subsequence (without relabeling), the sequence of curves
    \[
    \tilde{\gamma}(t_j, \cdot + s_j ) - \tilde{\gamma}(t_j,s_j)
    \]
    locally smoothly converges to an arclength parametrized curve \( \gamma_\infty \in \dot{C}^\infty(\R; \R^n) \) with \( \gamma_\infty(0) = 0 \).
    To establish that \( \gamma_\infty \) is indeed a $\lambda$-elastica, it suffices to show that
    \begin{equation}
        \text{\( \lim_{t\to \infty}\|V\|_{L^2(ds)}(t) = 0 \), \quad where $V=\nabla_s^2\kappa +\frac{1}{2}|\kappa|^2\kappa-\lambda\kappa$.}
    \end{equation}
    Indeed, by Fatou's lemma, this would then imply \( \|V\|_{L^2(ds)} = 0 \) for \( \gamma_\infty \), which confirms that \( \gamma_\infty \) is a $\lambda$-elastica.
    In order to show that $\|V\|_{L^2(ds)}(t) \to 0$, it suffices to prove that $t \mapsto \|V\|_{L^2(ds)}^2(t)$ is globally Lipschitz on $[1,\infty)$,    
    since $\|V\|_{L^2(ds)}\in L^2(0,\infty)$ by property \eqref{eq:convergence_lambdaEF_1}.
    To that end, let $t\geq 1$. 
    Using the definition of $V$ and also \eqref{eq:energy_estimates_0} (with $m=0,2$), properties \eqref{eq:convergence_lambdaEF-04-04} and \eqref{eq:convergence_lambdaEF_2} above imply, for $t\geq1$,
    \begin{align}\label{eq:conv_lambdaEF_4}
       \Vert V\Vert_\infty +\Vert V\Vert_{L^2(ds)} +  \Vert \partial_t^\perp V\Vert_{L^2(ds)} \leq C(\lambda, E_\lambda[\gamma_0]).
    \end{align}
    Taking any cutoff function $\eta\in C_c^\infty(\R)$, we find that
    \begin{align}\label{eq:conv_lambdaEF_5}
        \frac{d}{dt} \int|V|^2\eta ds &= 2 \int\langle V, \partial_t^\perp V\rangle \eta ds - \int|V|^2 \langle \kappa, V\rangle \eta ds.
    \end{align}
    Integrating \eqref{eq:conv_lambdaEF_5}, taking absolute values, sending $\eta\nearrow 1$ while using \eqref{eq:conv_lambdaEF_4} yield
    \begin{align}
        \Big|\|V\|_{L^2(ds)}^2(t) - \|V\|_{L^2(ds)}^2(t')\Big|
        %\Vert V(t)\Vert_{L^2(ds(t)}^2 -\Vert V(t')\Vert_{L^2(ds(t'))}^2 \Big| 
        \leq C |t-t'|
    \end{align}
    for all $t,t'\geq 1$, which is the desired global Lipschitz continuity.

    Property \eqref{item:convergence_lambdaEF_4} then easily follows from \eqref{eq:convergence_lambdaEF_1}, the invariance of the energy $E_\lambda$ under orientation-preserving reparametrizations, and  Fatou's lemma.
\end{proof}

We now examine the precise convergence behavior of the elastic flow with $\lambda=1$, and the free elastic flow ($\lambda=0$).

We first address the free elastic flow, proving \Cref{thm:main_FEF}.

\begin{proof}[Proof of \Cref{thm:main_FEF}]  
    Thanks to \Cref{thm:global_lambda_elastic_flow,thm:convergence_lambda_elastic_flow}, we only need to discuss the asymptotic behavior.
    By \Cref{thm:convergence_lambda_elastic_flow} \eqref{item:convergence_lambdaEF_4}, the free elastica $\gamma_\infty$ in \Cref{thm:convergence_lambda_elastic_flow} \eqref{item:convergence_lambdaEF_3} satisfies $B[\gamma_\infty]<\infty$.
    The well-known classification of elasticae (see e.g.\ \cite{MR772124,Miura_elastica_survey}) implies that any free elastica is either a straight line with $B=0$, or a periodic elastica with $B=\infty$.
    Hence $\gamma_\infty$ must be a line.
    Therefore, by \Cref{thm:convergence_lambda_elastic_flow} \eqref{item:convergence_lambdaEF_3}, $|\partial_s^m\kappa(t_j,s_j)|$ always converges to zero for any $t_j\to\infty$, $\{s_j\}\subset\R$, and $m\in\N_0$.
    As in the proof of \Cref{thm:main_CSF_dichotomy}, the arbitrariness of $t_j,s_j$ implies the assertion.
\end{proof}

From now on we discuss the case $\lambda>0$.
As mentioned in the introduction, up to rescaling, we may assume that $\lambda=1$,
and only consider $E=E_1=B+D$.

In this case the stationary solution may not be a line even in the finite energy regime---it may be a so-called borderline elastica.
We recall basic facts on the (planar) borderline elastica.
A prototypical arclength parametrization is given by
\begin{equation}\label{eq:borderline}
    \gamma_b(s)\vcentcolon=
    \begin{pmatrix}
        s-2\tanh{s}\\
        2\sech{s}
    \end{pmatrix},
\end{equation}
cf.\ \Cref{fig:borderline}.
The tangential angle of $\gamma_b$, satisfying $\partial_s\gamma_b=(\cos\theta_b,\sin\theta_b)$, is given by
$\theta_b(s)=4\arctan(e^{s})$, which increases from $0$ to $2\pi$ as $s$ varies from $-\infty$ to $\infty$.
The signed curvature is then given by
$k_b(s)=\partial_s\theta_b(s)=2\sech{s}=4e^{s}(1+e^{2s})^{-1}$.
This satisfies $k_{ss}+\frac{1}{2}k^3-k=0$.
In particular, we can explicitly compute:
\begin{align}
    E[\gamma_b] &=\int_\R \left( \frac{1}{2}k_b^2 + 1-\cos\theta_b \right) ds = \int_\R \left( \frac{1}{2} k_b^2 + \frac{8\tan^2\frac{\theta_b}{4}}{(1+\tan^2\frac{\theta_b}{4})^2} \right) ds \\
    &= \int_\R \frac{16e^{2s}}{(1+e^{2s})^2} ds = \left[ -\frac{8}{1+e^{2s}} \right]_{-\infty}^\infty = 8. \label{eq:energy_8_borderline}
\end{align}

\begin{theorem}\label{thm:convergence_elastic_flow}
    Let $\lambda=1$ and $\gamma$ be the global-in-time solution to \eqref{eq:EF} obtained in \Cref{thm:global_lambda_elastic_flow}.
    Then the sequence of curves and the limit curve in \Cref{thm:convergence_lambda_elastic_flow} \eqref{item:convergence_lambdaEF_3} further satisfy the following properties:
    \begin{enumerate}
        \item\label{item:convergence_EF_1} The limit curve $\gamma_\infty$ is given by either the straight line
        \[
        \gamma_\infty(s) = se_1,
        \]
        or the (planar) borderline elastica of the form
        \[
        \gamma_\infty(s)=(s-2\tanh{(s+s_0)})e_1-(2\sech{(s+s_0)})\omega + v_0,
        \]
        for some $s_0\in\R$ and $\omega \in \Span\{e_2,\dots,e_2\}$ with $|\omega|=1$, and $v_0:=(2\tanh{s_0})e_1-(2\sech{s_0})\omega$ so that $\gamma_\infty(0)=0$.
        \item\label{item:convergence_EF_2} If $\partial_s\tilde{\gamma}(t_{j'},s_{j'})\to e_1$ as $j'\to\infty$, then $\gamma_\infty$ is a straight line.
        Otherwise, $\gamma_\infty$ is a borderline elastica.
        \item\label{item:convergence_EF_3} If $\gamma_\infty$ is a straight line, then $E[\gamma_\infty]=0$, while if $\gamma_\infty$ is a borderline elastica, then $|\kappa(s)|=2\sech{(s+s_0)}$ and $E[\gamma_\infty]=8$.
    \end{enumerate}
\end{theorem}

\begin{proof}
    Property \eqref{item:convergence_EF_1} is shown by using the known classification of elasticae (\cite{MR772124,Miura_elastica_survey}) as follows.
    The only infinite-length elasticae with bounded bending energy are a straight line and a borderline elastica, because otherwise the squared curvature $|\kappa|^2$ is a nonzero periodic function.
    We also note that, since $\gamma_\infty$ satisfies the elastica equation $\nabla_s^2\kappa+\frac{1}{2}|\kappa|^2\kappa-\kappa=0$, in the borderline case the size of the loop (or equivalently, the curvature maximum) is uniquely determined; in particular, the curvature must be of the form $|\kappa|=2\sech(s+s_0)$, and thus the image $\gamma_\infty$ is congruent to $\gamma_b$.
    Finally, since $E[\gamma_\infty]<\infty$ follows by \Cref{thm:convergence_lambda_elastic_flow}~\eqref{item:convergence_lambdaEF_4}, both for $\gamma_\infty$ being a line and a borderline elastica, the asymptotic tangent direction must be $e_1$.
    This with $\gamma_\infty(0)=0$ implies the desired form of $\gamma_\infty$.

    Property \eqref{item:convergence_EF_2} follows from smooth convergence and the fact that the tangent direction of our borderline elastica is never rightward, i.e., $\partial_s\gamma_\infty\neq e_1$ everywhere.
    
    Finally, property \eqref{item:convergence_EF_3} follows by direct computation, cf.\ \eqref{eq:energy_8_borderline}.
\end{proof}

\begin{remark}\label{rem:EF_conv_1}
    In the elastic flow case $\lambda>0$, the choice of the base points $\{s_j\}$ in \Cref{thm:convergence_lambda_elastic_flow}~\eqref{item:convergence_lambdaEF_3} is particularly important because the limit does depend on this choice.
    Indeed, since we always have $\partial_s\tilde{\gamma}(t,s)\to e_1$ as $s\to\pm\infty$ (see \Cref{lem:direction_horizontal}), for any solution $\gamma$ we can choose $s_t\in\R$ large enough depending on $t\geq0$ so that $\tilde{\gamma}(t,\cdot+s_t)-\tilde{\gamma}(t,s_t)$ locally smoothly converges to the $e_1$-axis as $t\to\infty$.
    This even occurs when $\gamma$ is a stationary solution given by a borderline elastica.
\end{remark}

\begin{remark}
    On the other hand, it is not clear whether the limit also depends on the choice of the time sequence $t_j\to\infty$ even if we choose a constant sequence of base points.
    It is an interesting problem if the limit in \Cref{thm:convergence_lambda_elastic_flow}~\eqref{item:convergence_lambdaEF_3} is always unique, i.e., independent of $t_j\to\infty$, provided that $s_j = 0$ for all $j\in\N$.
\end{remark}

Next we give several conditions on initial data for full convergence as $t\to\infty$ such that we can identify the type of the limit elastica.
We first complete the proof of global-in-space convergence results in \Cref{cor:EF_convergence_line_<8}.

\begin{proof}[Proof of \Cref{cor:EF_convergence_line_<8}]
    \Cref{thm:convergence_lambda_elastic_flow} \eqref{item:convergence_lambdaEF_4} and \Cref{thm:convergence_elastic_flow} \eqref{item:convergence_EF_1}, \eqref{item:convergence_EF_3} imply that the only limit $\gamma_\infty$ in \Cref{thm:convergence_lambda_elastic_flow} \eqref{item:convergence_lambdaEF_3} with $E[\gamma_\infty]<8$ is the line $\gamma_\infty(s)=se_1$.
    Therefore, the arbitrariness of $t_j,s_j$ yields the assertion.
\end{proof}

We also exhibit some conditions on initial data such that for specific choices of base points, we obtain full convergence results, locally in space.

\begin{corollary}\label{cor:EF_full_convergence}
    Let $\lambda=1$ and $\gamma$ be the global-in-time solution to \eqref{eq:EF} obtained in \Cref{thm:global_lambda_elastic_flow}.
    Let $\tilde{\gamma}$ be the canonical reparametrization of $\gamma$.
    \begin{enumerate}
        \item\label{item:EF_full_1} If $\gamma_0$ has point symmetry such that $\gamma_0(-x)=-\gamma_0(x)$ for $x\in\R$, 
        then the curve $\tilde{\gamma}(t,\cdot)$ locally smoothly converges to a straight line as $t\to\infty$.
        \item\label{item:EF_full_2} Suppose that $\gamma_0$ has reflection symmetry such that $\gamma_0(-x)=\gamma_0(x)-2\langle \gamma_0(x),e_1 \rangle e_1$ for $x\in\R$.
        If $\partial_s\gamma_0(0)=e_1$, then the curve $\tilde{\gamma}(t,\cdot)-\tilde{\gamma}(t,0)$ locally smoothly converges to a straight line as $t\to\infty$.
        If otherwise $\partial_s\gamma_0(0)=-e_1$, then for some $\{R_t\}_{t\geq0}\subset O(n)$ the curve $R_t[\tilde{\gamma}(t,\cdot)-\tilde{\gamma}(t,0)]$ locally smoothly converges to a borderline elastica $\gamma_\infty$ with $\partial_s\gamma_\infty(0)=-e_1$ as $t\to\infty$.
        \item\label{item:EF_full_3} Suppose that $n=2$ (planar).
        If $N[\gamma_0]>0$ (resp. $N[\gamma_0]<0$), then there are sequences of base points $\{s_t^i\}_{t\geq0}\subset\R$, where $i=1,\dots,|N[\gamma_0]|$, such that for each $i$ the curve $\tilde{\gamma}(t,\cdot+s_t^i)-\tilde{\gamma}(t,s_t^i)$ locally smoothly converges as $t\to\infty$ to a borderline elastica $\gamma_\infty$ with $\partial_s\gamma_\infty(0)=-e_1$ and $N[\gamma_\infty]=1$ (resp.\ $N[\gamma_\infty]=-1$).
        If in addition $|N[\gamma_0]|\geq2$, then the base points can be chosen so that $\lim_{t\to\infty}(s_t^{i+1}-s_t^{i})=\infty$ for each $i=1,\dots,|N[\gamma_0]|-1$.
    \end{enumerate}
\end{corollary}

\begin{proof}  
    We first prove case \eqref{item:EF_full_1}.
    By unique solvability, the point symmetry of the initial curve is preserved under the flow.
    Hence in particular $\tilde{\gamma}(t,0)$ is fixed at the origin, and thus by \Cref{thm:convergence_lambda_elastic_flow} \eqref{item:convergence_lambdaEF_3} (with $s_j\equiv0$) the curve $\tilde{\gamma}(t,\cdot)$ sub-converges (without translation) to an elastica $\gamma_\infty$ with point symmetry.
    No borderline elastica in \Cref{thm:convergence_elastic_flow} \eqref{item:convergence_EF_1} has such symmetry, so the limit is uniquely determined as the straight line $\gamma_\infty(s)=se_1$.
    Thanks to the uniqueness of the limit, the convergence is now upgraded to the full limit $t\to\infty$.

    We prove case \eqref{item:EF_full_2}.
    By symmetry, $\partial_s\gamma_0(0)\in\{e_1,-e_1\}$.
    As above, the reflection symmetry is preserved under the flow.
    This with smoothness of the flow implies that the initial condition on $\partial_s\gamma_0(0)$ is preserved under the flow, and hence $\partial_s\gamma_0(0)=\partial_s\gamma_\infty(0)\in\{e_1,-e_1\}$ holds for the limit $\gamma_\infty$ in \Cref{thm:convergence_lambda_elastic_flow} \eqref{item:convergence_lambdaEF_3} with $s_j\equiv0$.
    If $\partial_s\gamma_\infty(0)=e_1$, then by \Cref{thm:convergence_elastic_flow} \eqref{item:convergence_EF_2} the limit is again the unique straight line as above and thus the proof is complete.
    If $\partial_s\gamma_\infty(0)=-e_1$, then by \Cref{thm:convergence_elastic_flow} \eqref{item:convergence_EF_2} the limit is a borderline elastica.
    Since $\partial_s\gamma_\infty$ is then injective, the point $\gamma_\infty(0)$ must be the tip of the loop, so that the image of the limit curve is unique up to taking a suitable sequence of orthogonal matrices $\{R_t\}\subset O(n)$---more precisely, in view of the parametrization in \Cref{thm:convergence_elastic_flow} \eqref{item:convergence_EF_1}, if $n=2$ then each matrix $R_t$ is either the identity or reflection across the $e_1$-axis, while if $n\geq3$ then $R_t$ is rotation around the $e_1$-axis.
    Then again the uniqueness of the limit curve implies full convergence.

    We finally prove case \eqref{item:EF_full_3}.
    We only discuss the case $N:=N[\gamma_0]>0$; the other is parallel.
    Since the rotation number $N[\gamma(t,\cdot)]$ is preserved along the flow (\Cref{lem:rotation_number_preservation}), there exists a smooth tangential angle $\theta:[0,\infty)\times\R\to\R$, i.e., $\partial_s\tilde{\gamma}(t,s)=(\cos\theta(t,s),\sin\theta(t,s))$, such that \begin{equation}
        \lim_{s\to-\infty}\theta(t,s)=0, \quad \lim_{s\to\infty}\theta(t,s)=2\pi N.
    \end{equation}
    Then for each $t\geq0$ there are $N$ points $s_t^1<\dots<s_t^N$ such that $\theta(t,s_t^i)=2\pi(i-\frac{1}{2})$ and $\partial_s\theta(t,s_t^i)\geq0$ for $i=1,\dots,N$.
    Since $\partial_s\tilde{\gamma}(t,s_t^i)=-e_1$ and the signed curvature is nonnegative there, for each $i$ the curve $\tilde{\gamma}(t,\cdot+s_t^i)-\tilde{\gamma}(t,s_t^i)$ locally converges to the desired unique borderline elastica, namely $\gamma_\infty=\gamma_b-\gamma_b(0)$ (resp.\ $\gamma_\infty$ is the vertical reflection of $\gamma_b-\gamma_b(0)$).
    If $N\geq2$, then in view of the compatibility between those local convergences, the distance between any adjacent base points needs to diverge as $t\to\infty$.
\end{proof}

\begin{remark}
    Novaga--Okabe \cite[Theorem 3.3]{Novaga_Okabe_2014_infinite} also claimed convergence to a borderline elastica assuming nonzero rotation number, but their argument seems not ruling out the possibility that the limit curve is a non-rightward straight line.
    Our argument not only improves their result, but also clearly rules out non-rightward lines using the finiteness of the direction energy.
\end{remark}

The above results already reveal interesting dynamical aspects of the non-compact elastic flow.
In particular, we have the following peculiar example.

\begin{example}[Escaping loops]\label{ex:two_loop}
Consider the planar elastic flow \( \gamma \) starting from a planar initial curve \( \gamma_0:\R\to\R^2 \) with \( E[\gamma_0] < \infty \) and reflection symmetry in the sense that
\begin{itemize}
    \item \( \gamma_0(-x) = \gamma_0(x) - 2\langle \gamma_0(x), e_1 \rangle e_1 \) for \( x \in \R \), and \( \partial_s \gamma_0(0) = e_1 \).
\end{itemize}
Then \Cref{cor:EF_full_convergence}~\eqref{item:EF_full_2} implies that, around the origin, the curve \( \tilde{\gamma}(t,\cdot) \) locally converges to a line (after translation).

In addition, note that in this case \( N[\gamma_0] \in 2\Z \). We further assume that
\begin{itemize}
    \item \( N[\gamma_0] = 2 \).
\end{itemize}
A typical situation is that two loops are positioned symmetrically, as in \Cref{fig:twoloop}.
Then \Cref{cor:EF_full_convergence}~\eqref{item:EF_full_3}, combined with symmetry, implies the existence of a sequence \( \{s_t\}_{t\geq0} \subset (0,\infty) \) such that $\tilde{\gamma}(t,\cdot \pm s_t)-\tilde{\gamma}(t, \pm s_t)$ locally converges to a borderline elastica $\gamma_\infty$ with $N[\gamma_\infty]=1$.

In view of the compatibility of the above local convergences, we need to have $s_t\to\infty$ as $t\to\infty$.
This roughly describes the dynamics in which the two loops escape to infinity.

We also observe that the above solution directly shows that we cannot have global-in-space convergence.
Indeed, now we have $\|\partial_s\tilde{\gamma}(t,\cdot)-e_1\|_\infty=2$ since $|\partial_s\tilde{\gamma}(t,s_t)-e_1| = 2$,
although $\partial_s\tilde{\gamma}(t,\cdot)-e_1$ locally uniformly converges to zero.
A similar argument works for any order of derivative.
\end{example}

Motivated by the above observations, we finally pose a conjecture on the general behavior of planar elastic flows.
For a general planar curve $\gamma:\R\to\R^2$ with $\inf_\R|\partial_x\gamma|>0$ and $E[\gamma]<\infty$, we can obtain the general lower bound of the energy $E$ in terms of the absolute rotation number (in the spirit of the phase transition aspect \cite{Miura20}):
\begin{align}
    E[\gamma]&=\int_\R\left( \frac{1}{2}\partial_s\theta^2+1-\cos\theta \right)ds \geq \left|\int_\R \sqrt{2}\partial_s\theta\sqrt{1-\cos\theta} ds \right| \\
    &= \sqrt{2}\left|\int_0^{2\pi N[\gamma]} \sqrt{1-\cos\theta} d\theta \right| = 8|N[\gamma]|.\label{eq:EvsN}
\end{align}
Therefore, any planar elastic flow with $E[\gamma_0]<\infty$ satisfies
\begin{align}\label{eq:0327-01}
    \lim_{t\to\infty}E[\gamma(t)] \geq 8|N[\gamma_0]|.
\end{align}
In fact, we conjecture that the energy of a general solution will be asymptotically ``quantized'' by the number of loops captured in \Cref{cor:EF_full_convergence}~(\ref{item:EF_full_3}) in the sense that equality is attained in \eqref{eq:0327-01}.

\begin{conjecture}
    If $n=2$, then the elastic flow in \Cref{thm:main_EF_global} satisfies
    \begin{align}
        \lim_{t\to\infty}E[\gamma(t)] = 8|N[\gamma_0]|.
    \end{align}
\end{conjecture}

In view of \eqref{eq:energy_8_borderline}, equality in \eqref{eq:EvsN} is attained for the borderline elastica $\gamma_b$ (with $N=1$), hence
 the limit could be interpreted as $N[\gamma_0]$ copies of $\gamma_b$.

\section{Rigidity theorems for solitons}\label{sec:soliton}

In this final section, we observe that simple applications of our convergence as well as energy decay results yield new rigidity theorems for solitons with finite direction energy.

Let $\gamma:[0,T)\times\R\to\R^n$ be a solution to a flow with maximal existence time $T\in(0,\infty]$ and initial datum $\gamma_0$.
We call $\gamma$ an \emph{expanding soliton} (resp.\ \emph{shrinking soliton}) if there is an increasing (resp.\ decreasing) function $\alpha:[0,T)\to(0,\infty)$ with $\alpha(0)=1$ and $\lim_{t\to T}\alpha(t)=\infty$ (resp.\ $\lim_{t\to T}\alpha(t)=0$) such that $\gamma(t)$ is isometric to $\alpha(t)\gamma_0$ for all $t\in[0,T)$.
We also call $\gamma$ a \emph{non-scaling soliton} if $T=\infty$ and $\gamma(t)$ is isometric to $\gamma_0$ for all $t\geq0$; typical examples are translators and rotators.
In all cases, we call $\gamma_0$ the \emph{profile curve} of the soliton $\gamma$.

We first address the flows of curve shortening type.

\begin{corollary}\label{cor:CSF_SDF_CF_soliton}
    Let $\gamma_0\in\dot{C}^\infty(\R;\R^n)$ be the profile curve of either a non-scaling or expanding soliton to one of \eqref{eq:CSF}, \eqref{eq:SDF}, \eqref{eq:CF}.
    Suppose that $D[\gamma_0]<\infty$.
    Then $\gamma_0$ is a straight line.
\end{corollary}

\begin{proof}
    The expanding case is easily dealt with thanks to the energy decay results in \Cref{thm:main_CSF_dichotomy,thm:main_SDF_Chen}, since an expanding dilation increases the direction energy.
    In the non-scaling case, the convergence results there imply that $\gamma_0$ must be a line.
\end{proof}

For \eqref{eq:SDF}, several non-stationary solitons are known in compact and non-compact cases; for example, shrinking Bernoulli's figure-eight \cite{EGMWW15} and a non-compact complete translator \cite{Ogden-Warren_2024_raindrop}.
Also in \cite{Koch-Lamm_2012} an abstract existence theorem is obtained for expanding solitons of entire graphs with conical ends.
See also \cite{Asai-Giga_2014_SDFself,Kagaya-Kohsaka_2020_SDF} for solitons with free boundary.
For \eqref{eq:CF}, the only known non-stationary solitons would be the circles and the figure-eight \cite{Cooper_Wheeler_Wheeler_23}.
It is not known whether there are shrinking solitons with finite direction energy for these flows.

As for the free elastic flow, we obtain a similar type of new rigidity result.

\begin{corollary}\label{cor:FEF_soliton}
    Let $\gamma_0\in\dot{C}^\infty(\R;\R^n)$ be the profile curve of either a non-scaling or shrinking soliton to \eqref{eq:FEF}.
    Suppose that $B[\gamma_0]<\infty$.
    Then $\gamma_0$ is a straight line.
\end{corollary}

\begin{proof}
    The shrinking case can be treated by the energy decay in \Cref{thm:main_FEF}, while in the non-scaling case the convergence result implies the assertion.
\end{proof}

Concerning expanding solitons to \eqref{eq:FEF}, the circles and the figure-eight are again explicit examples \cite{Miura_Wheeler_2024_free}, and also the result in \cite{Koch-Lamm_2012} implies existence of expanding entire graphs.
It is unknown if the latter examples have finite bending energy.

Finally we also obtain a new rigidity result for \eqref{eq:EF}, where we can rule out both expanding and shrinking solitons.

\begin{corollary}\label{cor:EF_soliton}
    Let $\gamma_0\in \dot{C}^\infty(\R;\R^n)$ be the profile curve of either a non-scaling, expanding, or shrinking soliton to \eqref{eq:EF}.
    Suppose that $E[\gamma_0]<\infty$.
    Then $\gamma_0$ is stationary, i.e., either a straight line or a borderline elastica.
\end{corollary}

\begin{proof}
    The expanding case (with $D[\gamma_0]>0$) is ruled out since if $\gamma:[0,\infty)\times\R\to\R^n$ is an expanding soliton, then the limit property $\alpha(t)\to\infty$ implies that $D[\gamma(t)]\to\infty$, which contradicts $E[\gamma(t)]\leq E[\gamma_0]<\infty$.
    A similar argument rules out the shrinking case (with $B[\gamma_0]>0$), where $\alpha(t)\to0$ implies $B[\gamma(t)]\to\infty$.
    In the non-scaling case, again the convergence results imply the assertion.
\end{proof}

For \eqref{eq:CF}, \eqref{eq:FEF}, and \eqref{eq:EF}, non-stationary non-scaling solitons (with infinite energy) are not known.

\appendix

\section{Local well-posedness}\label{sec:local_wellposedness}

In this section, we discuss local well-posedness for a class of arbitrary order geometric flows of complete curves. The analogous results in the case of compact curves are standard. The usual parabolic H\"older space approach is however doomed to fail in the non-compact situation. For the convenience of the reader, we provide some details here.

A similar result for arbitrary order flows on \emph{compact} manifolds has been proven by Mantegazza--Martinazzi \cite{MR3060703}.
For $m\in\N_0$, $T>0$, and a suitable initial datum $\gamma_0\colon \R\to\R^n$, we consider the general initial value problem
\begin{align}\label{eq:STE_IVP}
	\left\lbrace \begin{array}{lll}
		\partial_t\gamma &= (-1)^m\nabla_s^{2m} \kappa + F(\partial_s\gamma, \dots, \partial_s^{2m+1}\gamma) & \text{ in }(0,T)\times \R,  \\
		\gamma(0,\cdot)&=\gamma_0, & \text{ in }\R.
	\end{array}
	\right.
\end{align}
Here $F\colon (\R^n)^{2m+1}\to\R^n$ is smooth.
We normalize \eqref{eq:STE_IVP} by assuming the coefficient of the leading order term to be $(-1)^m$. Of course, all the results in this section remain valid if the leading order term takes the form $(-1)^m\sigma$ for some fixed $\sigma>0$ by transforming time via $t\mapsto \sigma t$. Examples include the flows discussed in this paper, i.e., \eqref{eq:CSF}, \eqref{eq:SDF}, \eqref{eq:CF}, and \eqref{eq:lambda-EF}.

\subsection{Transformation into a parabolic problem}

Equation \eqref{eq:STE_IVP} is not parabolic due to its geometric invariance. %As in De Turck's work \cite{MR697987}, 
In the spirit of the ``De Turck trick'', this issue can be overcome by adding a suitable tangential velocity to the evolution law. % see also \cite{MR4768651,MR4278396,MR3906239} for related works on the elastic flow. 
We first make the following observation which is easily proven by induction on $m\in\N_0$.
\begin{lemma}\label{lem:STE_1}
   For each $m\in\N_0$, there exists a polynomial $\mathcal{P}$ such that
    \begin{align}
        (-1)^m\nabla_s^{2m}\kappa = \Big[(-1)^m\frac{\partial_x^{2m+2} \gamma}{|\partial_x\gamma|^{2m+2}} + \mathcal{P}(|\partial_x\gamma|^{-1},\partial_x\gamma, \dots, \partial_x^{2m+1}\gamma)\Big]^\perp.
    \end{align}
\end{lemma}

We now consider the operator of order $2m+2$ given by
\begin{align}
	\mathcal{A}[\gamma] &\vcentcolon=(-1)^m\frac{\partial_x^{2m+2} \gamma}{|\partial_x\gamma|^{2m+2}} + \mathcal{P}(|\partial_x\gamma|^{-1},\partial_x\gamma, \dots, \partial_x^{2m+1}\gamma),
\end{align}
so that $(-1)^m\nabla_{s}^{2m} \kappa = \mathcal{A}[\gamma]^\perp$ by \Cref{lem:STE_1}.  Writing $F[\gamma]\vcentcolon =  F(\partial_s\gamma, \dots, \partial_s^{2m+1}\gamma),$ we study the so-called \emph{analytic problem}
\begin{align}\label{eq:STE_IVP_analytic}
	\left\lbrace \begin{array}{lll}
		\partial_t\gamma &= \mathcal{A}[\gamma] +  F[\gamma] & \text{ in }(0,T)\times \R,  \\
		\gamma(0,\cdot)&=\gamma_0, & \text{ in }\R,
	\end{array}
	\right.
\end{align}
We will see that sufficiently regular solutions to \eqref{eq:STE_IVP_analytic} can be reparametrized to give solutions of \eqref{eq:STE_IVP}, see the proof of \Cref{thm:GF_well_posed} below.
The leading order term on the right hand side of \eqref{eq:STE_IVP_analytic} is now of the form 
$$(-1)^m|\partial_x\gamma|^{-(2m+2)}\partial_x^{2m+2},$$ 
making \eqref{eq:STE_IVP_analytic} a quasilinear \emph{parabolic} system of order $2m+2$.

We now formulate our short time well-posedness result using the notation of parabolic H\"older spaces $C^{\frac{l+\alpha}{l}, l+\alpha}([0,T]\times \R)$, see \cite{solonnikov1965boundary}. We also denote by $C^\infty([0,T]\times\R;\R^n)$ the space of $u\colon[0,T]\times\R\to\R^n$ such that all derivatives $\partial_t^\alpha\partial_x^\beta u$ are bounded in $[0,T]\times\R$ for any $\alpha,\beta\in\N_0$.

% Our short time well-posedness result can now be formulated as follows, see \cite{solonnikov1965boundary} for the definition of the parabolic H\"older spaces $C^{\frac{k+\alpha}{4}, k+\alpha}([0,T]\times \R)$.
\begin{theorem}\label{thm:STE_analytic}
	Let $\alpha\in (0,1)$, $m\in\N_0$, and $l\vcentcolon = 2m+2$. Let $\gamma_0\in \dot{C}^{l,\alpha}(\R)$ with $\inf_{\R}|\partial_x\gamma_0|>0$. Then there exist $T>0$ and a unique $\gamma\colon [0,T]\times \R\to\R^n$ with
    \begin{align}\label{eq:STE_IVP_analytic_uniform_immersion}
        \inf_{[0,T]\times\R} |\partial_x\gamma| >0
    \end{align}
    and
	\begin{align}\label{eq:STE_IVP_analytic_regularity}
		\gamma - \gamma_0\in C^{\frac{l+\alpha}{l}, l+\alpha}([0,T]\times\R)
	\end{align}
	which solves \eqref{eq:STE_IVP_analytic}. Moreover, there exist $\eta>0$ and $C>0$ such that for any $\tilde{\gamma}_0$ with $\Vert\gamma_0-\tilde{\gamma}_0\Vert_{C^{l,\alpha}(\R)}\leq \eta$, the corresponding solution $\tilde{\gamma}$ of \eqref{eq:STE_IVP_analytic} exists on $[0, T]$ and satisfies
	\begin{align}\label{eq:cont_dependence}
		\Vert\gamma-\tilde{\gamma}\Vert_{C^{\frac{l+\alpha}{l}, l+\alpha}([0,T]\times\R)}\leq C\Vert \gamma_0-\tilde{\gamma}_0\Vert_{C^{l,\alpha}(\R)}.
	\end{align}
	Further, if even $\gamma_0\in \dot{C}^{\infty}(\R)$, then $\gamma-\gamma_0\in C^{\infty}([0,T]\times \R)$.
\end{theorem}

This result can be established with classical tools from the theory of quasilinear parabolic systems.
Since it is not the main focus of this work, we will use \Cref{thm:STE_analytic} without proof, but instead we give a rough sketch of the main idea. It can be proven by applying the contraction principle and parabolic maximal regularity in H\"older spaces, in particular \cite[Theorem 4.10]{solonnikov1965boundary}. %, similarly as in previous works on the elastic flow for curves of finite length, see, for instance, \cite{MR3906239,MR4278396}. 
A subtlety is of course that the curves we consider are \emph{unbounded}, thus naturally not in the parabolic H\"older spaces $C^{\frac{l+\alpha}{l}, l+\alpha}([0,T]\times \R)$. Instead, as in \cite{Chou_Zhu_1998_complete}, we consider the affine space given by \eqref{eq:STE_IVP_analytic_regularity} and use that the quantities involved in the equation \eqref{eq:STE_IVP} do not depend on zeroth order quantities of $\gamma$, but only on its derivatives which are bounded in time and space in the regularity class \eqref{eq:STE_IVP_analytic_regularity}. To obtain the  continuous dependence \eqref{eq:cont_dependence}, one may proceed as  in \cite[Theorem 5.1]{MR3524106}.

\subsection{Local solutions to the geometric problem}

By construction, \Cref{thm:STE_analytic} gives a flow with a particular tangential velocity. We will now construct a solution to \eqref{eq:STE_IVP}. It is worth pointing out that a subtlety of our construction based on an ODE argument is that it does not preserve a finite degree of smoothness, see \cite{MR3695810,Adrian_erratum}. This is why we focus on the smooth category in the sequel.

\begin{theorem}\label{thm:GF_well_posed}
    Let $\gamma_0\in \dot{C}^\infty(\R)$ with $\inf_\R |\partial_x\gamma_0|>0$. Then there exist $T>0$ and a solution $\gamma\colon[0,T]\times \R\to\R^n$ of \eqref{eq:STE_IVP} with 
    \begin{align}\label{eq:GF_well_posed_immersion}
        \inf_{[0,T]\times\R} |\partial_x\gamma|>0
    \end{align}
    and
    \begin{align}\label{eq:GF_wel_posedness_regularity}
        \gamma-\gamma_0\in C^\infty([0,T]\times\R).    
    \end{align}
    The solution is given by a time-dependent reparametrization of the solution constructed in \Cref{thm:STE_analytic}, and it is unique in the class of solutions satisfying \eqref{eq:GF_well_posed_immersion} and \eqref{eq:GF_wel_posedness_regularity}.  Moreover, if $\gamma_0$ is proper, then $\gamma(t,\cdot)$ is proper for all $t\in [0,T]$.
\end{theorem}

\begin{proof}
    Let $\gamma$ be the solution of \eqref{eq:STE_IVP_analytic} constructed in \Cref{thm:STE_analytic}. Then, setting $\hat{\xi} \vcentcolon = |\partial_x\gamma|^{-2} \langle \mathcal{A}[\gamma], \partial_x\gamma\rangle$, we have
    %there exists $\hat\xi\in C^\infty([0,T]\times\R)$ such that
    \begin{align}
        \partial_t \gamma = \mathcal{A}[\gamma]^\perp +F[\gamma]+\hat\xi \partial_x\gamma.
    \end{align}
    From \eqref{eq:STE_IVP_analytic_uniform_immersion} and since $\gamma-\gamma_0\in C^\infty([0,T]\times\R)$, we deduce $\hat{\xi}\in C^\infty([0,T]\times\R)$.
	For all $x\in\R$, we now consider the ODE problem
	\begin{align}\label{eq:ODE_1}
		\left\lbrace \begin{array}{lll}
			\partial_t \Phi_t(x) &=- \hat{\xi}(t, \Phi_t(x)), & t>0, \\
			\Phi_0(x)&=x.&
		\end{array}
		\right.
	\end{align}
	Local existence and uniqueness for fixed $x\in\R$ follows from classical ODE theory. Since $\hat{\xi}$ is uniformly bounded, the solution exists on $[0, {T}]$ for all $x\in\R$. Moreover, since $\hat{\xi}\in C^\infty([0,T]\times\R)$ we have $(t,x)\mapsto \Phi_t(x)-x\in C^\infty([0,T]\times\R)$.
	Differentiating \eqref{eq:ODE_1} with respect to $x$ yields
    \begin{align}\label{eq:ODE_1_x}
		\partial_x \Phi_t(x) = \exp\Big(-\int_0^t\partial_x\hat{\xi}(t',\Phi_{t'}(x))dt'\Big),
	\end{align}
	so that $\Phi_t\colon\R\to\R$ is an orientation-preserving diffeomorphism for all $t\in [0,{T}]$.
%	Together with \eqref{eq:reg_xi}, \eqref{eq:ODE_1_x} implies that $\Phi$ is Lipschitz in $[0,T_0]\times\R$, \textcolor{red}{No!} and that $\partial_t\Phi, \partial_x\Phi \in C^{\frac{\alpha'}{4}, \alpha'}([0,T_0]\times \R)$.
    In abuse of notation, in the following we denote by $X\circ\Phi_t$ the composition of a time-dependent function $X$ with $\Phi_t$ in the spatial variable, i.e., $(X\circ\Phi_t)(t,x) = X(t,\Phi_t(x))$. We set $\tilde \gamma\vcentcolon = \gamma\circ\Phi_t$. Then $\tilde{\gamma}(0)=\gamma_0$, $\tilde{\gamma}-\gamma_0\in C^{\infty}([0,{T}]\times \R;\R^{n})$, and 
	\begin{align}
		\partial_t\tilde{\gamma} &= \partial_t\gamma\circ \Phi_t +(\partial_x\gamma\circ \Phi_t)\partial_t \Phi_t \\
		&=\Big(\mathcal{A}[\gamma]^\perp+ F[\gamma]\Big)\circ\Phi_t +\hat{\xi}\circ \Phi_t (\partial_x\gamma\circ\Phi_t) +(\partial_x\gamma\circ\Phi_t)\partial_t \Phi_t \\
        &= \mathcal{A}[\tilde{\gamma}]^\perp + F[\tilde\gamma]
	\end{align}
	using \eqref{eq:ODE_1} and the transformation of the geometric objects. We conclude that $\tilde{\gamma}$ is a solution to \eqref{eq:STE_IVP} with \eqref{eq:GF_wel_posedness_regularity}. Also, \eqref{eq:STE_IVP_analytic_uniform_immersion} and \eqref{eq:ODE_1_x} yield \eqref{eq:GF_well_posed_immersion}.

    For the uniqueness part, let $\gamma$ be any solution to \eqref{eq:STE_IVP} satisfying \eqref{eq:GF_well_posed_immersion} and \eqref{eq:GF_wel_posedness_regularity}. By the definition of $\mathcal{A}$, if $\Phi\colon\R\to\R$ satisfies $\inf_{\R}\partial_x\Phi>0$, then
    \begin{align}
        & \mathcal{A}[\gamma\circ\Phi] \\
        & \ = (-1)^m \frac{\partial_x^{2m+2}(\gamma\circ\Phi)}{|\partial_x(\gamma\circ\Phi)|^{2m+2}} + \mathcal{P}\Big(|\partial_x(\gamma\circ\Phi)|^{-1}, \partial_x(\gamma\circ\Phi),\dots,\partial_x^{2m+1}(\gamma\circ\Phi)\Big) \\
        & \ = (-1)^m\frac{(\partial_x\gamma\circ\Phi)\partial_x^{2m+2}\Phi}{|(\partial_x \gamma\circ\Phi) \partial_x\Phi|^{2m+2}} + H[\gamma,\Phi].
    \end{align}
    As before, the composition with $\Phi$ is only in the spatial variable. By Fa\`a di Bruno's formula
    \begin{align}
        H[\gamma,\Phi]
         = H\Big(|\partial
        _x\gamma \circ\Phi|^{-1}, |\partial_x\Phi|^{-1}, \partial_x\gamma\circ\Phi, \dots, \partial_x^{2m+2}\gamma\circ\Phi, \partial_x\Phi, \dots,\partial_x^{2m+1}\Phi\Big)
    \end{align}
    is a polynomial in its arguments. We now consider the PDE
    \begin{align}\label{eq:PDE_1}
		\left\lbrace \begin{array}{lll}
			\partial_t \Phi &= \frac{(-1)^m\partial_x^{2m+2}\Phi}{|\partial_x\gamma\circ\Phi|^{2m+2} (\partial_x\Phi)^{2m+2}} + \frac{\langle H[\gamma,\Phi], \partial_x\gamma\circ\Phi\rangle}{|\partial_x\gamma\circ\Phi|^2} & \text{ in }(0,T)\times \R,\\
			\Phi(0,x)&=x &\text{ for }x\in\R.
		\end{array}
		\right.
	\end{align}
    The structure of \eqref{eq:PDE_1} is similar to \eqref{eq:STE_IVP_analytic}, and thus the same techniques as for proving \Cref{thm:STE_analytic} can be used to show that there exists $\hat{T}\leq T$ and a unique solution $\Phi$ of \eqref{eq:PDE_1}  on $[0,\hat{T}]$ with
    \begin{align}
        (t,x)\mapsto \Phi(t,x) - x \in C^\infty([0,\hat{T}]\times \R)
    \end{align}
    satisfying $\inf_{[0,\hat{T}]\times\R} \partial_x\Phi>0$. Writing $\Phi_t(x)\vcentcolon = \Phi(t,x)$, by construction
    \begin{align}
        \partial_t\Phi_t = \frac{\langle \mathcal{A}[\gamma\circ\Phi_t],\partial_x \gamma \circ\Phi_t\rangle}{|\partial_x\gamma\circ\Phi_t|^2},
    \end{align}
    so that defining $\tilde{\gamma}\vcentcolon = \gamma\circ\Phi_t$,
    we compute
    \begin{align}
        \partial_t\tilde{\gamma}= \partial_t\gamma\circ\Phi_t +(\partial_x\gamma\circ\Phi_t) \partial_t \Phi_t &= \Big(\mathcal{A}[\gamma]^{\perp_\gamma}+ F[\gamma]\Big)\circ\Phi_t + (\partial_x\gamma\circ\Phi_t)\partial_t \Phi_t  \\
        & = \mathcal{A}[\tilde\gamma]^{\perp_{\tilde\gamma}} + F[\tilde\gamma] +
        \mathcal{A}[\tilde{\gamma}]^{\top_{\tilde{\gamma}}} = \mathcal{A}[\tilde
        {\gamma}] +F[\tilde\gamma],
    \end{align}
    where $\mathcal{A}[\tilde{\gamma}]^{\top_{\tilde{\gamma}}}$ is the tangential projection of $\mathcal{A}[\tilde{\gamma}]$ along $\tilde\gamma$.
    Thus, $\tilde{\gamma}$ solves \eqref{eq:STE_IVP_analytic} with $\tilde\gamma-\gamma_0\in C^\infty([0,\hat{T}]\times \R)$ and therefore coincides (on $[0,\hat{T}]$) with the unique solution from \Cref{thm:STE_analytic} with initial datum $\gamma_0$. 
    
    Hence, if $\gamma_1, \gamma_2$ are two solutions to \eqref{eq:STE_IVP} on $[0,T]$ with initial datum $\gamma_0$ satisfying \eqref{eq:GF_well_posed_immersion} and \eqref{eq:GF_wel_posedness_regularity}, there exists $0<\hat T\leq T$ and a smooth family of reparametrizations $\Phi_t$ such that $\Phi_0(x)=x$ and $\gamma_1=\gamma_2\circ\Phi_t$. Since both satisfy \eqref{eq:STE_IVP}, this implies
    \begin{align}
        \mathcal{A}[\gamma_1]^{\perp_{\gamma_1}}+ F[\gamma_1] = \partial_t\gamma_1 & = \partial_t\gamma_2 \circ\Phi_t + (\partial_x\gamma_2\circ\Phi_t) \partial_t\Phi_t \\
        &= \big(\mathcal{A}[\gamma_2]^{\perp_{\gamma_2}}+ F[\gamma_2]\big)\circ\Phi_t+ (\partial_x\gamma_2\circ\Phi_t) \partial_t\Phi_t \\
        &=\mathcal{A}[\gamma_1]^{\perp_{\gamma_1}}+ F[\gamma_1] + (\partial_x\gamma_2\circ\Phi_t)\partial_t\Phi_t,
    \end{align}
    where again we used the transformation of the geometric objects in the last step. It follows $\partial_t\Phi_t =0$, and thus $\gamma_1=\gamma_2$ on $[0,\hat{T}]$. Hence
    \begin{align}
        T_0 \vcentcolon = \sup\{ \tau \in [0,T]\mid \gamma_1(t)=\gamma_2(t) \text{ for }t\in [0,\tau]\}>0.
    \end{align}
    If $T_0<T$, one may repeat the same argument for the flow with initial datum $\gamma_1(T_0)=\gamma_2(T_0)$, contradicting maximality of $T_0$. We conclude $\gamma_1=\gamma_2$ on $[0,T]$.

    Lastly, the properness of $\gamma(t,\cdot)$ follows by observing that \eqref{eq:GF_wel_posedness_regularity} and properness of $\gamma_0$ imply that $|\gamma(t,x)| \geq |\gamma_0(x)| - t\Vert \partial_t\gamma\Vert_{L^\infty([0,T]\times\R)} \to \infty$ 
    as $|x|\to\infty$.
\end{proof}

\begin{remark}
    The finite degree of smoothness in \Cref{thm:STE_analytic}  is natural in view of the order of the PDE \eqref{eq:STE_IVP_analytic} and is enough for short time existence of \eqref{eq:STE_IVP} if one allows for a tangential velocity in the flow. The class $\dot{C}^\infty$ in \Cref{thm:GF_well_posed} appears to prevent a loss of regularity in the ODE argument. However, it is not entirely clear if restricting to this class is necessary to have local well-posedness for \eqref{eq:STE_IVP}. 
    
    Moreover, it is worth pointing out that even if an arclength parametrized initial curve has derivatives of all orders, it may not be possible to have a short-time existence result.
    An instructive example for \eqref{eq:CSF} is a locally convex curve with infinitely many loops such that the enclosed area converges to zero at infinity, which would however have unbounded curvature.
    This is in stark contrast to the compact case, where pointwise smoothness automatically yields $\dot{C}^\infty$-regularity (or even $C^\infty$).     
\end{remark}

\section{Technical lemmas}

\subsection{A general blow-up and convergence lemma}

\begin{lemma}\label{lem:blow-up_convergence_inequality}
    Let $T\in(0,\infty]$ and $f\colon [0,T)\to [0,\infty)$ be a continuous function.
    Suppose that there are $C>0$ and $1 \leq p<q<\infty$ such that for all $0\leq t_1\leq t_2<T$,
    \begin{align}\label{eq:gen_blow_up_diff_ineq}
        f(t_2)-f(t_1) \leq C \int_{t_1}^{t_2}(f(t)^p+f(t)^q)dt.
    \end{align}
    
    \begin{enumerate}
        \item\label{item:blow_up} If $T<\infty$ and $\limsup_{t\to T}f(t)=\infty$, then
        \begin{align}\label{eq:blow_up_rate_lemma}
            \liminf_{t\to T} (T-t)^{\frac{1}{q-1}}f(t) > 0.
        \end{align}
        \item\label{item:convergence} If $T=\infty$ and there exist $t_j\nearrow \infty$ with $\limsup_{j\to\infty}(t_{j+1}-t_j)<\infty$ and $f(t_j)\to 0$, then
        %and $f^r\in L^1(0,\infty)$ for some $r>0$, then
        \begin{align}
            \lim_{t\to\infty} f(t) = 0.
        \end{align}
    \end{enumerate}
\end{lemma}

\begin{proof}
 Using the inequality $x^p\leq C(p,q)(x+x^q)$, $x\geq 0$, we may without loss of generality assume that \eqref{eq:gen_blow_up_diff_ineq} is satisfied with $p=1$.

 To prove \eqref{item:blow_up}, fix $t_0\in [0,T)$ and define $F(t)\vcentcolon = f(t_0)+C\int_{t_0}^t (f(\tau)+f(\tau)^q)d\tau$. 
 Then $F$ is locally absolutely continuous on $[0,T)$ and $f(t)\leq F(t)$ by \eqref{eq:gen_blow_up_diff_ineq}. Hence
    \begin{align}
        F'(t)= C(f(t)+f^q(t)) \leq C(F(t)+F(t)^q).
    \end{align}
    For $t$ sufficiently close to $T$, we have $F(t)>0$. Indeed, otherwise there exist $t_j\to T$ with $F(t_j)=0$, implying $f\equiv 0$ on $[t_0,t_j]$ and contradicting $\limsup_{t\to T} f(t)=\infty$. 
    Since
    \begin{align}
        \frac{1}{F(t)+F(t)^q} = \frac{1}{F(t)} - \frac{F^{q-2}(t)}{1+F^{q-1}(t)},
    \end{align}
    choosing $t_j\to T$ with $f(t_j)\to\infty$, we find that 
    \begin{align}
        C(T-t_0)&\geq \lim_{j\to\infty} \int_{t_0}^{t_j} \frac{F'(t)}{F(t) +F(t)^q}dt \\
        &= \lim_{j\to\infty} \int_{t_0}^{t_j} \frac{d}{dt}\Big(\log F(t) - \frac{1}{q-1}\log(1+F(t)^{q-1})\Big)dt \\
        &= \frac{1}{q-1}\log\Big(\frac{1+F(t_0)^{q-1}}{F(t_0)^{q-1}}\Big)=\frac{1}{q-1}\log\Big(\frac{1+f(t_0)^{q-1}}{f(t_0)^{q-1}}\Big)
    \end{align}
    Taking $t_0\to T$, we conclude that $\lim_{t\to T}f(t)=\infty$. For $t$ close to $T$, we thus have $f(t)\geq 1$ and \eqref{eq:gen_blow_up_diff_ineq} implies
	\begin{align}
		f(t_2)-f(t_1)\leq 2C\int_{t_1}^{t_2} f(t)^q dt.
	\end{align}
	As above, we conclude that $G(t)\vcentcolon  = f(t_0)+2C\int_{t_0}^t f(\tau)^q d\tau$ satisfies $f(t)\leq G(t)$ and
	\begin{align}
		G'(t)\leq 2C G(t)^q.
	\end{align}
	Hence, taking $t_0$ close to $T$ and integrating from $t_0$ to $T$, we have
	\begin{align}
		2C(T-t_0)\geq \int_{t_0}^{T} \frac{d}{d t}\left(\frac{1}{1-q} G(t)^{1-q}\right)d t = \frac{1}{q-1}G(t_0)^{1-q} = \frac{1}{q-1} f(t_0)^{1-q}.
	\end{align}
	Renaming $t_0$ into $t$ and rearranging, \eqref{eq:blow_up_rate_lemma} follows.

    We prove part \eqref{item:convergence}.
    By assumption there are $j_0\in\N$ and $r>0$ such that $t_{j+1}-t_j\leq r <\infty$ for all $j\geq j_0$ and such that
    \begin{align}\label{eq:gronwall_conv_1}
        \sup_{j\geq j_0}  \Big(f(t_j)+\frac1j\Big) \leq e^{-2Cr}.
    \end{align}
    Fix any $j\geq j_0$.
    Let $U(t)\vcentcolon = f(t+t_j), V(t)\vcentcolon = (f(t_j)+\frac{1}{j}) e^{2C t}$, and $W\vcentcolon = U-V$. Then $W(0)=-\frac{1}{j}<0$ and, since $f$ is continuous, the number
    \begin{align}
        t_* \vcentcolon = \sup\{ t\in[0,r]\mid W(\tau)\leq 0 \text{ for all }\tau\in [0,t]\}
    \end{align}
    is strictly positive, i.e., $t_*\in(0,r]$. For $0\leq t\leq t_*\leq r$, we have $0\leq V(t)\leq 1$ by \eqref{eq:gronwall_conv_1}. It follows that
    \begin{align}
        V(t) = V(0) + 2 C \int_0^t V(\tau)d\tau \geq V(0) + C \int_0^t (V(\tau)+V^q(\tau))d\tau.
    \end{align}
    On the other hand, by \eqref{eq:gen_blow_up_diff_ineq} we have
    \begin{align}
        U(t)\leq U(0)+ C\int_0^t (U(\tau)+U^q(\tau))d\tau.
    \end{align}
    Combining the above estimates gives $W(t)<0$ for all $t\in [0,t_*]$, so by continuity $t_*=r$. Hence, we conclude $f(t)\leq (f(t_j)+\frac{1}{j}) e^{2 C(t-t_j)}$ for  $t\in [t_j,t_j+r]$,
    which implies
    \begin{align}
       \sup_{t\in [t_j,t_{j+1}]}f(t)\leq         
        \sup_{t\in [t_j,t_j+r]} f(t) \leq  e^{2Cr} \Big(f(t_j)+\frac1j\Big).
    \end{align}
    This yields that $f(t)\to 0$ as $t\to\infty$.
    % For part \eqref{item:convergence}, by the assumption, we have
    % \begin{align}
    %     \limsup_{j\to\infty} \inf_{t\in[j,j+1]} f^r(t) \leq \lim_{j\to\infty} \int_j^{j+1} f^r(t)dt = 0.
    % \end{align}
    % Hence, there exists $t_j\in[j,j+1]$ such that $f(t_j)\to 0$ as $j\to\infty$. We fix $j\geq j_0$ large enough such that 
    % \begin{align}\label{eq:gronwall_conv_1}
    %     \sup_{j\geq j_0}f(t_j)\leq e^{-4 C}
    % \end{align}
    % Let $U(t)\vcentcolon = f(t+t_j), V(t)\vcentcolon = \frac{1}{2} f(t_j) e^{2C t}$, and $W\vcentcolon = U-V$. Then $W(0)<0$ and, since $f$ is continuous, we have that
    % \begin{align}
    %     t_* \vcentcolon = \sup\{ t\in[0,2]\mid W(\tau)\leq 0 \text{ for all }\tau\in [0,t]\}
    % \end{align}
    % is strictly positive, i.e., $t_*\in[0,2]$. For $0\leq t\leq t_*\leq 2$, we have $0\leq V(t)\leq 1$ by \eqref{eq:gronwall_conv_1}. It follows that
    % \begin{align}
    %     V(t) = V(0) + 2 C \int_0^t V(\tau)d\tau \geq V(0) + C \int_0^t (V(\tau)+V^q(\tau))d\tau.
    % \end{align}
    % On the other hand, by \eqref{eq:gen_blow_up_diff_ineq} we have
    % \begin{align}
    %     U(t)\leq U(0)+ C\int_0^t (U(\tau)+U^q(\tau))d\tau.
    % \end{align}
    % Combining the above estimates gives $W(t)<0$ for all $t\in [0,t_*]$, so by continuity $t_*=2$. Hence, we conclude $f(t)\leq \frac{1}{2}f(t_j) e^{2 C(t-t_j)}$ for all $t\in [t_j,t_j+2]$ which implies
    % \begin{align}
    %     \sup_{t\in[j+1,j+2]} f(t) \leq \sup_{t\in [t_j,t_j+2]} f(t) \leq \frac{e^{4C}}{2} f(t_j).
    % \end{align}
    % Sending $j\to\infty$, we find $f(t)\to 0$ as $t\to\infty$.
\end{proof}

\subsection{Rotation number}

We define the rotation number for planar complete curves and prove its preservation along smooth flows.

For a planar curve $\gamma\in\dot{C}^2(\R;\R^2)$ such that $\inf_{\R}|\partial_x\gamma|>0$ and $D[\gamma]<\infty$, since $\partial_s\gamma(x)\to e_1$ as $x\to\pm\infty$ by \Cref{lem:direction_horizontal}, we can define the \emph{rotation number} by
\begin{equation}
    N[\gamma] \vcentcolon= \lim_{R\to\infty}\frac{1}{2\pi}\int_{-R}^R k ds = \lim_{R\to\infty}(\theta(R)-\theta(-R)),
\end{equation}
where $k\colon\R\to\R$ denotes the signed curvature and $\theta\colon\R\to\R$ the tangential angle of $\gamma$.
Note that by the boundary condition at infinity, as in the closed curve case, $N[\gamma]$ is always an integer.

\begin{lemma}\label{lem:rotation_number_preservation}
    Let $\gamma:[0,T]\times\R\to\R^2$ be a family of smooth curves with \linebreak $\inf_{[0,T]\times\R}|\partial_x\gamma|>0$ such that $\gamma(t)\in \dot{C}^2(\R)$, $\lim_{t' \to t}\|\gamma(t')-\gamma(t)\|_{C^1(\R)}=0$, and $\lim_{x\to\pm\infty}\partial_s\gamma(t,x)=e_1$ hold for each $t\in[0,T]$.
    Then for all $t\in [0,T]$ we have $N[\gamma(t,\cdot)]=N[\gamma(0,\cdot)]$.
\end{lemma}

\begin{proof}
    Since the range of $N$ is discrete, it is sufficient to show that at each $t_0\in[0,T]$ the map $t\mapsto N[\gamma(t,\cdot)]$ is continuous.
    Since $\partial_s\gamma(t_0,x)\to e_1$, there is $R_0>0$ such that for all $|x|\geq R_0$ we have $\langle \partial_s\gamma(t_0,x),e_1\rangle \geq \frac{1}{2}$.
    By $C^1$-continuity in time and the assumption on $|\partial_x\gamma|$, there is $\varepsilon_0>0$ such that if $|t-t_0|\leq\varepsilon_0$ and $|x|\geq R_0$, then $\langle \partial_s\gamma(t,x),e_1\rangle \geq \frac{1}{3}$.
    This implies that the curve $\gamma(t,\cdot)$ is graphical outside $[-R_0,R_0]$.
    Now we define $\hat{\gamma}(t,\cdot)$ by cutting off the second component of $\gamma$, namely $\hat{\gamma}\vcentcolon=(\gamma^1,\zeta\gamma^2)$ with $\zeta\in C^\infty_c(\R)$ such that $\chi_{[-R_0,R_0]}\leq \zeta \leq \chi_{[-2R_0,2R_0]}$.
    In view of its topological nature, the rotation number is preserved by this procedure whenever $|t-t_0|\leq\varepsilon_0$, because it only replaces graphical parts with other graphical curves having the same tangent directions at $x=\pm R_0$ and $x\to\pm\infty$.
    Therefore, if $|t-t_0|\leq\varepsilon_0$, then the curves $\gamma(t,\cdot)$ and $\hat{\gamma}(t,\cdot)$, and hence even $\hat{\gamma}(t,\cdot)|_{[-2R_0,2R_0]}$, have the same rotation number $N$.
    Thus the problem is now reduced to showing that $N[\hat{\gamma}(t,\cdot)|_{[-2R_0,2R_0]}]$ is continuous at $t=t_0$, but this is a standard matter since the tangents are always rightward at the endpoints, $\partial_s\hat{\gamma}(t,\pm2R_0)=e_1$, for $t\in[0,T]\cap(t_0-\varepsilon_0,t_0+\varepsilon_0)$.
\end{proof}

\bibliography{Lib}

\end{document}